\documentclass[11pt,reqno]{amsart}

\usepackage{a4wide}
\usepackage{amsmath}
\usepackage{psfrag}
\usepackage{etex,hyperref}
\usepackage[T1]{fontenc}
\usepackage{amsthm}
\usepackage{amssymb,amsxtra}
\usepackage{appendix}
\usepackage[usenames]{color}
\usepackage{stmaryrd}
\usepackage{mathrsfs}
\usepackage{upgreek}
\usepackage{ytableau}
\usepackage{arydshln}
\usepackage{comment,mathbbol}

\usepackage{paralist}
\usepackage{graphicx}
\usepackage[all]{xy}
\usepackage{tikz}
\usetikzlibrary{matrix,fit}
\usepackage{geometry}
\makeindex

 \allowdisplaybreaks

\DeclareSymbolFontAlphabet{\mathbbl}{bbold}

\newcommand{\NN}{\mathbb{N}}
\newcommand{\Z}{\mathbb{Z}}

\DeclareMathOperator{\supp}{supp}

\renewcommand{\b}{\boldsymbol}

\newcommand{\sym}[1]{\mathfrak{S}_{#1}}

\newcommand{\SYT}{\mathrm{SYT}}
\newcommand{\maj}{\mathrm{maj}}
\newcommand{\ch}{\mathrm{ch}}
\newcommand{\syt}{\mathrm{syt}}
\renewcommand{\P}{\Lambda^+}
\newcommand{\pReg}{\Lambda^+_p}
\newcommand{\ccl}[1]{\mathscr{C}_{#1}}
\newcommand{\Ccl}[2]{\mathcal{C}_{#1,#2}}

\newcommand{\B}{\mathcal{B}}
\newcommand{\R}{\mathcal{C}}
\newcommand{\Base}{\mathscr{B}}
\newcommand{\sgn}{\operatorname{\mathrm{sgn}}}
\newcommand{\C}{\Lambda}
\newcommand{\F}{\mathscr{F}}

\newcommand{\N}{\mathrm{N}}

\renewcommand{\t}{\mathfrak{t}}

\newcommand{\GL}{\mathrm{GL}}

\newcommand{\OO}{\mathcal{O}}
\newcommand{\bk}{\hat{\OO}}
\newcommand{\sk}{\bar{\OO}}
\newcommand{\cont}{\text{{\tiny $\#$}}}

\newcommand{\im}{\mathrm{im}}
\newcommand{\wleq}{\leq^{\mathrm{rco}}}
\newcommand{\swleq}{<^{\mathrm{rco}}}
\newcommand{\wgeq}{\geq^{\mathrm{rco}}}
\newcommand{\sref}{\leqslant}
\newcommand{\wref}{\preccurlyeq}
\newcommand{\swref}[2]{#1\prec #2}
\renewcommand{\k}{\ell}
\newcommand{\facmulti}[1]{{{#1}\rotatebox{180}{\raisebox{-1.5ex}{$!$}}}}
\newcommand{\torder}{>}
\newcommand{\Char}{\mathrm{char}}
\newcommand{\Q}{\mathbb{Q}}
\newcommand{\Fp}{\mathbb{F}}
\newcommand{\Ind}{\mathrm{ind}}
\newcommand{\Res}{\mathrm{res}}
\newcommand{\ind}{\mathrm{ind}}
\renewcommand{\inf}{\mathrm{inf}}
\newcommand{\res}{\mathrm{res}}
\newcommand{\Lie}{\mathrm{Lie}}
\newcommand{\sch}{\mathrm{s}}
\newcommand{\rad}{\mathrm{rad}}
\newcommand{\Des}[2]{\mathscr{D}_{#1,#2}}
\newcommand{\des}{\mathrm{Des}}

\newcommand{\proj}{\text{(proj)}}
\newcommand{\comp}[1]{#1^*}
\newcommand{\block}[2]{#1_{#2}}
\newcommand{\scont}{\sharp}
\newcommand{\Symm}{\Upupsilon}

\theoremstyle{definition}

\newtheorem{defn}{Definition}[section]
\newtheorem{rem}[defn]{Remark}
\newtheorem{ques}[defn]{Question}
\newtheorem{eg}[defn]{Example}
\newtheorem{conj}[defn]{Conjecture}

\theoremstyle{plain}

\newtheorem{prop}[defn]{Proposition}
\newtheorem{lem}[defn]{Lemma}
\newtheorem{cor}[defn]{Corollary}
\newtheorem{thm}[defn]{Theorem}
\newtheorem*{Main Result A}{Main Result A}
\newtheorem*{Main Result B}{Main Result B}
\newtheorem*{Main Result C}{Main Result C}

\numberwithin{equation}{section}

\allowdisplaybreaks

\newcommand{\txtblue}[1]{\textcolor{blue}{#1}}

\begin{document}
\title[Descent Algebras of Type A]{Modular Idempotents for the Descent Algebras of Type A and Higher Lie Powers and Modules}
\author{Kay Jin Lim}
\address[K. J. Lim]{Division of Mathematical Sciences, Nanyang Technological University, SPMS-05-16, 21 Nanyang Link, Singapore 637371.}
\email{limkj@ntu.edu.sg}

\begin{abstract} The article focuses on four aspects related to the descent algebras of type $A$. They are modular idempotents, higher Lie powers, higher Lie modules and the right ideals of the symmetric group algebras generated by the Solomon's descent elements. More precisely, we give a construction for the modular idempotents, describe the dimension and character for higher Lie powers and study the structures of the higher Lie modules and the right ideals both in the ordinary and modular cases.
\end{abstract}

\subjclass[2010]{17B01, 05E10, 20C30, 20C20}
\thanks{The author would like to thank Kai Meng Tan and Mark Wildon for fruitful discussions and is also grateful to the referee for useful suggestion and comments. The author is supported by Singapore Ministry of Education AcRF Tier 1 grant RG17/20. }

\maketitle

\section{Introduction}

Let $V$ be a finite dimensional vector space over an infinite field $F$ of characteristic $p$ (either $p=0$ or $p>0$) and $T(V)=\bigoplus_{n\in\NN_0}T^n(V)$ be the tensor algebra which is an associative $F$-algebra with unit where $T^n(V)$ denotes the $n$-fold tensor product of $V$. It can be made into a Lie algebra by means of the Lie bracket \[[v,w]=v\otimes w-w\otimes v\] for all $v,w\in T(V)$. The Lie subalgebra of $T(V)$ generated by $V$ is the free Lie algebra $L(V)$.  The subspace of $L(V)$ containing the homogeneous elements of degree $n$ is $L^n(V)$ and it is called the $n$th Lie power of $V$. Moreover, we have $L(V)=\bigoplus_{n\in\NN_0} L^n(V)$.  When $V$ is a right $FG$-module, $T^n(V)$ is naturally an $FG$-module and $L^n(V)$ is a submodule of $T^n(V)$.  For example, if $G=\GL(V)$ and the dimension of $V$ is at least $n$, applying the Schur functor $f$, we obtain the Lie module $\Lie_F(n)\cong f(L^n(V))$. The study of the Lie powers and Lie modules both in the zero and positive characteristic cases has drawn great attention. The dimension formula of the Lie powers is known by the work of Witt \cite{Witt}. In the classical case when $p=0$ and $G=\GL(V)$, the structure of $L^n(V)$ is well-studied in the earlier work of Thrall \cite{Thrall}, Brandt \cite{Brandt}, Wever \cite{Wever}, Klyachko \cite{Klya} and Kraskiewicz-Weyman \cite{KW}. In the case when $p>0$, the study of these two objects has become significantly more difficult and relatively unknown. Many authors have contributed to this case and we name a few which are by no means complete. When $G=\GL(V)$, the cases $p\nmid n$, $n=p$ and $n=pk$ with $p\nmid k$ have been studied by Donkin-Erdmann \cite{DE}, Bryant-St\"{o}hr \cite{BS05} and Erdmann-Schocker \cite{ES} respectively.  We should also like to mention the decomposition of the Lie powers obtained by Bryant-Schocker \cite{BS06}. The motivation of the study of Lie modules in the modular case also arises from algebraic topology in the work of Selick-Wu \cite{SW00}. For instance, they are interested in the maximal projective submodule of a Lie module.

In the celebrated paper of Solomon \cite{Sol}, he showed that the integral group algebra of a (finite) Coxeter group $G$ has a subalgebra now widely known as the Solomon's descent algebra $\Des{G}{\Z}$.  In that paper, he gave a $\Z$-basis for the descent algebra and studied its structure, especially its radical over the rational field $\Q$, via a ring epimorphism. The results in \cite{Sol} have been extended to the modular case by Atkinson-van Willigenburg \cite{AW} when $G$ is a symmetric group and Atkinson-Pfeiffer-van Willigenburg \cite{APW} when $G$ is an arbitrary Coxeter group. Let $\C(n)$, $\P(n)$ and $\pReg(n)$ be the sets consisting of all compositions, partitions and $p$-regular partitions of $n$ respectively. In the symmetric group $\sym{n}$ case, we have the descent algebra $\Des{n}{F}$ of type A and the $\Z$-basis introduced by Solomon is $\{\Xi^q:q\in\C(n)\}$. The descent algebra $\Des{n}{F}$ has deep connection with the Lie powers and Lie modules as follows. Let $V$ be a right $FG$-module. The symmetric group $\sym{n}$ acts on $T^n(V)$ on the left via the P\'{o}lya action and therefore $T^n(V)$ can be viewed as an ($F\sym{n}$-$FG$)-bimodule. The classical Lie power is $L^n(V)=\omega_n\cdot T^n(V)$, where $\omega_n$ is the Dynkin-Specht-Wever element, and therefore, when $G=\GL(V)$ and $\dim_FV\geq n$, applying the Schur functor$f$, $\Lie_F(n):=\omega_nF\sym{n}\cong f(L^n(V))$. The element $\omega_n$ belongs to $\Des{n}{\Z}$ and is almost an idempotent in the sense that $\omega_n^2=n\omega_n$ and called a Lie idempotent. Blessenohl-Laue \cite{BL} generalized it to obtain a set $\{\omega_q:q\in\C(n)\}$ and this set forms a basis for $\Des{n}{F}$ if $p=0$. The elements in this set are called higher Lie idempotents. In the case of $p=0$, Garsia-Reutenauer \cite{GR} also constructed another basis $\{I_q:q\in\C(n)\}$ for $\Des{n}{F}$ and obtained a complete set of orthogonal primitive idempotents $\{E_\lambda:\lambda\in\P(n)\}$ for $\Des{n}{F}$.  In the case of $p>0$, Erdmann-Schocker \cite{ES} showed that there is a complete set of orthogonal primitive idempotents $\{e_{\lambda,F}:\lambda\in\pReg(n)\}$ for $\Des{n}{F}$ and the image of $e_{\lambda,F}$ under the Solomon's epimorphism in the modular case is the characteristic function on the $p$-equivalent conjugacy class labelled by $\lambda\in \pReg(n)$. However, the modular idempotents of $\Des{n}{F}$ are relatively unknown and it is not clear how one can construct them.

Let $q\in\C(n)$ and $\lambda(q)$ be the partition obtained from $q$ by rearranging its parts. We define $L^q(V):=\omega_q\cdot T^n(V)$ as the higher Lie power so that $L^{(n)}(V)=L^n(V)$. The higher Lie module is defined as $\Lie_F(q):=\omega_qF\sym{n}$ so that, when $V$ is considered as $F\GL(V)$-module and $\dim_FV\geq n$, we have $\Lie_F(q)\cong f(L^q(V))$. In this paper, we focus on four aspects related to $\Des{n}{F}$. They are the modular idempotents, higher Lie powers, higher Lie modules and the right ideals $X^q_R:=\Xi^qR\sym{n}$'s where $q\in\C(n)$ and $R$ is any commutative ring with 1. More precisely, we give a construction for a complete set of orthogonal primitive idempotents of $\Des{n}{F}$ in the modular case, obtain the dimension and character formulae for higher Lie powers $L^q(V)$ when $q$ is coprime to $p$ (that is, $(q,p)=1$ as defined in Subsection \ref{SS: generality}) and $\lambda(q)$ is $p$-regular, study the decomposition problem for projective higher Lie modules, study the Heller translates and periods of non-projective periodic higher Lie modules and investigate the basic properties for $X^q_F$ (its dimension and decomposition in both the ordinary and modular cases). These  objects are interrelated in intricate ways. For example, see Theorem \ref{T: proj summand Xi}.

Whilst our paper focuses on both the higher Lie powers and higher Lie modules, we now give a quick summary of the results on the latter. For simplicity, assume for the moment that $F$ is algebraically closed. Let $q$ be a composition and $\lambda=\lambda(q)$. We have \[\xymatrix{X^q_F\ar@{->>}[r]^-{\text{\ref{L: XiMod surj Lie}}}&\Lie_F(q)\ar@{^(->}[r]_{\lambda\in\pReg(n)}^{\text{\ref{C: higher lie inj in proj}}}\ar[d]^\cong_{{\substack{(q,p)=1\\ \text{\ref{T: higher lie mod}}}}}&e_{\lambda,F}F\sym{n}\\ &\Ind^{\sym{n}}_{\prod^n_{i=1}(C_i\wr \sym{m_i})}F_\delta&}\] Furthermore, if both conditions $(q,p)=1$ and $\lambda\in\pReg(n)$ are satisfied, we have
\[\begin{tikzpicture}
  \node at (0,0) {$e_{\lambda,F}F\sym{n}\overset{\text{\ref{C: Lie iso proj}}}{\cong}\Lie_F(q)\overset{\text{\ref{T: proj summand Xi}}}{\cong} \Xi^qe_{\lambda,F}F\sym{n}$};
  \node at (0,-.6) {$\rotatebox{90}{{$\cong$\ }}$\ {\tiny\ref{C: proj higher Lie decomp}}};
  \node at (-.2,-1.6) {$\displaystyle\bigoplus_{\gamma\in\pReg(n)}m_\gamma P^\gamma$};
\end{tikzpicture}\]  and they are direct summands of $X^q_F$.

In the next section, we collate together the necessary background material and prove some preliminary results along the way. In Section \ref{S: mod idem}, we give a construction for the modular idempotents $\{e_{\lambda,F}:\lambda\in\pReg(n)\}$ as mentioned earlier. The main result is Theorem \ref{T: modular idem} which shows that $e_{\lambda,F}$ has the `leading term' $\frac{1}{\facmulti{\lambda}}\Xi^\lambda$ (see Subsection \ref{SS: generality} for the notation $\facmulti{\lambda}$). In Section \ref{S: higher lie}, we turn our attention to higher Lie modules. Under the assumption that $q$ is coprime to $p$, we construct an explicit basis for $\Lie_F(q)$ and show that it is isomorphic to certain induced module involving $\Lie_F(i)$ where $p\nmid i$ (see Theorems \ref{T: Lie mod basis} and \ref{T: higher lie mod}). Since $\Lie_F(i)$ is well-studied in the case of $p\nmid i$, we obtain a few corollaries; namely, when $q$ is coprime to $p$, $\Lie_F(q)$ is a $p$-permutation module, its ordinary character can be computed using a result Schocker \cite{Schoc03}, its support variety and complexity are known and, when $\lambda\in\pReg(n)$ is coprime to $p$, $\Lie_F(\lambda)$ is isomorphic to $e_{\lambda,F}F\sym{n}$. In Section \ref{S: ch and dim}, we study the structure of $L^q(V)$ when $(q,p)=1$ and $\lambda(q)$ is $p$-regular. In particular, we obtain the dimension formula  for $L^q(V)$ and, when $G=\GL(V)$, its character formula and decomposition in terms of the tilting modules. The multiplicities in the decomposition are described in terms of the irreducible Brauer characters of the symmetric groups. As a direct consequence, we obtain a decomposition of the projective higher Lie module $\Lie_F(q)$. In the next section, Section \ref{S: period}, we study the non-projective periodic higher Lie modules. In this case, they are labelled by $q$ such that $\lambda(q)$ is no longer $p$-regular but there is a unique $k$ such that $\lambda(q)$ has parts of size $k$ with multiplicity at least $p$ but not more than $2p-1$ and the remaining parts still form a $p$-regular partition. More precisely, we compute their Heller translates and periods. Section \ref{S: pivot} is devoted to the combinatorics to set the scene for Section \ref{S: Xi module}. In Section \ref{S: pivot}, we study the permutations in $\sym{n}$ as words and, for each $q\in\C(n)$, define a subset $\Base_q$ of $\sym{n}$. In Section \ref{S: Xi module}, we show that $\{\Xi^qw:w\in\Base_q\}$ is a free basis for $X^q_R$ (see Theorem \ref{T: basis Xiq}) and study the module $X^q_F$ both in the cases when $p=0$ and $p>0$. When $p=0$, in Theorem \ref{T: Xi as higher Lie}, we show that $X^q_F$ is a direct sum of the higher Lie modules $\Lie_F(\lambda)$'s one for each partition $\lambda$ of $n$ such that $\lambda$ is a weak refinement of $q$. When $p>0$ and $q$ is coprime to $p$, if $\lambda=\lambda(q)$, in Theorem \ref{T: proj summand Xi}, we show that $X^q_F$ has a projective summand $\Xi^qe_{\lambda,F}F\sym{n}$ which is also isomorphic to both $\Lie_F(q)$ and $e_{\lambda,F}F\sym{n}$. 

\section{Preliminary}\label{S: prelim}

Throughout this article, let $\NN_0$ and $\NN$ be the sets of non-negative and positive integers respectively. Furthermore, $R$ is a commutative ring with 1 \index{$R$}and $F$ \index{$F$}is an infinite field of characteristic $p$ (either $p=0$ or $p>0$). We will consider both the ordinary $p=0$ and modular $p>0$ cases. The reader will be reminded when the condition on $p$ is imposed. 
All the infinite (direct) sums and infinite (direct) products encountered in this paper are indeed finite (direct) sums and finite (direct) products respectively. For example, $\sum^\infty_{i=1}\delta_i$ means there exists $r\in\NN_0$ such that $\delta_i=0$ for all $i>r$ and hence $\sum^\infty_{i=1}\delta_i=\sum^r_{i=1}\delta_i$.

\subsection{Generalities}\label{SS: generality} For any integers $a\leq b$, the set $\{a,a+1,\ldots,b\}$ of consecutive integers is denoted as $[a,b]$\index{$[a,b]$}. For a totally order set $S$, we write $S_1\sqcup \cdots\sqcup S_k=S$ \index{$\sqcup$} if $S$ is a disjoint union of the ordered subsets $S_1,\ldots,S_k$. In this case, we denote $S_i\sqsubseteq S$ \index{$\sqsubseteq$}for the ordered subset $S_i$ in $S$. Sometimes, we view $[a,b]$ as an ordered set, in the obvious order, and consider its ordered subsets.

A composition $q$ of $n$ is a finite sequence $(q_1,\ldots,q_k)$ in $\NN$ such that $\sum_{i=1}^kq_i=n$ (or, by our convention, $\sum^\infty_{i=1}q_i=n$ where $q_i=0$ for all $i>k$). In this case, we write $q\vDash n$\index{$\vDash$} and $\ell(q)=k$\index{$\ell(q)$}. We say that $q$ is coprime to $p$ if $p\nmid q_i$ for all $i\in [1,k]$ or simply write $(q,p)=1$\index{$(q,p)$}. For instance, any composition is coprime to $0$ and, as such, the condition $(q,p)=1$ is superfluous when $p=0$. Furthermore, for each $j\in [1,k]$, we write \[q^+_j=\sum^{j}_{i=1}q_i,\] i.e., \index{$q_j^+$}the sum of the first $j$ parts of $q$. By convention, $q^+_0=0$. The composition $q$ is a partition if $q_1\geq q_2\geq \cdots\geq q_k$ and we write $q\vdash n$. The conjugate partition of a partition $\lambda$ is denoted as $\lambda'$\index{$\lambda'$}. We denote the set of all compositions and partitions of $n$ by $\C(n)$ and $\P(n)$ respectively\index{$\C(n)$}\index{$\P(n)$}. The concatenation of two compositions $r,q$ (not necessarily of the same integer) is denoted as $r\cont q$, i.e., \[r\cont q=(r_1,\ldots,r_{\ell(r)},q_1,\ldots,q_{\ell(q)}).\index{$\cont$}\]

Let $q,r\in\C(n)$. If the parts of $q$ can be rearranged to $r$ then we write $q\approx r$\index{$\approx$}. Clearly, this is an equivalence relation and the equivalence classes are represented by $\P(n)$. As such, we write $\lambda(q)$ for the partition such that $q\approx \lambda(q)$\index{$\lambda(q)$}. The composition $r$ is a (strong) refinement of $q$ if there are integers $0=i_0<i_1<\cdots<i_k$ where $k=\ell(q)$ such that, for each $j\in [1,k]$, \[r^{(j)}:=(r_{i_{j-1}+1},r_{i_{j-1}+2},\ldots,r_{i_j})\vDash q_j,\] and we denote it as $r\sref q$\index{$r\sref q$}. In this case, we define \index{$\ell"!(r,q)$}\begin{align*}
\ell(r,q)&=\prod^k_{j=1}\ell(r^{(j)}),\ \  \ell!(r,q)=\prod^k_{j=1}\ell(r^{(j)})!,\ \  F_{q}(r)=\prod_{j=1}^kr_{i_j}.
\end{align*}\index{$F_q(r)$} \index{$\ell(r,q)$} If $r\sref q$ but $r\neq q$, we write $r<q$\index{$r<q$}.  On the other hand, the composition $r$ is a weak refinement of $q$ if there is a rearrangement of $r$ which is a refinement of $q$, i.e., $r\approx s\sref q$ for some $s$ and we denote this as \index{$r\wref q$}$r\wref q$. In the case $r\wref q$ but $r\not\approx q$, we write \index{$r\prec q$}$\swref{r}{q}$.

Let $q\vDash n$. For each $i\in \NN$, we denote $m_i(q)=|\{j:q_j=i\}|$, i.e., $m_i(q)$ is the number of parts of $q$ of size $i$. Notice that $m_i(q)=0$ for all $i>n$\index{$m_i(q)$}. Furthermore, we let \[m(q)=(\ldots,m_2(q),m_1(q)).\index{$m(q)$}\] In the case when $\lambda\vdash n$, we often write $\lambda=(\ldots,2^{m_2(\lambda)},1^{m_1(\lambda)})$. The partition $\lambda$ is called $p$-regular if either $p=0$ or, $p>0$ and $m_i(\lambda)<p$ for all $i\in\NN$. We denote the set of all $p$-regular partitions by $\pReg(n)$\index{$\pReg(n)$} (so $\P_0(n)=\P(n)$). Furthermore, let $\facmulti{q}=\prod_{i=1}^\infty m_i(q)!$ \index{$q\rotatebox{180}{\raisebox{-1.5ex}{"!}}$}and  \[q?=\prod_{i=1}^\infty i^{m_i(q)}m_i(q)!=\facmulti{q}\prod_{j=1}^{\ell(q)}q_j.\index{$q?$}\] For example, $\facmulti{(n)}=1$ and $\facmulti{(1^n)}=n!$. Notice that $\facmulti{\lambda}\neq 0$ in the field $F$ if and only if $\lambda\in\pReg(n)$. Also, if $q\approx r$, then we have both $\facmulti{q}=\facmulti{r}$ and $q?=r?$.

For a matrix $A$, we denote the $i$th row and $j$th column of $A$ by $r_i(A)$ and $c_j(A)$ respectively, i.e.,
\begin{align*}
  r_i(A)&=(A_{i1},\ldots,A_{in}),\index{$r_i(A)$}\\
  c_j(A)&=(A_{1j},\ldots,A_{mj})\index{$c_j(A)$},
\end{align*} if $A$ is an $(m\times n)$-matrix. Let $q,r,s\in \C(n)$. We shall now define the sets $N^s_{r,q}$ and $\overline{N^s_{r,q}}$ whose cardinalities play an important role in the descent algebras of type A. Let $N_{r,q}^s$ \index{$N_{r,q}^s$}be the set consisting of all the $(\ell(r)\times \ell(q))$-matrices $A$ with entries in $\NN_0$ such that
\begin{enumerate}[(a)]
\item for each $i\in [1,\ell(r)]$, $\comp{r_i(A)}$ is a composition of $r_i$,
\item for each $j\in [1,\ell(q)]$, $\comp{c_j(A)}$ is a composition of $q_j$, and
\item $s=\comp{(r_1(A)\cont\cdots\cont r_{\ell(q)}(A))}$. 
\end{enumerate} Here (and in the proof of Lemma \ref{L: basic coef approx}), $\comp{\delta}$ \index{$\comp{\delta}$}denotes the composition obtained from a sequence $\delta$ in $\NN_0$ by deleting all the zero entries. Furthermore, we let $\overline{N_{r,q}^s}$ \index{$\overline{N_{r,q}^s}$}be the subset consisting of $A\in N_{r,q}^s$ such that each column of $A$ contains exactly one nonzero entry.

\begin{eg} Let $q=(2,1)$, $r=(1,2)$ and $s=(1,1,1)$. We have
\begin{align*}
  N^{q}_{q,r}&=\left \{\begin{pmatrix}
    0&2\\ 1&0
  \end{pmatrix}\right \},& N^{r}_{r,q}&=\left \{\begin{pmatrix}
    0&1\\ 2&0
  \end{pmatrix}\right \},& N^{s}_{q,r}&=\left \{\begin{pmatrix}
    1&1\\ 0&1
  \end{pmatrix}\right \},& N^{s}_{r,q}&=\left \{\begin{pmatrix}
    1&0\\ 1&1
  \end{pmatrix}\right \}.
\end{align*} Therefore $\overline{N^{q}_{q,r}}=N^{q}_{q,r}$, $\overline{N^{r}_{r,q}}=N^{r}_{r,q}$ but $\overline{N^{s}_{q,r}}=\emptyset=\overline{N^{s}_{r,q}}$.
\end{eg}

\begin{lem}\label{L: basic coef approx} Let $q,r\in\C(n)$ such that $q\approx r$. We have $|\overline{N^r_{r,q}}|=\facmulti{r}=\facmulti{q}$.
\end{lem}
\begin{proof} Since $q\approx r$, the set $\overline{N^r_{r,q}}$ consists of `permutation matrices' where the nonzero entries are precisely the  components of $q$. Notice that, for each $i\in [1,n]$, any permutation of the $j$th rows of $A\in \overline{N^r_{r,q}}$ such that $\comp{r_j(A)}=(i)$ yields another element in $\overline{N^r_{r,q}}$.
\end{proof}

Let $\sym{A}$ be the symmetric group acting on a finite subset $A$ of $\NN$. By convention, $\sym{\emptyset}$ is the trivial group. For $n\in\NN$, we write $\sym{n}$ for $\sym{[1,n]}$. In this article, the composition of the permutations is read from left to right. For example, $(1,2)(2,3)=(1,3,2)$. The conjugacy class of $\sym{n}$ labelled by $\lambda\in\P(n)$ is denoted by $\ccl{\lambda}$\index{$\ccl{\lambda}$}. Then we have \[|\ccl{\lambda}|=\frac{n!}{\lambda?}.\] Two partitions $\lambda,\mu\in\P(n)$ are $p$-equivalent if the $p'$-parts of any $\sigma\in\ccl{\lambda}$ and $\tau\in\ccl{\mu}$ are conjugate in $\sym{n}$. In this case, we write $\lambda\sim_p\mu$ \index{$\sim_p$}for the equivalence relation. Notice that $\lambda\sim_0\mu$ if and only if $\lambda=\mu$. The $p$-equivalent classes are represented by the $p$-regular partitions in $\pReg(n)$. For $\lambda\in\pReg(n)$, we write $\ccl{\lambda,p}$ \index{$\ccl{\lambda,p}$}for the union of the conjugacy classes $\ccl{\mu}$ such that $\mu\sim_p\lambda$. We record an easy lemma which we will need later.

\begin{lem}\label{L: p-equiv conj} Suppose that $\lambda\in\pReg(n)$ and $(\lambda,p)=1$. We have $\ccl{\lambda,p}=\ccl{\lambda}$.
\end{lem}
\begin{proof} We only need to check $\ccl{\lambda,p}\subseteq \ccl{\lambda}$. Let $\sigma\in\ccl{\mu}$ such that $\mu\sim_p\lambda$. By definition, given that $p\nmid \lambda_i$ for all $i\in [1,\ell(\lambda)]$, the $p'$-part of $\sigma$ is conjugate to $\lambda$. As noted in \cite[\S2]{ES}, the cycle type of the $p'$-part of an element in $\ccl{\mu}$ is obtained by replacing each entry $\mu_i=kp^m$ where $p\nmid k$ of $\mu$ by $(k,\ldots,k)=(k^{p^m})\vdash kp^m$. Since $\lambda$ is $p$-regular, the cycle type of the $p'$-part of $\sigma$ is $\mu$. Therefore, $\mu=\lambda$.
\end{proof}

For any $r\in\NN_0$ and $\sigma\in\sym{n}$, we write $\sigma^{+r}$ \index{$\sigma^{+r}$}for the permutation such that $(i+r)\sigma^{+r}=(i)\sigma+r$ if $i\in[1,n]$ and fixes the remaining numbers. For $q\vDash n$ (or more generally, a sequence $q$ in $\NN_0$ such that $n=\sum^\infty_{i=1}q_i$), the Young subgroup of $\sym{n}$ with respect to $q$ is \[\sym{q}=\prod^\infty_{j=1}\sym{[1+q^+_{j-1},q^+_j]}.\index{$\sym{q}$}\] We refer the reader to \cite[\S2.4]{LT16} for the precise definitions of the following subgroups. For any subgroup $H$ of $\sym{n}$ and $m\in\NN$, let $\Delta_m(H)$ \index{$\Delta_m(H)$}be the diagonal subgroup of $\sym{mn}$ isomorphic to $H$. On the other hand, we denote $H^{[m]}$ \index{$H^{[m]}$}the subgroup of $\sym{mn}$ consisting of permutations acting on the $n$ successive blocks of size $m$ according to the elements in $H$.

In this article, we will also use another presentation for permutations, the word or one-line notation, which we shall now describe. Let $A$ be a set consisting of distinct elements called the alphabets. A word in $A$ is $w=w_1w_2\ldots w_n$ where $w_1,w_2,\ldots,w_n\in A$. In this case, the length of the word $w$ is $n$ and we write $|w|=n$\index{$\mid \hspace{-2pt}w\hspace{-2pt}\mid$}. A subword of $w$ is of the form $w_{i_1}\ldots w_{i_k}$ for some $1\leq i_1<\cdots<i_k\leq n$ which is not necessarily contiguous. The group $\sym{n}$ acts on the words of length $n$ via the P\'{o}lya action; namely, \[\tau\cdot w_1\ldots w_n=w_{(1)\tau}\ldots w_{(n)\tau}.\index{$\tau\cdot w_1\ldots w_n$}\] In the case $A=\NN$, we can identify a word $w=w_1\ldots w_n$ such that $\{w_1,\ldots,w_n\}=[1,n]$ with the permutation $\sigma$ whereby $(i)\sigma=w_i$ for all $i\in [1,n]$. As such, notice that \[\tau\sigma=\tau w_1\ldots w_n=w_{(1)\tau}\cdots w_{(n)\tau}=\tau\cdot w\] for another element $\tau\in\sym{n}$. In other words, the multiplication in $\sym{n}$ coincides with the P\'{o}lya action. With this identification, $\sym{n}$ can be totally ordered by the reverse colexicographic order $\wleq$; \index{$\wleq$} namely $w\wleq v$ whenever either $w=v$ or, if $j$ is the least positive integer such that $w_{j+1}=v_{j+1},\ldots,w_n=v_n$, we have $w_j>v_j$. In the case $w\wleq v$ and $w\neq v$, we write $w\swleq v$. For example, in $\sym{3}$, omitting the superscript $\mathrm{rco}$, we have \[123<213<132<312<231<321.\]

\subsection{Modules for algebras} Let $G$ be a finite group. We denote the group algebra over $R$ by $RG$. By an $RG$-module, we refer to right $RG$-module unless otherwise stated. The trivial $RG$-module is denoted as $R$. 
 For any $RG$-module $M$ and non-negative integer $n$, we write $M^{\wr n}$ \index{$M^{\wr n}$}for the $R[G\wr\sym{n}]$-module which is the $R$-module $M^{\otimes n}$ such that $G\wr \sym{n}$ acts via \[(m_1\otimes\cdots\otimes m_n)(\tau;g_1,\ldots,g_n)=m_{(1)\tau^{-1}}g_1\otimes\cdots\otimes m_{(n)\tau^{-1}}g_n\] where $m_1,\ldots,m_n\in M$, $\tau\in \sym{n}$ and $g_1,\ldots,g_n\in G$. By convention, $M^{\wr 0}=R$. For another $RG'$-module $M'$, we write $M\boxtimes M'$ \index{$M\boxtimes M'$}for the outer tensor product of $M$ and $M'$ which is an $R[G\times G']$-module. The induction and restriction functors with respect to a subgroup $H$ of a group $G$ are denoted as $\ind^G_H$ and $\res^G_H$ \index{$\ind^G_H$}\index{$\res^G_H$}respectively. Suppose that $H\lhd G$ and let $N$ be an $R[G/H]$-module. We denote the inflation of $N$ from $G/H$ to $G$ (via the canonical surjection $G\twoheadrightarrow G/H$) by $\inf^G_{G/H} N$.

The group algebra $FG$ can be identified with $\Z G\otimes_\Z F$. For convenience, for any $x\in \Z G$, we also write $x\in FG$ for its `reduction modulo $p$' under this identification. We have the following basic lemma.

\begin{lem}\label{L: dim mod p} Let $G$ be a finite group and $x\in \Z G$. Then, taking reduced modulo $p$, we have \[\dim_F xFG\leq \dim_{\Q}x\Q G.\]
\end{lem}
\begin{proof} The $\Z$-module $x\Z G$ is torsion free and therefore it has a $\Z$-basis $B$. The dimension $d:=\dim_\Q x\Q G$ is obviously the cardinality of $B$. Upon taking reduction modulo $p$, the set $B$ clearly spans $xFG$ but could be linearly dependent. Therefore $\dim_F xFG$ is not larger than $d$.
\end{proof}

For the remainder of this subsection, assume that $p>0$.
\smallskip

Let $M$ be an indecomposable $FG$-module. A vertex of $M$ is a minimal subgroup $H$ such that $M\mid \Ind^G_HN$ for some $FH$-module $N$. In this case, if $N$ is indecomposable, it is called an $FH$-source of $M$. By a result of Green \cite{Green}, a vertex of $M$ must be a $p$-subgroup, say $P$, all vertices of $M$ are $G$-conjugate and all $FP$-sources of $M$ are $\N_G(P)$-conjugate.

Let $M$ be an $FG$-module.  The module $M$ has trivial source if every indecomposable summand of $M$ has trivial module as its source.  The module $M$ is projective if it is a direct summand of a direct sum of regular modules, or equivalently, every indecomposable summand of the module has the trivial subgroup $\{1\}$ as its vertex.  The module $M$ is a $p$-permutation module if, for every Sylow $p$-subgroup $P$ of $G$, there exists a $F$-basis of $M$ that is permuted by $P$.

The following theorem gives a characterization of $p$-permutation modules.

\begin{thm}[{\cite[(0.4)]{Brou}}]\label{T: Broue} An indecomposable $FG$-module $M$ is a $p$-permutation module if and only if there exists a $p$-subgroup $P$ of $G$ such that $M\mid \Ind^G_PF$; equivalently, $M$ has trivial source.
\end{thm}

Let $M$ be an $F G$-module. Let $P(M)$ be the projective cover of $M$. The kernel of the epimorphism from $P(M)$ to $M$ is denoted as $\Omega(M)$. Similarly, let $I(M)$ be the injective hull of $M$. The cokernel of the monomorphism from $M$ into $I(M)$ is denoted as $\Omega^{-1}(M)$. It is well-known that $\Omega(M)$ and $\Omega^{-1}(M)$ are well-defined up to isomorphism. For $m\geq 2$, we denote $\Omega^m(M)=\Omega(\Omega^{m-1}(M))$ and $\Omega^{-m}(M)=\Omega^{-1}(\Omega^{-(m-1)}(M))$. Furthermore, we write $\Omega^0(M)$ for the non-projective part of $M$. These are called the Heller translates of $M$\index{$\Omega^m(M)$}. If there exists a positive integer $m$ such that $\Omega^m(M)\cong\Omega^0(M)$, we say that $M$ is periodic and its period is the least such positive integer.

Consider a $p$-modular system $(\bk,\OO,\sk)$. That is, $\OO$ is a complete discrete valuation ring with the maximal ideal $(p)$, with quotient field $\bk$ of characteristic 0 and with residue field $\sk$. It is well-known that every trivial source $\sk G$-module $M$ lifts uniquely to an $\OO G$-module $M_\OO$ such that $M_\OO\otimes \sk\cong M$ (see, for example, \cite[Corollary 2.6.3]{Ben84}). In this case, the ordinary character $\ch(M)$ of $M$ is defined to be the character of $M_\OO\otimes \bk$. In particular, according to Theorem \ref{T: Broue}, $p$-permutation modules have ordinary characters.

\medskip

We now restrict ourselves to modules for the Schur algebras and symmetric groups. Let $q\in\C(n)$ and $M^q_R$ \index{$M^q_R$}be the permutation module induced from the trivial $R\sym{q}$-module to $\sym{n}$, that is \[M^q_R=\ind^{\sym{n}}_{\sym{q}}R.\] If $q\approx q'$ then $M^q_R\cong M^{q'}_R$. For each $\lambda\in\P(n)$, following James \cite{James}, we have the Specht module $S^\lambda_R$ \index{$S^\lambda_R$}which is a submodule of $M^\lambda_R$. For simplicity, we shall drop the suffix $R$ if there is no ambiguity. 

Over $F$, the indecomposable summands of the permutation modules are the Young modules $Y^\lambda$'s, \index{$Y^\lambda$}one for each $\lambda\in\P(n)$ (see \cite{GJ84}). The Young module $Y^\lambda$ is a summand of $M^\lambda$ with multiplicity one and the remaining summands of $M^\lambda$ are of the form $Y^\mu$ where $\mu$ dominates $\lambda$ in the usual dominance order. By the Submodule Theorem (see \cite{James}), $Y^\lambda$ is, up to isomorphism, the unique summand of $M^\lambda$ containing $S^\lambda$ as a submodule. Let $P^\lambda=Y^{\lambda'}\otimes\sgn$ \index{$P^\lambda$}where $\sgn$ is the signature representation for $\sym{n}$.  When $\lambda\in\pReg(n)$, $S^\lambda$ has a simple head $D^\lambda$\index{$D^\lambda$}. Moreover, the set $\{D^\lambda:\lambda\in\P_p(n)\}$ is a complete set of non-isomorphic simple $F\sym{n}$-modules. It turns out that $P^\lambda$ is isomorphic to the projective cover of $D^\lambda$ and therefore the set $\{P^\lambda:\lambda\in\P_p(n)\}$ is a complete set of non-isomorphic projective indecomposable $F\sym{n}$-modules. Let $\beta^\lambda$ \index{$\beta^\lambda$} denote the Brauer character of $D^\lambda$ and let $\zeta^\lambda_F$ \index{$\zeta^\lambda_F$}(or simply, $\zeta^\lambda$) be the ordinary character of $S^\lambda_F$.

We refer the reader to \cite{DE,Green830} for the representation theory of Schur algebras we shall need here. Let $m,n\in\NN$, $\P(m,n)$ \index{$\P(m,n)$}be the set of all partitions $\lambda$ of $n$ such that $\ell(\lambda)\leq m$, \[\P_p(m,n)=\P(m,n)\cap \P_p(n) \index{$\P_p(m,n)$}\]and $S(m,n)$ \index{$S(m,n)$}be the Schur algebra over $ F$. The category of the polynomial representations (right modules) of $\GL_m( F)$ of degree $n$ is equivalent to the category of (right) $S(m,n)$-modules. The simple modules of $S(m,n)$ are parameterised by $\P(m,n)$ such that, for each $\lambda\in\P(m,n)$, $L(\lambda)$ \index{$L(\lambda)$}has highest weight $\lambda$. For each $\lambda\in\P(m,n)$, there is a distinguished $S(m,n)$-module, the tilting module, $T(\lambda)$ \index{$T(\lambda)$}with highest weight $\lambda$ and contains $L(\lambda)$ as a composition factor. Moreover, if $\lambda$ is $p$-regular, $T(\lambda)$ is both projective and injective with simple head and socle both isomorphic to $L(\lambda)$.

Suppose further that $m\geq n$. We have the Schur functor $f$ \index{$f$}from the category of $S(m,n)$-modules to the category of $F\sym{n}$-modules. The functor $f$ maps $T(\lambda)$ to $P^\lambda$ and $L(\lambda')$ to $D^\lambda\otimes\sgn$ if $\lambda\in\P_p(n)$ and $0$ otherwise. 

\subsection{The descent algebras}\label{SS: descent} Consider the symmetric group algebra $R\sym{n}$. Notice that the operation $\sigma^{+r}$ for $\sigma\in\sym{n}$ and $r\in\NN_0$, we introduced earlier, can be extended linearly to $R\sym{n}$. A permutation $\sigma$ is said to have a descent at $i\in [1,n-1]$ if $(i)\sigma>(i+1)\sigma$. We write \[\des(\sigma)=\{i\in[1,n-1]:(i)\sigma>(i+1)\sigma\}.\index{$\des(\sigma)$}\] For each $q\in\C(n)$, we define the Solomon's descent element \[\Xi^q=\sum_{\des(\sigma)\subseteq \{q^+_i:i\in [1,\ell(q)]\}}\sigma\in R\sym{n}.\index{$\Xi^q$}\] The element $\Xi^q$ is easy to be written as a sum of words associated to the row standard $q$-tableaux. For example, $\Xi^{(n)}=1$, $\Xi^{(1^n)}=\sum_{\sigma\in\sym{n}}\sigma$ and \[\Xi^{(2,2)}=1234+1324+1423+2314+2413+3412.\]

Let $r\in\C(n)$ and $\tau\in \sym{k}$ where $k=\ell(r)$ and $q=(r_{(1)\tau},\ldots,r_{(k)\tau})$. We write, in two-line form,
\begin{equation}\label{Eq: tau_r}
\block{\tau}{r}=\begin{pmatrix}1&\cdots&q_1&\cdots&q_{k-1}^++1&\cdots&q_k^+\\
r^+_{(1)\tau-1}+1&\cdots&r^+_{(1)\tau-1}+q_1&\cdots&r^+_{(k)\tau-1}+1&\cdots&r^+_{(k)\tau-1}+q_k
\end{pmatrix}\index{$\block{\tau}{r}$}.
\end{equation} Notice that $\Xi^q=\block{\tau}{r}\Xi^r$. For example, if $r=(2,1,1)$, $q=(1,2,1)$, we could take $\tau=213$ (swapping the first and second components) and \[\block{\tau}{r}=\begin{pmatrix}
  1&2&3&4\\ 3&1&2&4
\end{pmatrix}=3124.\] Furthermore, if $s\sref q$, let $\sym{s}\backslash\sym{q}$ denote the set consisting of minimal length right coset representatives of $\sym{s}$ in $\sym{q}$, then we have $\Xi^s=\kappa\Xi^q$ where $\kappa=\sum_{\pi\in\sym{s}\backslash\sym{q}}\pi$. Therefore, we get the following lemma.

\begin{lem}\label{L: basic Xi} Let $r\wref q$ be compositions of $n$. Then there is a surjection $\Xi^qR\sym{n}\twoheadrightarrow \Xi^rR\sym{n}$ given by the left multiplication of $\tau_s\kappa$ where $r\approx s\sref q$ and $\kappa=\sum_{\pi\in\sym{s}\backslash \sym{q}}\pi$.
\end{lem}

Let $\Xi(n)=\{\Xi^q:q\in\C(n)\}\subseteq R\sym{n}$\index{$\Xi(n)$} where we have deliberately suppressed the involvement of $R$ in $\Xi(n)$. Let $\Des{n}{R}$ \index{$\Des{n}{R}$}be the $R$-linear span of $\Xi(n)$. In \cite{Sol}, Solomon proved that $\Des{n}{\Z}$ (more generally, for a Coxeter group) is an $\Z$-algebra and $\Xi(n)$ forms a $\Q$-basis for $\Des{n}{\Q}$. For the modular case, we refer the readers to \cite{AW,APW}. The algebra $\Des{n}{R}$ is now commonly known as the Solomon's descent algebra (of type A). In particular, the product $\Xi^q\Xi^r\in\Des{n}{R}$ can be written explicitly as an $R$-linear combination of the elements in the set $\Xi(n)$. There is an explicit combinatorial description for the structure constants as below.


\begin{thm}[{\cite[Proposition 1.1]{GR}}]\label{T: GR 1.1} For $q,r\in\C(n)$, we have \[\Xi^r\Xi^q=\sum_{s\in\C(n)}|N^s_{r,q}|\Xi^s.\] Furthermore, the set $\Xi(n)$ is an $R$-basis for the ring $\Des{n}{R}$.
\end{thm}

For $q\in\C(n)$, let $Y_{q}$, $Y_{q}^{\circ}$ \index{$Y_{q}$}\index{$Y_{q}^{\circ}$}be the subsets of $\Des{n}{R}$ which are $R$-spanned by
\begin{align*}
  B_{q}:&=\{\Xi^\xi:\C(n)\ni \xi\wref q\},\index{$B_q$}\\
  B_{q}^{\circ}:&=\{\Xi^\xi:\C(n)\ni\swref{\xi}{q}\},\index{$B_{q}^{\circ}$}
\end{align*} respectively. As a consequence of the theorem, we have the following corollary.

\begin{cor}\label{C: product in Descent}  Let $q,r\in\C(n)$.
\begin{enumerate}[(i)]
  \item The product $\Xi^r\Xi^q$ is a linear combination of some $\Xi^s$ such that $s$ is both a strong and a weak refinements of $r$ and $q$ respectively, i.e., $r\geqslant s\wref q$. In particular, both $Y_{q}$ and $Y_{q}^{\circ}$ are two-sided ideals of $\Des{n}{R}$.
  \item The coefficient of $\Xi^r$ in $(\Xi^r)^2$ is $\facmulti{r}$.
\end{enumerate}
\end{cor}
\begin{proof} For part (i), any composition $s$ with $|N^s_{rq}|\neq 0$ is obtained by reading the entries of $A\in N^s_{rq}$ as in the beginning of Subsection \ref{SS: generality}. Therefore $s=s^{(1)}\cont \cdots\cont s^{\ell(r)}$ where each $s^{(i)}$ is a composition of $r_i$, i.e., $s\leqslant r$. Considering the column sums of $A$ which gives $q$, we get $s\wref q$. For part (ii), notice that each row and column of a matrix in $N^r_{r,r}$ has exactly one nonzero entry. Now use Lemma \ref{L: basic coef approx}.
\end{proof}

The Dynkin-Specht-Wever element $\omega_n$ \index{$\omega_n$}for the group algebra $R\sym{n}$ is defined as \[\omega_n=(1-c_n)(1-c_{n-1})\cdots (1-c_2)\in R\sym{n}\] where, for each $i\in [2,n]$, $c_i=(i,i-1,\ldots,1)$ is the descending cycle of length $i$. Notice that \[\omega_n=\sum(-1)^{|\mathbf{j}|}s_{\mathbf{j}}\] where the sum is taken over all subsets $\mathbf{j}=\{j_1<\cdots<j_t\}\sqsubseteq [2,n]$ where $\{1=k_1<\cdots<k_{n-t}\}\sqcup \mathbf{j}=[1,n]$ and \[s_\mathbf{j}=j_t\cdots j_1k_1\ldots k_{n-t}=\begin{pmatrix}
  1&\cdots&t&t+1&\cdots&n\\ j_t&\cdots&j_1&k_1&\cdots&k_{n-t}
\end{pmatrix}=c_{j_t}\cdots c_{j_1}.\index{$s_{\mathbf{j}}$}\] For any word $w=w_1\ldots w_n$ or ordered set $S=\{w_1,\ldots,w_n\}$, we denote \[Q_S=Q_w=\omega_n\cdot w=\sum_{\mathbf{j}\sqsubseteq [2,n]}(-1)^{|\mathbf{j}|}s_{\mathbf{j}}\cdot w.\index{$Q_S$}\index{$Q_w$}\] Notice that $w$ is the unique word involved in the summand of $Q_w$ such that the first alphabet is $w_1$. Furthermore, the coefficient of $w$ in $Q_w$ is 1. For  example, $\omega_3\cdot abc=abc-bac-cab+cba$. It is well-known that $\omega_n\in\Des{n}{R}$ and $\omega_n^2=n\omega_n$. Therefore, if $n$ is a unit in $R$, then $\frac{1}{n}\omega_n$ is an idempotent in $\Des{n}{R}$.

The Lie module for the group algebra $R\sym{n}$ is defined as the right ideal \[\Lie_R(n):=\omega_nR\sym{n}.\index{$\Lie_R(n)$}\] The following proposition is well-known.

\begin{prop}\label{P: Lie(n) basis} The set $\{\omega_n\sigma:(1)\sigma=1\}$ is an $R$-basis for $\Lie_R(n)$.
\end{prop}

\subsection{The higher Lie modules} In the case $p=0$, there is another basis for the descent algebra $\Des{n}{F}$ found by Blessenohl-Laue \cite{BL}. In fact, in their paper, some results and proofs hold over arbitrary commutative ring $R$ with 1. We have summarized the results we shall need in Theorem \ref{T: BL results} below and refer the readers to the proofs in that paper.

For each $q\in\C(n)$, we define \[\omega_q=\omega^q\Xi^q\in\Des{n}{R}\index{$\omega_q$}\index{$\omega^q$}\]  where $\omega^q:=\omega_{q_1}^{+q^+_0}\cdots\omega_{q_k}^{+q^+_{k-1}}\in R\sym{q}$ (see \cite[Proposition 1.1]{BL}). Clearly, $\omega_{(n)}=\omega_n$. The element $\omega_q$ can also be defined using the convolution product for the symmetric group but we have chosen this presentation which suits us best. By \cite[Proposition 1.2]{BL}, when $p=0$, the set $\{\omega_q:q\in \C(n)\}$ is an $F$-basis for $\Des{n}{F}$. This is not true when $p>0$ (see Example \ref{Eg: p=2, (2,1)} below).

The following theorem summarizes a few properties which we will be using throughout. 

\begin{thm}[\cite{BL}]\label{T: BL results} Let $r,q\in\C(n)$. In $\Des{n}{R}$, we have
\begin{enumerate}[(i)]
  \item $\omega_q=\sum_{s\sref q}(-1)^{\ell(s)-\ell(q)}F_q(s)\Xi^s$,
  \item $\Xi^r\omega_q=\sum_{q\approx s\sref r}|\overline{N_{r,q}^s}|\omega_s$,
  \item $\omega_q\omega_r=q?\omega_q$ if $q\approx r$,
  \item $\omega_q\omega_r=0$ unless $q\wref r$.
\end{enumerate}
\end{thm}

When $q?$ is a unit in $R$, the element $\nu_q=\frac{1}{q?}\omega_q$ \index{$\nu_q$}in $\Des{n}{R}$ is called a higher Lie idempotent due to Theorem \ref{T: BL results}(iii). In particular, when $p=0$, any idempotent $e\in \Des{n}{F}$ is called a higher Lie idempotent if $eF\sym{n}=\nu_qF\sym{n}$ for some $q\vDash n$.

For any $q\in\C(n)$, we define the higher Lie module with respect to $q$ as the right ideal \[\Lie_R(q)=\omega_qR\sym{n}.\index{$\Lie_R(q)$}\] We have the following basic property.

\begin{lem}\label{L: translate w} Suppose that $q,r$ are compositions of $n$ and $q\approx r$. We have $\omega_q=\sigma\omega_r$ for some $\sigma\in \sym{n}$. In particular, $\Lie_R(q)\cong \Lie_R(r)$.
\end{lem}
\begin{proof} Let $k=\ell(q)$, $\tau\in \sym{k}$ such that $q=(r_{(1)\tau},\ldots,r_{(k)\tau})$ and $\sigma=\tau_r$ be given as in Equation \ref{Eq: tau_r}. So $\Xi^q=\sigma\Xi^r$. Also, $\sigma\omega^r\sigma^{-1}=\omega^q$ as $\omega^r$ (respectively, $\omega^q)$ is a product of Dynkin-Specht-Wever elements on disjoint supports and $\sigma$ acts by permuting the supports according to $\tau_r$. Therefore, \[\sigma\omega_r=\sigma\omega^r\sigma^{-1}\sigma\Xi^r=\omega^q\Xi^q=\omega_q.\] The isomorphism is therefore given by the left multiplication by $\sigma$.
\end{proof}

In the case when $p=0$, it is well-known that $\dim_F\Lie_F(q)=\frac{n!}{q?}$ (see, for example, \cite[Theorem 8.24]{Reu}). Therefore, using Lemma \ref{L: dim mod p}, we obtain the following inequality. 

\begin{prop}\label{P: dim higher lie 0} For any $q\in\C(n)$, we have $\dim_F\Lie_F(q)\leq \frac{n!}{q?}$.
\end{prop}

The inequality in Proposition \ref{P: dim higher lie 0} is not an equality in general. We give an example below and also refer the reader to Appendix \ref{Appen C} for the computational data using Magma \cite{Magma}.

\begin{eg}\label{Eg: p=2, (2,1)} Let $p=2$. Using Theorem \ref{T: BL results}(i), in $\Des{3}{F}$, we have \[\omega_{(2,1)}=2\Xi^{(2,1)}-\Xi^{(1,1,1)}=\Xi^{(1,1,1)}=\sum_{\sigma\in\sym{3}}\sigma.\] Therefore $\Lie_F((2,1))\cong F$ and the inequality in Proposition \ref{P: dim higher lie 0} is strict in this case. Similarly, we have $\omega_{(1,2)}=\omega_{(2,1)}=\omega_{(1,1,1)}=\Xi^{(1,1,1)}$ and $\omega_{(3)}=\Xi^{(3)}+\Xi^{(2,1)}+\Xi^{(1,1,1)}$. Therefore, $\{\omega_q:q\in \C(3)\}$ cannot be a basis for $\Des{3}{F}$.
\end{eg}

Continuing with the case when $p=0$, the ordinary irreducible constituents of $\Lie_F(q)$ has also been computed in \cite{Schoc03} (see also \cite[Theorem 5.11]{AS}). The computation is done by reducing the general case to the case when $q=(d^k)\vdash dk$ for some $d,k\in\NN$. Since we do not need the exact combinatorial description of the multiplicity, we simply denote the multiplicity of the irreducible character $\zeta^\mu$ in $\Lie_F((d^k))$ given in \cite[Main Theorem 3.1]{Schoc03} as $\sch^\mu_{d,k}$\index{$\sch^\mu_{d,k}$}. Furthermore, for partitions $\mu(1),\ldots,\mu(k)$ such that $\sum_{i=1}^k |\mu(i)|=|\lambda|=n$, let $c^\lambda_{\mu(1),\ldots,\mu(k)}$ \index{$c^\lambda_{\mu(1),\ldots,\mu(k)}$}be the multiplicity of the irreducible character $\zeta^\lambda$ in the induced module $\Ind^{\sym{n}}_{\sym{r}}(S^{\mu(1)}\boxtimes \cdots\boxtimes S^{\mu(k)})$ where $r=(|\mu(1)|,\ldots,|\mu(k)|)$; namely, it is just the number obtained using Littlewood-Richardson Rule repeatedly (see \cite{JamesKerber}). 

\begin{thm}[{\cite[Lemma 2.1 and Theorem 3.1]{Schoc03}}]\label{T: Schoc Lie(q)} Suppose that $p=0$. For any $\lambda\in\P(n)$ and $q\in\C(n)$, the multiplicity of the irreducible character $\zeta^\lambda$ in $\Lie_F(q)$ is \[C^\lambda_q=\sum_{\mu(i)\vdash im_i(q)}c^{\lambda}_{\mu(1),\ldots,\mu(n)}\prod^n_{i=1}\sch^{\mu(i)}_{i,m_i(q)}.\index{$C^\lambda_q$}\]
\end{thm}

\subsection{The Solomon's epimorphism}\label{SS: Solomon epi}

Let $q\vDash n$. The Young character $\varphi^q$ \index{$\varphi^q$}is defined as the character of $M^q_\Z$ where, for each $\mu\in\P(n)$, $\varphi^q(\mu)$ is the number of right cosets (or $q$-tabloids) of $\sym{q}$ in $\sym{n}$ fixed by a permutation with cycle type $\mu$, where, by abuse of notation, we have identified $\mu$ with the conjugacy class $\ccl{\mu}$ of $\sym{n}$. Therefore, $\varphi^q=\varphi^{q'}$ if $q\approx q'$. We have the following lemma.

\begin{lem}\label{L: phi nonzero} If $\varphi^q(\mu)\neq 0$ then $\mu\wref q$. Furthermore, $\varphi^q(q)=\facmulti{q}$.
\end{lem}

Recall the set $N^s_{q,r}$ defined in Subsection \ref{SS: generality}. We have the following well-known identity, analogous to the Mackey's formula, \[\varphi^q\varphi^r=\sum_{s\in\C(n)} |N^s_{q,r}|\varphi^s\] which essentially gives rise to the next theorem. We denote $\varphi^{q,R}$ for the $R$-valued Young character, i.e., for any $\mu\in\P(n)$, \[\varphi^{q,R}(\mu)=\varphi^q(\mu)\cdot 1_R\in R.\index{$\varphi^{q,R}$}\]  Let $\Ccl{n}{R}$ \index{$\Ccl{n}{R}$}be the $R$-linear span of the $R$-valued Young characters.

\begin{thm}[{\cite{AW,Sol}}]\label{T: Sol epi} The $F$-linear map \[c_{n,F}:\Des{n}{F}\to \Ccl{n}{F}\index{$c_{n,F}$}\] sending $\Xi^q$ to $\varphi^{q,F}$ is a surjective $F$-algebra homomorphism. Furthermore, $\ker(c_{n,F})=\mathrm{rad}\Des{n}{F}$ is spanned by $\Xi^q$ such that $\lambda(q)\not\in\pReg(n)$ and together with $\Xi^q-\Xi^r$ such that $q\approx r$ with $q\neq r$.
\end{thm}

The map $c_{n,F}$ in the theorem is called the Solomon's epimorphism. For $\lambda\in \pReg(n)$, let $\Char_{\lambda,F}$ \index{$\Char_{\lambda,F}$}be the characteristic function on the $p$-equivalent class $\ccl{\lambda,p}$.  The set $\{\Char_{\lambda,F}:\lambda\in\pReg(n)\}$ forms a basis and complete set of orthogonal primitive idempotents of $\Ccl{n}{F}$ (see \cite[Proposition 5]{ES}). Furthermore, when $p=0$, by \cite[Proposition 1]{JS}, under the Solomon's epimorphism, we have \[c_{n,F}(\nu_q)=c_{n,F}\left (\frac{1}{q?}\omega_q\right )=\Char_{\lambda(q),F}.\]

In the $p=0$ case, there are different sets of orthogonal primitive idempotents of $\Des{n}{F}$ in the literature. Notably, those given by Blessenohl-Laue \cite{BL,BL02} (see Proposition \ref{P: ord idem} below) and Garsia-Reutenauer \cite{GR} (see Subsection \ref{SS: higher Lie power}). In the $p>0$ case, Erdmann-Schocker \cite[Corollary 6]{ES} showed that there exists a complete set of orthogonal primitive idempotents $\{e_{\lambda,F}:\lambda\in\pReg(n)\}$ for $\Des{n}{F}$ such that $c_{n,F}(e_{\lambda,F})=\Char_{\lambda,F}$.


\begin{prop}\label{P: ord idem} Suppose that $p=0$. Using the orthogonalization procedure (see \cite[\S4]{PR} and also \cite[Proposition 2.5]{Schoc032}) for the set of higher Lie idempotents $\{\nu_\lambda:\lambda\in\P(n)\}$, we obtain a complete set $\{e_\lambda:\lambda\in\P(n)\}$ of orthogonal primitive idempotents of $\Des{n}{F}$ such that $\sum_{\lambda\in\P(n)}e_\lambda=1$ and they satisfy the `triangularity property' \[e_\lambda=\frac{1}{\facmulti{\lambda}}\Xi^\lambda+\epsilon_\lambda\] where $\epsilon_\lambda$ is a linear combination of the $\Xi^\xi$ such that $\xi<\lambda$.
\end{prop}
\begin{proof} Let $m=|\P(n)|$ and $\lambda^{(1)},\ldots,\lambda^{(m)}$ be the partitions ordered so that, if $\lambda^{(j)}\wref\lambda^{(i)}$, then $j\leq i$. Notice that $\lambda^{(m)}=(n)$ and $\lambda^{(1)}=(1^n)$. Let $\nu_i=\nu_{\lambda^{(i)}}$ and $\Char_i=\Char_{\lambda^{(i)},F}$. By Theorem \ref{T: BL results}(iv), $\nu_i\nu_j=0$ whenever $i>j$. Therefore the set $\{e_i:i\in [1,m]\}$ forms a set of orthogonal primitive idempotents such that $\sum_{i=1}^me_i=1$ where \[e_i=\nu_i(1-\nu_{i+1})\cdots (1-\nu_m).\] It is a complete set due to Theorem \ref{T: Sol epi}. Since, by Theorem \ref{T: BL results}(i), $\nu_i$ is a linear combination of the $\Xi^\xi$ such that $\xi\sref\lambda^{(i)}$, we have, by Corollary \ref{C: product in Descent}(i), \[e_i=\sum_{s\sref\lambda^{(i)}}d_s\Xi^s\] for some $d_s\in F$. On the other hand, we have $c_{n,F}(\nu_i)=\Char_{\lambda^{(i)},F}$. Since $c_{n,F}$ is an $F$-algebra homomorphism, we have
\begin{align}\label{Eq: 2.1}
\sum_{s\sref\lambda^{(i)}}d_s\varphi^{s,F}=c_{n,F}(\sum_{s\sref\lambda^{(i)}}d_s\Xi^s)=c_{n,F}(e_i)=\Char_i(1-\Char_{i+1})\cdots (1-\Char_m)=\Char_i.
\end{align} For $\lambda^{(i)}\wref s\sref \lambda^{(i)}$, we must have $s=\lambda^{(i)}$. Evaluating Equation \ref{Eq: 2.1} at $\lambda^{(i)}$, by Lemma \ref{L: phi nonzero}, we get $\varphi^{s,F}(\lambda^{(i)})=0$ unless $s=\lambda^{(i)}$ and hence $1=\facmulti{\lambda^{(i)}}d_{\lambda^{(i)}}$.
\end{proof}

\subsection{The higher Lie powers}\label{SS: higher Lie power}

Let $V$ be an $R$-module with finite rank and $T^n(V)=V^{\otimes n}$ \index{$T^n(V)$}be the $n$-fold tensor product of $V$ over $R$. By convention, $T^0(V)=R$. Then $T(V)=\bigoplus_{n\in\NN_0}T^n(V)$ \index{$T(V)$}is naturally an associative $R$-algebra with unit. It can be made into a Lie algebra by means of the Lie bracket \[[v,w]=v\otimes w-w\otimes v\] for all $v,w\in T(V)$. We denote by $L(V)$ the free  Lie subalgebra of $T(V)$ generated by $V$. For each $n\in\NN_0$, the $n$th Lie power of $V$ is defined as $L^n(V)=T^n(V)\cap L(V)$\index{$L^n(V)$}. Every element in $L^n(V)$ is called a Lie element and a Lie monomial is $P_1\otimes P_2\otimes \cdots \otimes P_m$ \index{$P_1\otimes P_2\otimes \cdots \otimes P_m$}(or simply written as $P_1P_2\cdots P_m$\index{$P_1P_2\cdots P_m$}) such that each $P_i$ is a Lie element in $L^{n_i}(V)$ for some $n_i\in \NN$. In this case, we denote $n_i=|P_i|$ and the composition $(n_1,\ldots,n_m)$ is called the type of $P_1\otimes\cdots\otimes P_m$.

The symmetric group $\sym{n}$ acts on $T^n(V)$ via the P\'{o}lya action; namely, \[\sigma\cdot (v_1\otimes \cdots\otimes v_n)=v_{1\sigma}\otimes\cdots\otimes v_{n\sigma}.\index{$\sigma\cdot (v_1\otimes \cdots\otimes v_n)$}\] It is well-known that $L^n(V)= \omega_n\cdot T^n(V)$. If $V$ is an $RG$-module then it is readily checked that $T^n(V)$ is an ($R\sym{n}$-$RG$)-bimodule and hence $L^n(V)$ is a (right) $RG$-module. Let $G=\GL(V)$.  The Lie power $L^n(V)$ is a polynomial representation of $\GL(V)$ and therefore an $S(m,n)$-module where $m=\dim_F V$.

Suppose that $A$ is a finite set consisting of distinct alphabets. The $R$-module $R[A]$ \index{$R[A]$}with an $R$-basis consisting of all words in $A$ can be identified with $T(V)$ where $V$ has $R$-basis $A$. As such, the symmetric group $\sym{n}$ acts on the $n$th homogeneous component of $R[A]$. Furthermore, the multiplication in $R[A]$ induced by the concatenation of words is the same as the multiplication in $T(V)$. 

When $V$ is an $F$-vector space, we have the following well-known dimension formula for the Lie powers by Witt in which $\mu$ denotes the M\"{o}bius function\index{$\mu(d)$}.

\begin{thm}[\cite{Witt}]\label{T: lie dim} Let $m=\dim_F V$. The dimension of $L^n(V)$ is \[\frac{1}{n}\sum_{d\mid n}\mu(d)m^{n/d}.\]
\end{thm}

Also, when $p=0$, the structure of $L^n(V)$ is well-studied by the earlier work of Thrall \cite{Thrall}, Brandt \cite{Brandt}, Wever \cite{Wever}, Klyachko \cite{Klya} and Kraskiewicz-Weyman \cite{KW}. When $p>0$, the decomposition of $L^n(V)$ appears to depend on the $p$-adic valuation of $n$. In particular, when $p\nmid n$, we have the following result by Donkin-Erdmann.

\begin{thm}[{\cite[\S3.3 Theorem]{DE}}]\label{T: DE proj decomp} Suppose that $p>0$, $n\in\NN$, $p\nmid n$ and $m=\dim_FV$. We have $L^n(V)\cong \bigoplus_{\nu\in\pReg(m,n)} n_\nu T(\nu)$ where, for each $\nu\in\pReg(m,n)$, \[n_\nu=\frac{1}{n}\sum_{d\mid n}\mu(d)\beta^\nu(\sigma_d)\index{$n_\nu$}\] where $\sigma_d$ \index{$\sigma_d$}is the partition $((n/d)^d)$.
\end{thm}

Therefore, when $\dim_F V\geq n$ and $p\nmid n$, using Theorem \ref{T: DE proj decomp} and applying the Schur functor $f$, we have $f(L^n(V))\cong \Lie_F(n)$ and obtain the decomposition of the Lie module $\Lie_F(n)$ into the projective indecomposable modules $P^\nu$'s, that is \[\Lie_F(n)\cong \bigoplus_{\nu\in\pReg(n)}n_\nu P^\nu\] where $n_\nu$ is given as in the theorem.

\medskip

We now turn to the higher Lie powers. For each $q\in \C(n)$, we define the higher Lie power $L^q(V)=\omega_q\cdot T^n(V)$\index{$L^q(V)$}.  Unlike the $p=0$ case in which $T^n(V)$ decomposes as a direct sum of $L^\lambda(V)$ such that $\lambda\vdash n$, the $p>0$ case does not. The following is an easy example. 


\begin{eg} Let $p=2$ and $V$ has an $F$-basis $\{v_1,v_2\}$. Since $\omega_{(2)}=\Xi^{(1,1)}=\omega_{(1,1)}$, we have both $L^{(2)}(V)$ and $L^{(1,1)}(V)$ are identical with the $F$-basis $\{v_1\otimes v_2+v_2\otimes v_1\}$.
\end{eg}

Suppose further that $V$ is a right $RG$-module.  As such, $L^q(V)$ is a right $RG$-module. By Lemma \ref{L: translate w}, we have the following.

\begin{lem}\label{L: Lq isom Lr} If $q\approx r$ then $L^q(V)\cong L^r(V)$ as $RG$-modules.
\end{lem}

\subsection{Ordinary idempotents}\label{SS: ord idem} This subsection is devoted to the description of the results in \cite{GR} and we assume that $p=0$ throughout. One of the main adaptation is by reversing the order of multiplication according to the conventions. Similar to \cite{BL}, some of the results and proofs in \cite{GR} hold over $R$ and we shall refer the reader to the paper. In particular, we have:

\begin{thm}[{\cite[Theorem 2.1]{GR}}]\label{T: GR 2.1} If $q\in\C(n)$, $k=\ell(q)$ and $P_1P_2\cdots P_m$ is a Lie monomial of type $r\in\C(n)$ then, over $R$, \[\Xi^q\cdot P_1P_2\cdots P_m=\sum P_{S_1}P_{S_2}\cdots P_{S_k}\] where the sum is taken over all $S_1\sqcup\cdots\sqcup S_k=[1,m]$ such that, for each $i\in [1,k]$, $(|P_{t_1}|,\ldots,|P_{t_{c_i}}|)\vDash q_i$ and $P_{S_i}=P_{t_1}\cdots P_{t_{c_i}}$ if $S_i=\{t_1<\cdots<t_{c_i}\}$ for some $c_i$.  In particular, if $r\not\wref q$, we have $\Xi^q\cdot P_1P_2\cdots P_m=0$.
\end{thm}

For example, if $r,q$ are compositions such that $r\not\wref q$, then, by Theorem \ref{T: GR 2.1}, $\Xi^q\omega^r=0$ as $\omega^r\in R\sym{r}$ can be viewed as a Lie monomial of type $r$.

\smallskip Recall the notations $\ell(r,q)$ and $\ell!(r,q)$ for a (strong) refinement $r$ of $q$ in Subsection \ref{SS: generality}. For any $q\in\C(n)$ and $\lambda\in\P(n)$, we write
\begin{align*}
  I_q&=\sum_{r\sref q}\frac{(-1)^{\ell(r)-\ell(q)}}{\ell(r,q)}\Xi^r\in \Des{n}{F},\index{$I_q$}\\
  E_\lambda&=\frac{1}{\ell(\lambda)!}\sum_{\lambda(q)=\lambda}I_q\in \Des{n}{F}.\index{$E_\lambda$}
\end{align*} The set $\{E_\lambda:\lambda\in\P(n)\}$ is a complete set of orthogonal primitive idempotents of $\Des{n}{F}$. Each of the element $\Xi^q$ can be written explicitly in terms of the $I_r$'s as follows.

\begin{thm}[{\cite[Theorem 3.4]{GR}}]\label{T: GR 3.4} The set $\{I_q:q\in\C(n)\}$ is a basis for $\Des{n}{F}$. Furthermore, for $q\vDash n$, we have \[\Xi^q=\sum_{r\sref q}\frac{1}{\ell!(r,q)}I_r.\]
\end{thm}

Next,  we need to examine how the set of elements $\Xi^q$'s,  $I_r$'s and $E_\lambda$'s interact with each others.  We give a quick summary of the results we need and draw easy observations following them.

\begin{thm}[{\cite[Theorem 4.1]{GR}}]\label{T: GR 4.1} Let $q,r\in\C(n)$, $k=\ell(q)$ and $m=\ell(r)$. Then \[\Xi^qI_r=\sum I_{(r_j)_{j\in S_1}\cont\cdots\cont (r_j)_{j\in S_k}}\] where the sum is taken over all $S_1\sqcup\cdots\sqcup S_k=[1,m]$ such that, for each $i\in [1,k]$, we have $(r_j)_{j\in S_i}\vDash q_i$.  In particular, if $r\not\wref q$, then
\begin{enumerate}[(i)]
  \item $\Xi^qI_r=0$,
  \item $I_qI_r=0$, and
  \item $E_\lambda I_r=0$ where $\lambda=\lambda(q)$.
\end{enumerate}
\end{thm}
\begin{proof} The equation $I_qI_r=0$ follows since $I_q$ is a linear combination of $\Xi^s$ such that $s\sref q$ and $\Xi^sI_r=0$. This implies $E_\lambda I_r=0$.
\end{proof}

\begin{thm}[{\cite[Theorem 4.2(1,2,4)]{GR}}]\label{T: EI=0} Let $q,r\in\C(n)$ such that $q\approx r$ and let $\lambda=\lambda(q)\in\P(n)$. We have
\begin{enumerate}[(i)]
  \item $I_qI_r=\Xi^qI_r=\facmulti{\lambda}I_q$,
  \item $I_qE_\lambda=I_q$, and
  \item $E_\lambda I_q=\facmulti{\lambda}E_\lambda$.
\end{enumerate}
\end{thm}

We give an example illustrating the above results.

\begin{eg} By definition, we have
\begin{align*} I_{(2,1,1)}&=\Xi^{(2,1,1)}-\frac{1}{2}\Xi^{(1^4)}, &I_{(1,2,1)}&=\Xi^{(1,2,1)}-\frac{1}{2}\Xi^{(1^4)}, & I_{(1,1,2)}&=\Xi^{(1,1,2)}-\frac{1}{2}\Xi^{(1^4)}.
\end{align*} By Theorem \ref{T: GR 4.1}, we have
$\Xi^{(3,1)}I_{(1,1,2)}=2I_{(1,2,1)}$ where $\{1,3\}\sqcup \{2\}$ and $\{2,3\}\sqcup \{1\}$ are the only possibilities for such $S_1\sqcup S_2$.  Also, by definition, \[E_{(2,1,1)}=\frac{1}{6}(I_{(2,1,1)}+I_{(1,2,1)}+I_{(1,1,2)})=\frac{1}{6}(\Xi^{(2,1,1)}+\Xi^{(1,2,1)}+\Xi^{(1,1,2)}-\frac{3}{2}\Xi^{(1^4)}).\] Using Theorem \ref{T: EI=0}(i), we have, for instance,  \[E_{(2,1,1)}I_{(1,2,1)}=\frac{1}{6}(2I_{(2,1,1)}+2I_{(1,2,1)}+2I_{(1,1,2)})=2E_{(2,1,1)}.\] One could of course also check the equations above using Theorem \ref{T: GR 1.1} as well.
\end{eg}

To conclude the subsection, we record the following proposition we shall need later.

\begin{prop}\label{P: E iso nu} Suppose that $p=0$. Let $q\vDash n$ and $\lambda(q)=\lambda$. Then $E_\lambda F\sym{n} \cong  \Lie_F(q)$ as $F\sym{n}$-modules.
\end{prop}
\begin{proof} By \cite[Equation (1.6)]{GR}, we have $c_{n,\Q}(E_\lambda)=\Char_{\lambda,\Q}$. Using \cite[Lemma 4.1]{Schoc031}, we have the required isomorphism.
\end{proof}

\subsection{Symmetric functions and formal characters}\label{SS: symmetric functions} We refer the reader to \cite{Mac} for the necessary details for the theory of symmetric functions. Let $\Symm(X)$ (or simply, $\Symm$) \index{$\Symm$}be the set consisting of symmetric functions in a set of countably infinite commuting variables $X=\{x_i:i\in\NN\}$ over $\Z$. Notice that $\Symm$ is a commutative graded ring where $\Symm=\bigoplus_{n\in\NN_0}\Symm^n$ and $\Symm^n$ consists of homogeneous symmetric functions of degree $n$.  For any $n\in\NN$, the $n$th complete symmetric function and $n$th power sum are
\begin{align*}
  h_n&=\sum_{i_1\leq \cdots\leq i_n}x_{i_1}\cdots x_{i_n},\index{$h_n$}\\
  p_n&=\sum_{i=1}^\infty x_i^n\index{$p_n$}
\end{align*} respectively. By convention, $h_0=1=p_0$.  For each sequence $\alpha$ in $\NN_0$ such that $\sum^\infty_{i=1}\alpha_i=n$, we denote $h_\alpha=\prod_{i=1}^\infty h_{\alpha_i}$ \index{$h_\alpha$}and $p_\alpha=\prod_{i=1}^\infty p_{\alpha_i}$ \index{$p_\alpha$}for the elements in $\Symm^n$.  
Given $f,g\in \Symm$. The plethysm of $f$ and $g$ is denoted as $f\circ g$,\index{$f\circ g$} that is, since $\Symm$ is the polynomial ring on $\{p_i:i\in\NN\}$, let $f=F((p_i)_{i=1}^\infty)$ and $g=G((p_i)_{i=1}^\infty)$, the plethysm is defined as \[f\circ g=F((g_i)_{i=1}^\infty)\] where $g_i=G((p_{ij})_{j=1}^\infty)$ and $p_{ij}$ is the power sum at degree $ij$.

Let $\R$ \index{$\R$}be the generalised characters for the symmetric groups, that is, $\R=\bigoplus_{n=0}^\infty \Ccl{n}{\Z}$ where $\Ccl{n}{\Z}$ is the set consisting of generalised characters for $\sym{n}$ as in Subsection \ref{SS: Solomon epi}. Recall that $\zeta^\lambda$ is the irreducible character of $\sym{n}$ and  $\Char_{\lambda,\Z}$ is the characteristic function on the conjugacy class labelled by $\lambda\in\P(n)$. Let $\ch:\R\to \Symm$ \index{$\ch$}be the characteristic map such that $\ch(\zeta^\lambda)=s_\lambda$ where $s_\lambda$ is the Schur function. It is well-known that \[\ch(\lambda?\Char_{\lambda,\Z})=p_\lambda.\]

Consider the $p$-modular system $(\Q_p,\Z_p,\Fp)$ where $\Q_p$ and $\Z_p$ are the $p$-adic numbers and $p$-adic integers respectively and $\Fp$ is the field of $p$ elements.  Let $\R^n_\text{proj}$ \index{$\R^n_{\text{proj}}$}be the $\Z$-span of the characters of the projective $\Z_p \sym{n}$-modules, i.e., by \cite[(18.26)]{CR1}, it consists of elements in $\Ccl{n}{\Z}$ that vanish on $p$-singular elements of $\sym{n}$.

For each $S(m,n)$-module $W$, by abuse of the notation $\ch$, the formal character $\ch(W)$ (\cite[\S3.4]{Green830}) is a symmetric function in $X$ (by this, we understand that we restrict to the first $m$ variables). We have the following well-known character formula for the Lie powers by Brandt.

\begin{thm}[\cite{Brandt}]\label{T: lie ch} Consider $V$ as the  $F\GL(V)$-module. The $n$th Lie power $L^n(V)$ has formal character \[\ch(L^n(V))=\frac{1}{n}\sum_{d\mid n}\mu(d)p_d^{n/d}.\]
\end{thm}

\section{Modular Idempotents of $\Des{n}{F}$}\label{S: mod idem}

Throughout this section, we assume that $p>0$ and,  by Theorem \ref{T: Sol epi}, we identify $\Des{n}{F}/\rad(\Des{n}{F})$ with $\Ccl{n}{F}$ through the Solomon's epimorphism $c_{n,F}$. We shall give a construction for the modular idempotents $e_{\mu,F}$'s for the descent algebra $\Des{n}{F}$ which satisfy the property as in \cite[Corollary 6]{ES}. The main result is Theorem \ref{T: modular idem}. A comparison with Proposition \ref{P: ord idem} shows that these two sets of idempotents enjoy the similar `triangularity property'. We remark that the construction of the modular idempotents in this section can be extended for the descent algebras of Coxeter groups. We shall present them in a forthcoming paper.

We begin with some notations. Recall the Young characters $\varphi^\lambda$ in Subsection \ref{SS: Solomon epi}. Fix a total order $\leq$ on $\P(n)$ refining the weak refinement $\wref$ such that \[\P(n)=\{(1^n)\leq \cdots\leq (n)\}\] and, with respect to the total order, let $\Phi^F=(\varphi^{\lambda,F}(\mu))_{\lambda,\mu\in\pReg(n)}$ \index{$\Phi^F$}where both the last row and column are labelled by $(n)$. When $\lambda\leq \mu$ and $\lambda\neq \mu$, we write $\lambda<\mu$. Notice that the total order also refines the strong refinement.

\begin{lem}\label{L: Phi} The matrix $\Phi^F$ is lower triangular with the diagonal entries $\varphi^{\lambda,F}(\lambda)=\facmulti{\lambda}\neq 0$.
\end{lem}
\begin{proof} This follows from Lemma \ref{L: phi nonzero} and our choice of the total order. Since $\lambda$ is $p$-regular, we have $m_i(\lambda)<p$ for all $i$ and therefore $\facmulti{\lambda}\neq 0$ in $F$.
\end{proof}

By Lemma \ref{L: Phi}, the matrix $\Phi^F$ is invertible. Let $\Psi^F=(b_{\lambda,\mu})$ \index{$\Psi^F$}\index{$b_{\lambda,\mu}$}be the inverse matrix of $\Phi^F$. For each $\lambda\in\pReg(n)$, define \[f_{\lambda,F}=\sum_{\mu\in \pReg(n)}b_{\lambda,\mu}\Xi^\mu\in \Des{n}{F}.\index{$f_{\lambda,F}$}\] Recall the ideals $Y_q$ and $Y_q^\circ$ of $\Des{n}{F}$ as in Subsection \ref{SS: descent}.

\begin{lem}\label{L: f*} Let $\lambda\in\pReg(n)$.
\begin{enumerate}[(i)]
  \item $c_{n,F}(f_{\lambda,F})=\Char_{\lambda,F}$.
  \item For any $\pReg(n)\ni\mu\not\wref\lambda$, we have $b_{\lambda,\mu}=0$. In particular, \[f_{\lambda,F}=\frac{1}{\facmulti{\lambda}}\Xi^\lambda+\sum_{\pReg(n)\ni\swref{\mu}{\lambda}}b_{\lambda,\mu}\Xi^\mu.\]
  \item For any positive integer $r$, we have $(f_{\lambda,F})^r=\frac{1}{\facmulti{\lambda}}\Xi^\lambda+\epsilon_{\lambda,r}$ for some $\epsilon_{\lambda,r}\in Y_\lambda^\circ$.
  \item $\sum_{\lambda\in\pReg(n)}f_{\lambda,F}=1$.
\end{enumerate}
\end{lem}
\begin{proof} For part (i), for each $\gamma\in\pReg(n)$, since $c_{n,F}$ is an $F$-algebra homomorphism, we have \[c_{n,F}(f_{\lambda,F})(\gamma)=\sum_{\mu\in\pReg(n)}b_{\lambda,\mu}\varphi^{\mu,F}(\gamma)=\delta_{\lambda,\gamma}.\] Since $\{\Char_{\lambda,F}:\lambda\in \pReg(n)\}$ forms a basis for $\Ccl{n}{F}$, we have $c_{n,F}(f_{\lambda,F})=\Char_{\lambda,F}$.

For the first assertion in part (ii), suppose that $b_{\lambda,\mu}\neq 0$ for some largest (with respect to the total order $\leq$ we have fixed on $\pReg(n)$) $\mu$ such that $\mu\not\wref\lambda$. Evaluate part (i) at $\mu$, we get
\begin{align}\label{Eq: 3.1}
0=\sum_{\xi\in\pReg(n)}b_{\lambda,\xi}\varphi^{\xi,F}(\mu)= \sum_{\pReg(n)\ni\xi\wref\lambda}b_{\lambda,\xi}\varphi^{\xi,F}(\mu)+\sum_{\pReg(n)\ni\xi\not\wref\lambda}b_{\lambda,\xi}\varphi^{\xi,F}(\mu).
\end{align} If $\xi\wref\lambda$, then $\varphi^{\xi,F}(\mu)=0$ by Lemma \ref{L: phi nonzero} (else, $\mu\wref \xi\wref \lambda$ and hence $\mu\wref\lambda$). Therefore the first summand of the right-hand side of Equation \ref{Eq: 3.1} is 0. Suppose now that $\xi\not\wref \lambda$. If $\mu\torder \xi$ then, in particular, $\mu\not\wref \xi$ and, again by Lemma \ref{L: phi nonzero}, $\varphi^{\xi,F}(\mu)=0$. If $\xi\torder \mu$, by our choice of $\mu$, we have $b_{\lambda,\xi}=0$. Therefore, using Lemma \ref{L: phi nonzero}, Equation \ref{Eq: 3.1} reduces to \[0=b_{\lambda,\mu}\varphi^{\mu,F}(\mu)=b_{\lambda,\mu}\facmulti{\mu}\neq 0.\] This is a contradiction.  Therefore \[f_{\lambda,F}=\sum_{\pReg(n)\ni\mu\wref \lambda}b_{\lambda,\mu}\Xi^\mu.\] Since $\Psi^F$ is the inverse matrix of $\Phi^F$, we have $b_{\lambda,\lambda}=\frac{1}{\facmulti{\lambda}}$.

We now prove part (iii). Notice that
\begin{align}\label{Eq: 3.2}
  (f_{\lambda,F})^2&=\left (\frac{1}{\facmulti{\lambda}}\Xi^\lambda+\sum_{\pReg(n)\ni\swref{\mu}{\lambda}}b_{\lambda,\mu}\Xi^\mu\right )^2=\frac{1}{\facmulti{\lambda}^2}(\Xi^\lambda)^2+Z.
\end{align} where $Z$ is a linear combination of the products of the form $\Xi^\zeta \Xi^\gamma$ where $\zeta,\gamma\in\C(n)$ such that either $\swref{\zeta}{\lambda}$ or $\swref{\gamma}{\lambda}$. By Corollary \ref{C: product in Descent}, $Z\in Y_\lambda^\circ$ and $(\Xi^\lambda)^2=\facmulti{\lambda}\Xi^\lambda+Y$ where $Y\in Y_\lambda^\circ$. Substitute these into Equation \ref{Eq: 3.2}, we get the case for $r=2$. The general case follows inductively.

For part (iv), let $z=\sum_{\lambda\in\pReg(n)}f_{\lambda,F}$. Notice that, for any $\gamma\in\pReg(n)$, we have \[c_{n,F}(z)(\gamma)=\sum_{\lambda,\mu\in\pReg(n)} b_{\lambda,\mu}\varphi^{\mu,F}(\gamma)=\sum_{\lambda\in\pReg(n)}\delta_{\gamma,\lambda}=1.\] Therefore $c_{n,F}(z)=1$ in $\Ccl{n}{F}$.  Since $c_{n,F}(\Xi^{(n)})=1$, we have $z-\Xi^{(n)}\in \ker(c_{n,F})$.  Since $z-\Xi^{(n)}$ is a sum involving $\Xi^\mu$ such that $(n)\neq \mu\in\pReg(n)$, by Theorem \ref{T: Sol epi}, we must have $z-\Xi^{(n)}=0$.
\end{proof}


The elements $\{f_{\lambda,F}:\lambda\in\pReg(n)\}$ are generally not idempotents and far from being orthogonal to each other. We want to lift $f_{\lambda,F}+\rad(\Des{n}{F})=\Char_{\lambda,F}$ and orthogonalize them. To do so, we use the idea in the proofs of Idempotent Lifting Theorem (see, for example, \cite[Theorem 1.7.3]{Ben1}) and \cite[Corollary 1.7.4]{Ben1} as follows. 

\begin{prop}\label{P: Idem Lift} Let $A$ be a finite-dimensional algebra over $F$ and $N$ be a nilpotent ideal in $A$.
\begin{enumerate}[(i)]
\item Suppose that $c\in A/N$ is an idempotent and $c=a+N$ for some $a\in A$. Then there is a large enough $k$ (depending just on the nilpotency index of $N$) such that $a^{p^k}$ is an idempotent lifting $c$.
\item Let $c_1,\ldots,c_n$ be orthogonal primitive idempotents in $A/N$ such that $\sum_{i=1}^nc_i=1+N$. Suppose that $e_1'=1$ and, for $i>1$, $e_i'$ is a lift of $\sum_{j\geq i} c_j$ to an idempotent in $e_{i-1}'Ae_{i-1}'$. Then $\{e_i=e_i'-e_{i+1}':i\in [1,n]\}$ are orthogonal primitive idempotents in $A$ such that $\sum_{i=1}^ne_i=1$ and $e_i+N=c_i$.
\end{enumerate}
\end{prop}

We write $a^{p^\infty}$ for the idempotent in Proposition \ref{P: Idem Lift}(i). In the case of the descent algebra $\Des{n}{F}$, it is sufficient to choose $k_n$ with $p^{k_n}\geq n-1$ (see \cite[Theorem 3]{AW})) such that, for each idempotent $e\in \Ccl{n}{F}$ and $x\in \Des{n}{F}$ such that $c_{n,F}(x)=e$, we have $x^{p^{\infty}}=x^{p^{k_n}}$ \index{$x^{p^\infty}$}is an idempotent lifting $e$. We remark that, instead of repeatedly using $3a^2-2a^3$ for the idempotent-lifting procedure,  the `advantage' here allows us to lift $e$ immediately to an idempotent by taking enough high power of $p$ for $a$ and is easier to present in the subsequent proofs.

In the sequel, suppose that $m=|\pReg(n)|$, with respect to the total order $\torder$ we have fixed in the beginning of this section, and we write \[(n)=\lambda^{(m)}\torder\cdots\torder\lambda^{(1)}\] for the elements in $\pReg(n)$. Furthermore, to simplify notations, we denote
\begin{align*}
f_i&=f_{\lambda^{(i)},F},\ \  f_{\geq i}=\sum_{j\geq i}f_j, \ \ \Char_i=\Char_{\lambda^{(i)},F},\ \  \Char_{\geq i}=\sum_{j\geq i}\Char_j,
\index{$f_i$}\index{$f_{\geq i}$}\index{$\Char_i$}\index{$\Char_{\geq i}$}\end{align*} so that, by Lemma \ref{L: f*}(i), $c_{n,F}(f_{\geq i})=\Char_{\geq i}$. Furthermore, for each $i\in [1,m]$, let $Y_{\leq i}=\sum_{j\in [1,i]}Y_{\lambda^{(j)}}$ \index{$Y_{\leq i}$}and $Y_{\leq i}^\circ=Y_{\lambda^{(i)}}^\circ+\sum_{j\in [1,i-1]}Y_{\lambda^{(j)}}$, \index{$Y_{\leq i}^\circ$}i.e.,  $Y_{\leq i}$ and $Y_{\leq i}^\circ$ have the following respective $F$-bases
\begin{align*}
&\{\Xi^\xi:\C(n)\ni \xi\wref \lambda^{(j)},\ j\in [1,i]\},\\
&\{\Xi^\xi:\C(n)\ni \xi\wref \lambda^{(j)},\ j\in [1,i],\ \xi\not\approx\lambda^{(i)}\}.
\end{align*} The following follows from Corollary \ref{C: product in Descent}(ii).

\begin{lem}\label{L: Yi ideals} We have a chain of ideals of $\Des{n}{F}$ given by \[Y_{\leq 1}^\circ\subsetneq Y_{\leq 1}\subsetneq Y_{\leq 2}^\circ\subsetneq Y_{\leq 2}\subsetneq \cdots\subsetneq Y_{\leq m}^\circ\subsetneq Y_{\leq m}=\Des{n}{F}.\]
\end{lem}

We now define the elements $e_{\lambda,F}$'s which will eventually form a complete set of orthogonal primitive idempotents for $\Des{n}{F}$.

\begin{defn}\label{D: ErdSch e} Let $f_1'=1$ and, inductively, for $i\in [2,m]$, we define $f_i'=(f_{i-1}'f_{\geq i}f'_{i-1})^{p^\infty}$ \index{$f_i'$}and $f_{m+1}'=0$. For each $i\in [1,m]$, define \[e_{\lambda^{(i)},F}=f_i'-f_{i+1}'.\index{$e_{\lambda^{(i)},F}$}\]
\end{defn}



In view of Definition \ref{D: ErdSch e}, we want to show that the elements $f_i'$'s satisfy the hypothesis stated in Proposition \ref{P: Idem Lift}(ii). 



\begin{lem}\label{L: f'}\
\begin{enumerate}[(i)]
  \item For each $i\in [1,m]$, the element $f_i'$ is an idempotent such that $c_{n,F}(f_i')=\Char_{\geq i}$ and, for $i>1$, $f_i'\in f_{i-1}'\Des{n}{F}f_{i-1}'$.
  \item For $i\in [2,m+1]$, \[f_i'=1-\left (\frac{1}{\facmulti{\lambda^{(i-1)}}}\Xi^{\lambda^{(i-1)}}+\epsilon_{i-1}\right )\] for some $\epsilon_{i-1}\in Y_{\leq i-1}^\circ$.
\end{enumerate}
\end{lem}
\begin{proof} We first prove part (i). We argue by induction on $i$. It is clearly true when $i=1$ as $\Char_{\geq 1}=1$ and $f'_1=1$. Let $i>1$. Since $c_{n,F}$ is an $F$-algebra homomorphism, we have \[c_{n,F}(f_{i-1}'f_{\geq i}f'_{i-1})=\Char_{\geq i-1}\Char_{\geq i}\Char_{\geq i-1}=\Char_{\geq i}.\] Therefore, $c_{n,F}(f_i')=\Char_{\geq i}$. By Proposition \ref{P: Idem Lift}(i), $f_i'=(f_{i-1}'f_{\geq i}f'_{i-1})^{p^\infty}$ is an idempotent. Finally, since $f_{i-1}'$ is an idempotent, we have \[f_i'=f_{i-1}'(f_{i-1}'f_{\geq i}f'_{i-1})^{p^\infty}f_{i-1}'\in  f_{i-1}'\Des{n}{F}f_{i-1}'.\]

For part (ii), we again argue by induction on $i$. For $i=2$, notice that \[f_2'=(f_{\geq 2})^{p^\infty}=(1-f_1)^{p^\infty}=1-f_1^{p^\infty}=1-\left (\frac{1}{\facmulti{\lambda^{(1)}}}\Xi^{\lambda^{(1)}}+\epsilon_1\right ),\] for some $\epsilon_1\in Y_{\leq 1}^\circ$ where the second and last equalities are obtained using parts (iv) and (iii) of Lemma \ref{L: f*} respectively. Suppose inductively now that $f_i'=1+\epsilon_{i-1}'$ where $\epsilon_{i-1}'\in Y_{\leq i-1}$. Therefore, again using Lemma \ref{L: f*}(iv),
\begin{align*}
  f_{i+1}'&=(f_i'f_{\geq i+1}f_i')^{p^\infty}\\
  &=((1+\epsilon_{i-1}')(1-f_1-\cdots-f_i)(1+\epsilon_{i-1}'))^{p^\infty}\\
  &=((1+\epsilon_{i-1}')^2-(1+\epsilon_{i-1}')f_1(1+\epsilon_{i-1}')-\cdots-(1+\epsilon_{i-1}')f_i(1+\epsilon_{i-1}'))^{p^\infty}\\
  &=(1-f_i+E)^{p^\infty}\\
  &=(1-f_i)^{p^\infty}+E'\\
  &=1-f_i^{p^\infty}+E'\\
  &=1-\frac{1}{\facmulti{\lambda^{(i)}}}\Xi^{\lambda^{(i)}}-\epsilon_i
\end{align*} where $E$ is the sum of products involving at least one of the $f_1,\ldots,f_{i-1},\epsilon_{i-1}'$ and $E'$ is the sum of products involving at least one $E$. Since $f_1,\ldots,f_{i-1},\epsilon_{i-1}'\in Y_{\leq i-1}$, by Lemma \ref{L: Yi ideals}, $E\in Y_{\leq i-1}$ and consequently $E'\in Y_{\leq i-1}$. Lastly, by Lemmas \ref{L: f*}(iii) and \ref{L: Yi ideals}, we have $\epsilon_i=f_i^{p^\infty}-\frac{1}{\facmulti{\lambda^{(i)}}}\Xi^{\lambda^{(i)}}-E'\in Y_{\leq i}^\circ$.
\end{proof}

We are now ready to state and prove our main result in this section regarding the modular idempotents for the descent algebra $\Des{n}{F}$.

\begin{thm}\label{T: modular idem} The set $\{e_{\lambda,F}:\lambda\in\pReg(n)\}$ is a complete set of orthogonal primitive idempotents of $\Des{n}{F}$ such that $c_{n,F}(e_{\lambda,F})=\Char_{\lambda,F}$ and $\sum_{\lambda\in\pReg(n)}e_{\lambda,F}=1$.  Furthermore, \[e_{\lambda,F}=\frac{1}{\facmulti{\lambda}}\Xi^\lambda+\epsilon_\lambda\index{$e_{\lambda,F}$}\] where $\epsilon_\lambda$ is a linear combination of some $\Xi^\xi$ such that $\swref{\xi}{\lambda}$, i.e., $\lambda\not\approx\xi\wref \lambda$.
\end{thm}
\begin{proof} Using Proposition \ref{P: Idem Lift}(ii) and Lemma \ref{L: f'}(i), we get the first sentence of our theorem.  When $i=1$, by Lemma \ref{L: f'}(i), we have \[e_{\lambda^{(1)},F}=1-f_2'=\frac{1}{\facmulti{\lambda^{(1)}}}\Xi^{\lambda^{(1)}}+\epsilon_1\] where $\epsilon_{\lambda^{(1)}}=\epsilon_1\in Y_{\leq 1}^\circ=Y^\circ_{\lambda^{(1)}}$. For $i\in [2,m]$, again by Lemma \ref{L: f'}(ii), we have \[e_{\lambda^{(i)},F}=-\left (\frac{1}{\facmulti{\lambda^{(i-1)}}}\Xi^{\lambda^{(i-1)}}+\epsilon_{i-1}\right )+\left ( \frac{1}{\facmulti{\lambda^{(i)}}}\Xi^{\lambda^{(i)}}+\epsilon_i\right )=\frac{1}{\facmulti{\lambda^{(i)}}}\Xi^{\lambda^{(i)}}+\epsilon_{\lambda^{(i)}}\] where $\epsilon_{\lambda^{(i)}}=\epsilon_i-\frac{1}{\facmulti{\lambda^{(i-1)}}}\Xi^{\lambda^{(i-1)}}-\epsilon_{i-1}\in Y_{\leq i}^\circ$. To complete the proof, we claim that $e_{\lambda^{(i)},F}\in Y_{\lambda^{(i)}}$, so that $\epsilon_{\lambda^{(i)}}=e_{\lambda^{(i)},F}-\frac{1}{\facmulti{\lambda^{(i)}}}\Xi^{\lambda^{(i)}}\in Y_{\leq i}^\circ\cap Y_{\lambda^{(i)}}=Y^\circ_{\lambda^{(i)}}$. The case $i=1$ is done. Suppose that $i\in [2,m]$. We have
\begin{align*}
&f'_{i+1}=(f'_if_{\geq i+1}f'_i)^{p^\infty}\\
&=(f'_if_{\geq i}f'_i-f'_if_if'_i)^{p^\infty}\\
&=(f'_if_{\geq i}f'_i)^{p^\infty}+E\\
&=((f_{i-1}'f_{\geq i}f_{i-1}')^{p^\infty}f_{\geq i}(f_{i-1}'f_{\geq i}f_{i-1}')^{p^\infty})^{p^\infty}+E\\
&=((f_{i-1}'f_{\geq i}f_{i-1}')^{p^\infty}f_{i-1}'f_{\geq i}f_{i-1}'(f_{i-1}'f_{\geq i}f_{i-1}')^{p^\infty})^{p^\infty}+E\\
&=f_i'+E,
\end{align*} where $E$ is a sum of product involving the term $f_i'f_if_i'$ and the fifth equality is due to the fact that $f_{i-1}'$ is an idempotent. Since $f_i'f_if_i'\in Y_{\lambda^{(i)}}$ by Lemmas \ref{L: f*}(iii) and \ref{L: Yi ideals}, we have $E\in Y_{\lambda^{(i)}}$ and therefore $e_{\lambda^{(i)},F}=f_i'-f_{i+1}'\in Y_{\lambda^{(i)}}$.
\end{proof}

\begin{rem}\
\begin{enumerate}[(i)]
  \item The construction of the idempotents works for arbitrary field $F$ of characteristic $p$. In fact, the coefficients of the $\Xi^\xi$'s involved in $e_{\lambda,F}$ belong to the prime subfield $\mathbb{F}_p$ of $F$.
  \item The properties stated in the theorem do not guarantee uniqueness of the idempotents. Let $p=2$ and consider the idempotents $e_{(4),F},e_{(3,1),F}$ obtained using the construction (see Appendix \ref{Appen A}). Let $z=\Xi^{(2,1,1)}+\Xi^{(1,2,1)}$ and
      \begin{align*}
        e_{(4),F}'&=e_{(4),F}+z, & e_{(3,1),F}'&=e_{(3,1),F}+z.
      \end{align*} Notice that $z\in\ker(c_{4,F})$ and it is easily checked that $\{e_{(4),F}',e_{(3,1),F}'\}$ is again another complete set of orthogonal primitive idempotents of $\Des{4}{F}$ satisfying the conditions in Theorem \ref{T: modular idem}, that is $c_{n,F}(e_{\lambda,F}')=\Char_{\lambda,F}$ and $e_{\lambda,F}'-\frac{1}{\facmulti{\lambda}}\Xi^\lambda\in\mathrm{span}_F\{\Xi^\xi:\swref{\xi}{\lambda}\}$.
\end{enumerate}
\end{rem}


We demonstrate our construction of the idempotents with an example.

\begin{eg} Let $p=3$, $n=4$ and $(2,1,1)<(2,2)< (3,1)<(4)$ be the total order on $\P_3(4)$. It can be easily checked that we have

\[\Phi^F=\begin{pmatrix}
  2&0&0&0\\ 2&2&0&0\\ 2&0&1&0\\ 1&1&1&1
\end{pmatrix},\ \ \Psi^F=(\Phi^F)^{-1}=\begin{pmatrix}
  2&0&0&0\\ 1&2&0&0\\ 2&0&1&0\\ 1&1&2&1
\end{pmatrix}.\] Therefore we have
\begin{align*}
  f_1&=2\Xi^{(2,1,1)},& f_2&=2\Xi^{(2,2)}+\Xi^{(2,1,1)},\\ f_3&=\Xi^{(3,1)}+2\Xi^{(2,1,1)},& f_4&=\Xi^{(4)}+2\Xi^{(3,1)}+\Xi^{(2,2)}+\Xi^{(2,1,1)}.
\end{align*} As we mentioned earlier, it suffices to take power $p^{k_n}\geq n-1$ (see the paragraph right after Proposition \ref{P: Idem Lift}). In this case, it is enough to take $k_n=1$. We get $f_1'=1$,
\begin{align*}
  f_2'&=(\Xi^{(4)}+\Xi^{(2,1,1)})^3=1-(2\Xi^{(2,1,1)}+2\Xi^{(1,1,1,1)})\\
  f_3'&=(f_2'(\Xi^{(4)}+\Xi^{(2,2)})f_2')^3=1-(2\Xi^{(2,2)}+\Xi^{(2,1,1)}+2\Xi^{(1,1,2)}+\Xi^{(1,1,1,1)})\\
  f_4'&=(f_3'f_4f_3')^3=1-(\Xi^{(3,1)}+2\Xi^{(2,2)}+2\Xi^{(1,1,2)}+\Xi^{(1,1,1,1)}),
\end{align*} and $f_5'=0$. Therefore,
\begin{align*}
  e_{(2,1,1),F}&=2\Xi^{(2,1,1)}+2\Xi^{(1,1,1,1)},&
  e_{(2,2),F}&=2\Xi^{(2,2)}+2\Xi^{(2,1,1)}+2\Xi^{(1,1,2)}+2\Xi^{(1,1,1,1)},\\
  e_{(3,1),F}&=\Xi^{(3,1)}+2\Xi^{(2,1,1)},&
  e_{(4),F}&=\Xi^{(4)}+2\Xi^{(3,1)}+\Xi^{(2,2)}+\Xi^{(1,1,2)}+2\Xi^{(1,1,1,1)}.
\end{align*}
\end{eg}


We have presented the computation results of some modular idempotents in Appendices \ref{Appen A} and \ref{Appen B} using Magma \cite{Magma}. We now state a few corollaries following our theorem. 

\begin{cor}\label{C: e * omega} Let $\lambda\in\pReg(n)$ and $q\in \C(n)$. Then $e_{\lambda,F}\omega_q=0$ unless $q\wref \lambda$. When $\lambda=\lambda(q)$, we have $e_{\lambda,F}\omega_q=\omega_\lambda$.
\end{cor}
\begin{proof} By Theorem \ref{T: BL results}(ii), $\Xi^s\omega_q=0$ unless $q\wref s$. Therefore, by Theorem \ref{T: modular idem}, we get the first assertion. For the second assertion, similarly, using Theorem \ref{T: modular idem}, we have \[e_{\lambda,F}\omega_q=\frac{1}{\facmulti{\lambda}}\Xi^\lambda\omega_q+\epsilon_\lambda\omega_q =\frac{1}{\facmulti{\lambda}}\Xi^\lambda\omega_q=\frac{1}{\facmulti{\lambda}}|\overline{N_{\lambda,q}^\lambda}|\omega_\lambda =\frac{1}{\facmulti{\lambda}}\facmulti{\lambda}\omega_\lambda=\omega_\lambda\]where the third and fourth equalities follow from Theorem \ref{T: BL results}(ii) and Lemma \ref{L: basic coef approx} respectively.
\end{proof}

By Lemma \ref{L: translate w}, if $\lambda\approx q$, there exists $\sigma$ such that $\sigma\omega_q=\omega_\lambda$. Therefore, Corollary \ref{C: e * omega} implies the following result.

\begin{cor}\label{C: higher lie inj in proj} Let $q\in\C(n)$ such that $q\approx \lambda\in \pReg(n)$. The map $\alpha:\Lie_F(q)\to e_{\lambda,F} F\sym{n}$ defined by the left multiplication by $e_{\lambda,F}$ is an injection.
\end{cor}

As we have mentioned earlier, in the $p=0$ case, we also could construct a list of orthogonal primitive idempotents for $\Des{n}{F}$ using the method in this section by replacing the idempotent lifting $a^{p^\infty}$ with performing the ($3a^2-2a^3$)-process repeatedly. To conclude the section, we give a simple example and compare it with the orthogonalization of the Lie idempotents $\nu_\lambda$'s as in Proposition \ref{P: ord idem}.

\begin{eg} Let $n=3$, $p=0$ (or $p>3$) and $(1^3)<(2,1)<(3)$. We have \[\Phi^0=\begin{pmatrix}
  6&0&0\\ 3&1&0\\ 1&1&1
\end{pmatrix},\ \ \Psi^0=\begin{pmatrix}
  \frac{1}{6}&0&0\\
  -\frac{1}{2}&1&0\\
  \frac{1}{3}&-1&1
\end{pmatrix}.\] Therefore,
\begin{align*}
  f_3&=\frac{1}{6}\Xi^{(1^3)},& f_2&=\Xi^{(2,1)}-\frac{1}{2}\Xi^{(1^3)},& f_1&=\Xi^{(3)}-\Xi^{(2,1)}+\frac{1}{3}\Xi^{(1^3)}.
\end{align*} It is readily checked that $\{f_1,f_2,f_3\}$ is already a complete set of orthogonal primitive idempotents and therefore it is the orthogonalization of itself; namely, $e_{(3)}=f_1$, $e_{(2,1)}=f_2$ and $e_{(1,1,1)}=f_3$.

On the other hand, if we orthogonalize $\{\nu_{(3)},\nu_{(2,1)},\nu_{(1^3)}\}$ as in Proposition \ref{P: ord idem}, where
\begin{align*}
  \nu_{(3)}&=\Xi^{(3)}-\frac{1}{3}\Xi^{(2,1)}-\frac{2}{3}\Xi^{(1,2)}+\frac{1}{3}\Xi^{(1^3)}, & \nu_{(2,1)}&=\Xi^{(2,1)}-\frac{1}{2}\Xi^{(1^3)},& \nu_{(1^3)}&=\frac{1}{6}\Xi^{(1^3)},
\end{align*} we get
\begin{align*}
e_{(3)}'&=\nu_{(3)}=\Xi^{(3)}-\frac{1}{3}\Xi^{(2,1)}-\frac{2}{3}\Xi^{(1,2)}+\frac{1}{3}\Xi^{(1^3)},\\
e_{(2,1)}'&=\nu_{(2,1)}(1-\nu_{(3)})=\Xi^{(2,1)}+\frac{1}{2}\Xi^{(1^3)},\\
e_{(1,1,1)}'&=\nu_{(1,1,1)}(1-\nu_{(2,1)})(1-\nu_{(3)})=\frac{1}{6}\Xi^{(1^3)}.
\end{align*} The set $\{e_\lambda:\lambda\in\P(3)\}$ appears to be simpler than $\{e'_\lambda:\lambda\in\P(3)\}$ in terms of the $\Xi^q$'s. But, when $p=0$, in general, the computation requires the computation of the inverse matrix $\Psi^0$ which is not obvious at all.  Indeed, in this case, $KC=\Phi^0$ where $K$ and $C$ are the Kostka matrix (see, for example, \cite{Mac}) and character table for the symmetric group respectively. Therefore the knowledge of the inverse of $\Phi^0$ would give us knowledge about the inverse of the Kostka matrix or character table because $K^{-1}=C\Psi^0$ and $C^{-1}=\Psi^0K$.
\end{eg}

\section{Higher Lie Modules}\label{S: higher lie}

In this section, we study the higher Lie module $\Lie_F(q)$ when $q$ is coprime to $p$. Our main results in this section are Theorems \ref{T: Lie mod basis} and \ref{T: higher lie mod} which basically establish a basis and describe the structure for $\Lie_F(q)$ when $(q,p)=1$. Both results work for arbitrary $p$. The extra assumption of $p>0$ only applies after Corollary \ref{C: higher lie mod split}. We begin with some notations.

Let $\lambda\in\P(n)$ and $k=\ell(\lambda)$. Recall the notations $m_i(\lambda)$ and $m(\lambda)$ in Subsection \ref{SS: generality}. Let $V$ be a finite dimensional vector space over $F$ and, for each $i\in [1,k]$, let $P_i\in\omega_{\lambda_i}\cdot T^{\lambda_i}(V)=L^{\lambda_i}(V)$. We write
\begin{align*}
  (P_1,\ldots,P_k)_\lambda&=\sum_{\pi\in \sym{m(\lambda)}}P_{1\pi} \cdots P_{k\pi}=\Xi^\lambda\cdot P_1\cdots P_k\index{${(P_1,\ldots,P_k)}_{\lambda}$}
\end{align*} where  the last equality follows from Theorem \ref{T: GR 2.1}. To describe a basis for the higher Lie module, we need the following notation.

\begin{defn}\label{D: Gamma} Let $\Gamma(\lambda)$ \index{$\Gamma(\lambda)$}be the set consisting of sequences \[(T_{n,1},\ldots,T_{n,m_n(\lambda)},\ldots,T_{1,1},\ldots,T_{1,m_1(\lambda)})\] such that
\begin{enumerate}[(a)]
\item $T_{i,j}$ are sequences of positive integers of length $i$ such that $\bigcup_{i,j}\{(T_{i,j})_k:k\in[1,i]\}=[1,n]$,
\item $\min(T_{i,j})=(T_{i,j})_1$ for all admissible $i,j$, and
\item $\min(T_{i,1})<\min(T_{i,2})<\cdots< \min(T_{i,m_i(\lambda)})$ for all $i\in [1,n]$.
\end{enumerate}
\end{defn}

Notice that $|\Gamma(\lambda)|=|\ccl{\lambda}|=\frac{n!}{\lambda?}$ because $\Gamma(\lambda)$ is in one-to-one correspondence with the set of permutations of cycle type $\lambda$ in $\sym{n}$ and also the set of standard right coset representatives $\prod_{i=1}^n(\sym{i}\wr\sym{m_i})\backslash \sym{n}$. The second correspondence is obtained by reading, from left to right and top to bottom, the entries of $\lambda$-tableaux $\mathfrak{t}$ such that $\mathfrak{t}_{1,j}<\mathfrak{t}_{1,j+1}$ if $\lambda_j=\lambda_{j+1}$.

Recall the notation $Q_S$ in Subsection \ref{SS: descent}. For each $T\in \Gamma(\lambda)$ with the presentation as before, we write
\begin{align*}
  Q_T&=Q_{T_{n,1}}\cdots Q_{T_{n,m_n(\lambda)}}\cdots Q_{T_{11}}\cdots Q_{T_{1,m_n(\lambda)}},\index{$Q_T$}\\
  (Q_T)_\lambda&=(Q_{T_{n,1}},\ldots, Q_{T_{n,m_n(\lambda)}},\ldots, Q_{T_{11}},\ldots, Q_{T_{1,m_n(\lambda)}})_\lambda.\index{${(Q_T)}_{\lambda}$}
\end{align*} They are viewed as elements in $R\sym{n}$. For example, \[Q_{((1,2),(3,4))}=Q_{(1,2)}Q_{(3,4)}=1234-2134-1243+2143.\] Similar as before, we have $(Q_T)_\lambda=\Xi^\lambda Q_T$. Notice that $\omega_i \cdot Q_{T_{i,j}}=iQ_{T_{i,j}}$ since $\omega_i^2=i\omega_i$. Therefore, $\omega^\lambda  Q_T=\prod_{i=1}^ni^{m_i(\lambda)}Q_T$ and likewise for $(Q_T)_\lambda$. With respect to $\lambda$, for each pair of admissible numbers $i,j$, denote $d_{i,j}=i(j-1)+\sum_{t\in [i+1,n]}tm_t(\lambda)$ and call the subword \[w_{1+d_{i,j}}w_{2+d_{i,j}}\cdots w _{i+d_{i,j}}\] the $(i,j)$-section of a word $w\in\sym{n}$.

Our first result in this section establishes a basis for the higher Lie module $\Lie_F(q)$ when $q$ is coprime to $p$ (cf. Appendix \ref{Appen C}). We note that, in the characteristic zero case, a basis for a higher Lie module, up to isomorphism, can also be derived, for example, using the proof of \cite[Theorem 8.23]{Reu} or \cite[Lemma 2.1 and Theorem 2.2]{Schoc03}. 


\begin{thm}\label{T: Lie mod basis} Let $q\vDash n$, $(q,p)=1$ and $\lambda=\lambda(q)$. The set $\{\omega_q Q_T:T\in\Gamma(\lambda)\}$ is a basis for $\Lie_F(q)$. In particular, $\dim_F\Lie_F(q)=|\ccl{\lambda}|=\frac{n!}{q?}$.
\end{thm}
\begin{proof} By Lemma \ref{L: translate w}, we may assume that $q=\lambda(q)=\lambda\in\P(n)$. By definition, \[\omega_\lambda Q_T=\omega^\lambda\Xi^\lambda Q_T=\omega^\lambda(Q_T)_\lambda=\left (\prod_{i=1}^ni^{m_i(\lambda)}\right )(Q_T)_\lambda.\] Since $\lambda$ is coprime to $p$, we have $\left (\prod_{i=1}^ni^{m_i(\lambda)}\right )\neq 0$. By Proposition \ref{P: dim higher lie 0}, it suffices to show that the set $\{(Q_T)_\lambda:T\in\Gamma(\lambda)\}$ is linearly independent.

Notice that, for a word $w$ appearing in $(Q_T)_\lambda$, by the definition, the alphabets appearing in $(i,j)$-section of $w$ must be the set $\{t:t\in T_{i,j'}\}$ for some $j'\in [1,m_i(\lambda)]$. Fix an arbitrary $T\in\Gamma(\lambda)$. Therefore, it suffices to consider \[\sum  a_S(Q_S)_\lambda=0\] where the sum is taken over all $S\in\Gamma(\lambda)$ such that $\{s:s\in S_{ij}\}=\{t:t\in T_{ij}\}$ for all admissible $i,j$. Fix an $S\in \Gamma(\lambda)$ satisfying such property. We want to show that $a_{S}=0$. For each admissible $i,j$, suppose that $S_{ij}=(u_1,\ldots,u_i)$. It follows from the definition of $\omega_i$ (see Subsection \ref{SS: descent}) that $u_1\ldots u_i$ is the unique word involved in the summand of $Q_{S_{ij}}=\omega_i\cdot (u_1\ldots u_i)$ such that its first alphabet is $u_1$. Notice that its coefficient is 1. Let $w^{(i,j)}$ denote the $(i,j)$-section of a word $w$. Let $w_S$ be the unique word $w$ involved in $(Q_S)_\lambda$ where, for each $i,j$, $w^{(i,j)}_1$ is the smallest among all alphabets involved in $w^{(i,j)}$ and, for each $i$, we have $w^{(i,1)}_1<w^{(i,2)}_1<\cdots<w^{(i,m_i(\lambda))}_1$. By Definition \ref{D: Gamma}, the word $w_S$ is not involved in $(Q_T)_\lambda$ if $T\neq S$. As such, $a_S=0$. The proof is now complete.
\end{proof}

We could now prove the isomorphism \[\Lie_F(\lambda)\cong  \Ind^{\sym{n}}_{\prod_{i=1}^n(\sym{i}\wr \sym{m_i})}(\Lie_F(1)^{\wr m_1}\boxtimes\cdots\boxtimes \Lie_F(n)^{\wr m_n})\] where $m_i=m_i(\lambda)$ as in Theorem \ref{T: higher lie mod} directly. But we could also prove this using the polynomial representation of $\GL_n$ and applying Schur functor. By doing this, it reveals that, in general, the higher Lie power $L^q(V)$ when $(q,p)=1$ is a quotient of the tensor product of some symmetric powers, which we will be using to study $L^q(V)$ in Section \ref{S: ch and dim}.

In the next lemma, we denote the $n$th symmetric power of a vector space $V$ by $S^n(V)$\index{$S^n(V)$}, i.e., $S^n(V)\cong F\otimes_{F\sym{n}} T^n(V)$ which has a basis represented by monomials $v_{i_1}\ldots v_{i_n}$ such that $1\leq i_1\leq \cdots\leq i_n\leq m$ given that $\{v_1,\ldots,v_m\}$ is a basis for $V$.

\begin{lem}\label{L: Sym surj Lie} Let $\lambda\in\P(n)$, $(\lambda,p)=1$ and $m_i=m_i(\lambda)$ for each $i\in [1,n]$. There is a surjection $\psi:S^{m_n}(L^n(V))\otimes \cdots\otimes S^{m_1}(L^1(V))\to \omega_\lambda\cdot T^n(V)=L^\lambda(V)$ given by \[\psi((P_{n,1}\cdots P_{n,m_n})\otimes \cdots\otimes (P_{1,1}\cdots P_{1,m_1}))=\omega^\lambda\cdot (P_{n,1},\ldots,P_{n,m_n},\ldots,P_{1,1},\ldots,P_{1,m_1})_\lambda\] where $P_{i,j}\in L^i(V)$ for each admissible $i,j$. Furthermore, if $V$ is a right $FG$-module then $\psi$ is an $FG$-module homomorphism.
\end{lem}
\begin{proof} Notice that the image is fixed by the action of $\sym{m(\lambda)}$ on the left.  Also, as we have noted earlier, \[\omega^\lambda\cdot (P_{n,1},\ldots,P_{1,m_1})_\lambda=(\omega^\lambda\Xi^\lambda) \cdot P_{n,1}\cdots P_{1,m_1}=\omega_\lambda\cdot P_{n,1}\cdots P_{1,m_1}.\] So $\psi$ is well-defined.  Let $k=\ell(\lambda)$ and $N=\prod_{i=1}^k\lambda_i$ so that $N\neq 0$ in $F$ by our assumption. Notice that
\begin{align*}
  \Xi^\lambda=\sum_{\substack{S_1\sqcup \cdots\sqcup S_k=[1,n],\\ |S_i|=\lambda_i}}S_1\cdots S_k&=\sum_{\pi\in\sym{m(\lambda)}}\sum_{(S_1,\ldots,S_k)\in\Theta} S_{1\pi}\cdots S_{k\pi}
\end{align*} where $\Theta$ consists of $(S_1,\ldots,S_k)$ such that $S_1\sqcup \cdots\sqcup S_k=[1,n]$, $|S_i|=\lambda_i$ and, if $i<i'$ and $|S_i|=|S_{i'}|$, we have $\min(S_i)<\min(S_{i'})$. Therefore,
\begin{align*}
  \omega_\lambda&=\omega^\lambda\Xi^\lambda=\sum_{\pi\in\sym{m(\lambda)}}\sum_{(S_1,\ldots,S_k)\in\Theta} (\omega_{\lambda_1}S_{1\pi})\cdots (\omega_{\lambda_k}S_{k\pi}).
\end{align*} For any $u_1,\ldots,u_n\in V$, we have
\begin{align*}
\omega_\lambda\cdot (u_1\otimes\cdots\otimes u_n)&=\sum_{\pi\in\sym{m(\lambda)}}\sum_{(S_1,\ldots,S_k)\in\Theta} (\omega_{\lambda_1}S_{1\pi})\cdots (\omega_{\lambda_k}S_{k\pi})\cdot (u_1\otimes\cdots\otimes u_n)\\
&=\sum_{\pi\in\sym{m(\lambda)}}\sum_{(S_1,\ldots,S_k)\in\Theta} P_{S_{1\pi}}\cdots P_{S_{k\pi}}\\
&=\sum_{(S_1,\ldots,S_k)\in\Theta} (P_{S_1},\ldots, P_{S_k})_\lambda\\
&=\sum_{(S_1,\ldots,S_k)\in\Theta} \frac{1}{N}\omega^\lambda\cdot (P_{S_1},\ldots, P_{S_k})_\lambda\\
&=\sum_{(S_1,\ldots,S_k)\in\Theta} \frac{1}{N}\psi((P_{S_1}\cdots P_{S_{m_n}})\otimes\cdots\otimes (P_{S_{k-m_1+1}}\cdots P_{S_k}))
\end{align*} where, for each $j\in [1,k]$, $P_{S_j}=\omega_{\lambda_j}\cdot (u_{s_1}\otimes \cdots\otimes u_{s_{\lambda_j}})$ if $S_j=\{s_1<\cdots<s_{\lambda_j}\}$.
Therefore $\psi$ is surjective. The fact that $\psi$ is an $FG$-module homomorphism follows from the $FG$-module structures of tensor product and symmetric power.
\end{proof}

In general, the map $\psi$ in Lemma \ref{L: Sym surj Lie} is not an injection. For example, let $p=2$, $V$ be 2-dimensional and $\lambda=(1,1)$. Then $S^2(V)$ is 3-dimensional but $\omega_{(1,1)}\cdot T^2(V)$ is 1-dimensional. However, we will see in Section \ref{S: ch and dim} that, with the extra assumption that $\lambda(q)$ is $p$-regular, we get injectivity.

\smallskip
We can now state and prove the second main result for this section.

\begin{thm}\label{T: higher lie mod} Let $q\in\C(n)$, $m_i=m_i(q)$ and suppose that $(q,p)=1$. We have an isomorphism of $F\sym{n}$-modules
\begin{align*}
\Lie_F(q)&\cong \Ind^{\sym{n}}_{\prod_{i=1}^n(\sym{i}\wr \sym{m_i})}(\Lie_F(1)^{\wr m_1}\boxtimes\cdots\boxtimes \Lie_F(n)^{\wr m_n}).
\end{align*} Suppose further that $F$ contains, for each $m_i>0$, a primitive $i$th root of unity. For each $p\nmid i$, there is a certain one-dimensional $FC_i$-module $F_{\delta_i}$ such that \[\Lie_F(q)\cong \Ind^{\sym{n}}_{\prod^n_{i=1}(C_i\wr \sym{m_i})}F_\delta\] where $F_\delta=(F_{\delta_1}^{\wr m_1})\boxtimes\cdots\boxtimes (F_{\delta_n}^{\wr m_n})$ as a module for $\prod^n_{i=1}(C_i\wr \sym{m_i})$.
\end{thm}
\begin{proof} Let $\P(n)\ni\lambda\approx q$, $H=\prod_{i=1}^n(\sym{i}\wr\sym{m_i})$ and $Z_F$ denote the induced module $\Ind^{\sym{n}}_H(\Lie_F(1)^{\wr m_1}\boxtimes\cdots\boxtimes \Lie_F(n)^{\wr m_n})$ defined over $F$.  
Consider the natural $F\GL(V)$-module $V$ and assume that $\dim_FV\geq n$.  By \cite[\S2.5 Lemma]{DE} and \cite[Corollary 3.2(i)]{LT12}, the Schur functor $f$ maps the module $S^{m_n}(L^n(V))\otimes \cdots\otimes S^{m_1}(L^1(V))$ isomorphically to $Z_F$. Since $f$ is exact, by Lemma \ref{L: Sym surj Lie}, we have a surjection from $Z_F$ to $\Lie_F(\lambda)$. Since $\lambda$ is coprime to $p$, by Theorem \ref{T: Lie mod basis}, $\dim_F\Lie_F(\lambda)=\frac{n!}{\lambda?}$. On the other hand, using Proposition \ref{P: Lie(n) basis}, we have
\begin{align*}
\ \dim_F(Z_F)=&\ \frac{n!}{\prod_{i=1}^n(i!)^{m_i}m_i!}\prod_{i=1}^n((i-1)!)^{m_i}=\frac{n!}{\prod_{i=1}^ni^{m_i}m_i!}=\frac{n!}{\lambda?}.
\end{align*} Comparing the dimensions, the surjection is indeed an isomorphism.

Alternatively, the isomorphism can also be proved by an explicit map which we shall now demonstrate. Let $W$ be the $F$-linear span of \[\{\omega_\lambda Q_T:T\in\Gamma(\lambda),\ \{t:t\in T_{i,j}\}=[d_{i,j}+1,d_{i,j}+i]\}.\] As we pointed out, $\omega_\lambda Q_T=\omega^\lambda (Q_T)_\lambda$. By the characterization of the induced module as in \cite[Section 8, Corollary 3]{Alperin}, the fact that $W$ generates $\Lie_F(\lambda)$ using Theorem \ref{T: Lie mod basis} and our previous observation that $\dim_F\Lie_F(\lambda)=|\sym{n}:H|\dim_FW$, it suffices to check that $\Lie_F(n)^{\wr m_n}\boxtimes \cdots\boxtimes \Lie_F(1)^{\wr m_1}$ is isomorphic to $W$ as $FH'$-modules where $H'=(\sym{n}\wr\sym{m_n})\times\cdots\times (\sym{1}\wr\sym{m_1})$. Since it is a direct product, we may further assume that $\lambda=(d^k)$ and $H'=\sym{d}\wr\sym{k}$.  We claim that the linear map $\phi$ sending $\omega_d\sigma_1\otimes\cdots\otimes \omega_d\sigma_k$, where, for all $j\in [1,k]$, $\sigma_j\in\sym{d}$ such that $1\sigma_j=1$, to $\omega_\lambda Q_T$ is an isomorphism, where \[Q_T=(\omega_d\sigma_1)(\omega_d\sigma_2)^{+d} \cdots(\omega_d\sigma_k)^{+(k-1)d}\] and $(\omega_d\sigma_j)^{+(j-1)d}$ means translation of the alphabets of $\omega_d\sigma_j$ by $(j-1)d$. Let $\tau$ be an element in the top group of $\sym{d}\wr\sym{k}$. We have
\begin{align*}
  \omega_\lambda Q_T\tau=\omega^\lambda(Q_T)_\lambda\tau
  &=\omega^\lambda ((\omega_d\sigma_1)^{+(1\tau-1)d},\ldots,(\omega_d\sigma_k)^{(k\tau-1)d})_\lambda\\
  &=\omega^\lambda (\omega_d\sigma_{1\tau^{-1}},(\omega_d\sigma_{2\tau^{-1}})^{+d},\ldots,(\omega_d\sigma_{k\tau^{-1}})^{(k-1)d})_\lambda\\
  &=\omega_\lambda \omega_d\sigma_{1\tau^{-1}}(\omega_d\sigma_{2\tau^{-1}})^{+d}\cdots (\omega_d\sigma_{k\tau^{-1}})^{(k-1)d}\\
  &=\phi(\omega_d\sigma_{1\tau^{-1}}\otimes\cdots\otimes \omega_d\sigma_{k\tau^{-1}}).
\end{align*} For the action of the base group of $\sym{d}\wr\sym{k}$, it can be easily checked that it commutes with $\phi$. Therefore, $\phi$ is an isomorphism of $FH'$-modules.

We now prove the second statement. For $p\nmid i$, we have $\Lie_F(i)\cong \Ind_{C_i}^{\sym{i}}F_{\delta_i}$ where $C_i$ is generated by a cyclic permutation of length $i$ in $\sym{i}$, $F_{\delta_i}$ is one-dimensional and $\delta_i$ is a primitive $i$th root of unity (see, for example, \cite[Theorem 8.24]{Reu} for the ordinary case and \cite[Lemma 3.1]{LT16} for the modular case). By \cite[Lemma 2.6]{ELT} (with $|I|=1$ in that section), we have \[\Lie_F(i)^{\wr m_i}\cong \Ind^{\sym{i}\wr \sym{m_i}}_{C_i\wr \sym{m_i}} F_{\delta_i}^{\wr m_i}.\] Therefore, we obtain the second isomorphism where $F_\delta=(F_{\delta_1}^{\wr m_1})\boxtimes\cdots\boxtimes (F_{\delta_n}^{\wr m_n})$. The proof is now complete since $\Lie_F(q)\cong \Lie_F(\lambda)$.
\end{proof}


We now establish some corollaries following our results. The following is immediate.

\begin{cor}\label{C: higher lie mod split} If $q\approx r\cont s$ such that, for each $i\in\NN$, we have either $m_i(q)=m_i(r)$ or $m_i(q)=m_i(s)$, then \[
\Lie_F(q)\cong \ind^{\sym{m+n}}_{\sym{m}\times\sym{n}}(\Lie_F(r)\boxtimes \Lie_F(s))
\] where $|r|=m$ and $|s|=n$.
\end{cor}

For the rest of this section, we assume further that $p>0$.

\begin{cor}\label{C: Lie trivial source} Let $q\in\C(n)$, $(q,p)=1$ and $m_i=m_i(q)$. Suppose that $F$ contains, for each $m_i>0$, a primitive $i$th root of unity. Then $\Lie_F(q)$ is a $p$-permutation module. Furthermore, any indecomposable summand of $\Lie_F(q)$ is a trivial source module and has a vertex a $p$-subgroup of $T=T_{m_1}\times \cdots\times T_{m_n}$ where $T_{m_i}$ is the top group of $C_i\wr \sym{m_i}$. 
\end{cor}
\begin{proof} The fact that $\Lie_F(q)$ is a $p$-permutation module follows from Theorem \ref{T: higher lie mod} and Theorem \ref{T: Broue}. Let $N$ be an indecomposable summand of $\Lie_F(q)$. By definition, $N$ has a vertex $Q$ a $p$-subgroup of $\prod^n_{i=1}(C_i\wr \sym{m_i})$. Since $m_i=0$ if $p\mid i$, we have that $Q$ is conjugate to a subgroup of $T$. Since $\Lie_F(q)$ is induced from a one-dimensional module, we have that $N$ is a trivial source module.
\end{proof}

\begin{cor}\label{C: ord char of Lie} Let $q\in\C(n)$, $(q,p)=1$ and $(\bk,\OO,\sk)$ be a $p$-modular system. The multiplicity of the irreducible character $\zeta^\mu_{\bk}$ in the ordinary character of $\Lie_{\sk}(q)$ is the number $C^\mu_q$ as given in Theorem \ref{T: Schoc Lie(q)}.
\end{cor}
\begin{proof} The module $\Lie_{\OO}(q)$ clearly is the unique lift of $\Lie_{\sk}(q)$. As such, the ordinary character of $\Lie_{\bk}(q)$ is the ordinary character of $\Lie_{\bk}(q)=\nu_q\bk \sym{n}$, which has been described as in Theorem \ref{T: Schoc Lie(q)}.
\end{proof}


Suppose that $F$ is algebraically closed. There are certain cohomology invariants of the $FG$-modules called the support variety and complexity (see \cite{AE81,AE82}). For instance, a module is projective if and only if its complexity is 0. A module is non-projective periodic if and only if it has complexity one. In the case for the Lie modules,  $\Lie_F(n)$ has complexity $c$ where $c\in\NN_0$ is the largest such that $p^c\mid n$ (see \cite{CHN,ELT}). In particular, $\Lie_F(n)$ is projective if and only if $p\nmid n$ (see \cite{DE}) and $\Lie_F(n)$ is non-projective periodic if and only if $p\mid n$ and $p^2\nmid n$. For the notation we use in the next corollary, we refer the reader to \cite[\S5.7]{Ben2}.


\begin{cor}\label{C: Complexity Lie} Suppose that $F$ is algebraically closed, $q\in\C(n)$, $(q,p)=1$, $m_i=m_i(q)$ and $P$ be a Sylow $p$-subgroup of $\prod^n_{i=1}(C_i\wr \sym{m_i})$. The support variety of $\Lie_F(q)$ is $\Res^*_{\sym{n},P}V_{P}(F)$ and its complexity is $\sum_{i=1}^n\lfloor \frac{m_i}{p}\rfloor$ where $\lfloor-\rfloor$ \index{$\lfloor-\rfloor$}denotes the floor function. In particular,
\begin{enumerate}[(i)]
  \item $\Lie_F(q)$ is projective if and only if $m_i<0$ for all $i\in [1,n]$, i.e., $\lambda(q)\in\pReg(n)$, and
  \item $\Lie_F(q)$ is non-projective periodic if and only if $p\leq m_j<2p$ for some unique $j\in [1,n]$ and $m_i<p$ for $i\neq j$.
\end{enumerate}
\end{cor}
\begin{proof} By \cite[Proposition 5.7.5]{Ben2} and Corollary \ref{C: Lie trivial source}, the support variety of $\Lie_F(q)$ is equal 
to $\Res^*_{\sym{n},P}V_{P}(F)$ (see, for example, \cite[Lemma 6.2]{DL}). Again, by \cite[Lemma 6.2]{DL}, the complexity of $\Lie_F(q)$ is the $p$-rank of the group $\prod^n_{i=1}(C_i\wr \sym{m_i})$ which is $\sum_{i=1}^n\lfloor \frac{m_i}{p}\rfloor$ since $m_i=0$ if $p\mid i$. As such, the complexity is 0 if and only if $\lfloor\frac{m_i}{p}\rfloor=0$ for all $i$ and is 1 if and only if $\lfloor\frac{m_j}{p}\rfloor=1$ for precisely one such $j$ and $\lfloor\frac{m_i}{p}\rfloor=0$ for $i\neq j$.
\end{proof}

Recall the set of orthogonal primitive idempotents $\{e_{\lambda,F}:\lambda\in\pReg(n)\}$ of $\Des{n}{F}$ we have constructed in Section \ref{S: mod idem}. It follows that the regular module $F\sym{n}$ is isomorphic to $\bigoplus_{\lambda\in\pReg(n)}e_{\lambda,F}F\sym{n}$. It has been proved in \cite{ES} that,
\begin{equation}\label{Eq: proj dim eFS}
  \dim_Fe_{\lambda,F}F\sym{n}=|\ccl{\lambda,p}|.
\end{equation}
 We now study the relation between the higher Lie modules and such projective modules.

Since, by Corollary \ref{C: Complexity Lie}, $\Lie_F(q)$ is not projective in general, the map $\alpha$ in Corollary \ref{C: higher lie inj in proj} may not split. For a simple example, consider $p=2$ and $q=(2)=\lambda$. Clearly, $\Lie_F((2))\cong F$. On the other hand, since $e_{(2),F}=\Xi^{(2)}=1$, we have that $e_{(2),F}F\sym{2}$ is the regular module.

However, when $\lambda$ is both $p$-regular and coprime to $p$, we get isomorphism as shown in the next corollary.  




\begin{cor}\label{C: Lie iso proj} Let $q\in\C(n)$, $(q,p)=1$ and $\lambda=\lambda(q)\in\pReg(n)$. We have an isomorphism $\Lie_F(q)\cong e_{\lambda,F}F\sym{n}$.
\end{cor}
\begin{proof} By Corollary \ref{C: higher lie inj in proj}, we only need to check their dimensions. By Equation \ref{Eq: proj dim eFS}, Lemma \ref{L: p-equiv conj} and Theorem \ref{T: Lie mod basis}, we have \[\dim_F  e_{\lambda,F}F\sym{n}=|\ccl{\lambda,p}|=|\ccl{\lambda}|=\dim_F \Lie_F(q).\]
\end{proof}


To conclude this section, we give an example to illustrate how the modular twisted Foulkes modules are related to the higher Lie modules.

\begin{eg} Let $p\neq 2$ and $q=(2^a,1^b)$. Then $(q,p)=1$. By Theorem \ref{T: higher lie mod}, \[\Lie_F((2^a,1^b))\cong \Ind^{\sym{2a+b}}_{(\sym{1}\wr \sym{b})\times (\sym{2}\wr \sym{a})} (\Lie_F(1)^{\wr b}\boxtimes \Lie_F(2)^{\wr a}).\] Notice that $\Lie_F(1)^{\wr b}=F$ is the trivial $F\sym{b}$-module with the identification $S_1\wr S_b=S_b$. On the other hand, $\Lie_F(2)=\sgn(2)$ the signature representation for $F\sym{2}$-module. Observe that $\sgn(2)^{\wr a}=\Res^{\sym{2a}}_{\sym{2}\wr \sym{a}}\sgn(2a)$ as $F[\sym{2}\wr \sym{a}]$-modules and
\begin{align*}
  (\Ind^{\sym{2a}}_{\sym{2}\wr \sym{a}}\Lie_F(2)^{\wr a})\otimes \sgn(2a)&\cong \Ind^{\sym{2a}}_{\sym{2}\wr \sym{a}}(\sgn(2)^{\wr a}\otimes (\Res^{\sym{2a}}_{\sym{2}\wr \sym{a}}\sgn(2a))\\
  &=\Ind^{\sym{2a}}_{\sym{2}\wr \sym{a}}F=H^{(2^a)}
\end{align*} which $H^{(2^a)}$ \index{$H^{(2^a)}$}is the Foulkes module (see \cite{Foul}) for $F\sym{2a}$ in the modular case. Therefore we obtain \[\Lie_F((2^a,1^b))\otimes \sgn(2a+b)\cong \Ind^{\sym{2a+b}}_{\sym{b}\times \sym{2a}}(\sgn(b)\boxtimes H^{(2^a)}).\] The module on the right hand side is called a twisted Foulkes module.

\end{eg}

\section{Higher Lie Powers}\label{S: ch and dim}


Let $q$ be a composition coprime to $p$. In this section, we give the dimension and character (when $V$ is considered as the $F\GL(V)$-module) formulae and study the decomposition problem for the higher Lie power $L^q(V)$ when $\lambda(q)$ is $p$-regular. As a result, we can decompose the projective higher Lie module $\Lie_F(q)$ completely in terms of the projective indecomposable modules $P^\gamma$'s.

The main step is to prove that the surjective map in Lemma \ref{L: Sym surj Lie} is indeed an isomorphism under the extra assumption that $\lambda(q)$ is $p$-regular.

\begin{thm}\label{T: high Lie isom} Suppose that $\lambda\in\P_p(n)$, $(\lambda,p)=1$ and $m_i=m_i(\lambda)$. For any $ F G$-module $V$, the map $\psi:S^{m_n}(L^n(V))\otimes \cdots\otimes S^{m_1}(L^1(V))\to L^\lambda(V)$ is an isomorphism where \[\psi(\prod_{j=1}^{m_n}P_{n,j}\otimes \cdots\otimes\prod_{j=1}^{m_1}P_{1,j})=\omega_\lambda\cdot P_{n,1}\cdots P_{1,m_1}\] where $P_{i,j}\in L^i(V)$ for each admissible $i,j$. Suppose further $q$ is a composition such that $\lambda(q)=\lambda$. We have \[L^q(V)\cong S^{m_n}(L^n(V))\otimes \cdots\otimes S^{m_1}(L^1(V)).\] In particular, when $G=\GL(V)$, $L^q(V)$ is isomorphic to a direct sum of tilting modules of the form $T(\lambda)$ where $\lambda\in\P_p(m,n)$.
\end{thm}
\begin{proof} By Lemma \ref{L: Sym surj Lie}, the map is surjective and commutes with the action of $G$.

First assume that $\dim_ F(V)\geq n$. Consider the natural $F\GL(V)$-module $V$. Since $\omega_\lambda^2=\lambda?\omega_\lambda$ and $\lambda?\neq 0$ in $F$, $L^\lambda(V)$ is a direct summand of $T^n(V)$. When $0<m_i<p$, in particular $p\nmid i$, we have $L^i(V)$ is a direct summand of $T^i(V)$ and hence $S^{m_i}(L^i(V))$ is a direct summand of $(T^i(V))^{\otimes m_i}=T^{im_i}(V)$. In turn, $S:=S^{m_n}(L^n(V))\otimes\cdots\otimes S^{m_1}(L^1(V))$ is a direct summand of $T^n(V)$. As such, both $L^\lambda(V)$ and $S$ are direct sums of projective (and injective) tilting modules (see \cite[Proposition 4.2]{Erd94}) of the form $T(\lambda)$ for some $\lambda\in\P_p(m,n)$. The surjective map $\psi$ therefore splits. Since the Schur functor $f$ is exact and does not map any tilting module to zero, by the counting argument in the proof of Theorem \ref{T: higher lie mod} which asserts $f(S)\cong f(L^\lambda(V))$, we must have $S\cong L^\lambda(V)$, i.e., $\psi$ is an isomorphism. 

Without the assumption $\dim_ F V\geq n$, let $W$ be a vector space containing $V$ as a subspace such that $\dim_ F W\geq n$. We use $\psi_V$ and $\psi_W$ to denote the maps $\psi$ with respect to these spaces. By the previous paragraph, $\psi_W$ is an isomorphism. Since $\psi_V$ is obtained from $\psi_W$ by restricting to $S^{m_n}(L^n(W))\otimes \cdots\otimes S^{m_1}(L^1(W))\cap T^n(V)$, $\psi_V$ is also injective. So $\psi_V$ is an isomorphism. The proof is now complete using Lemma \ref{L: Lq isom Lr} for $q\approx \lambda$.
\end{proof}

We will draw some corollaries following our Theorem \ref{T: high Lie isom}. Before this, we need the following notation.

\smallskip

Let $\delta,\eta$ be compositions. Assuming that $\delta_i=0=\eta_j$ whenever $i>\ell(\delta)$ and $j>\ell(\eta)$. We define
\begin{align*}
  \delta\eta&=(\delta_i\eta_i)_{i=1}^\infty,\index{$\delta\eta$}\\
  \delta^{[\eta]}&=\#^\infty_{i=1}(\underbrace{\delta_i,\ldots,\delta_i}_{\text{$\eta_i$ times}})=({\delta_1}^{\eta_1},{\delta_2}^{\eta_2},\ldots),\index{$\delta^{[\eta]}$}
\end{align*} and they are to be read as compositions by deleting the zeroes. If $\ell(\delta)=\ell(\eta)$ and $\gamma=\delta\eta$, we write $\delta\mid \gamma$ \index{$\delta\mid \gamma$}and write $\gamma/\delta$ \index{$\gamma/\delta$}for $\eta$. By convention, $\varnothing\mid\varnothing$. For example, if $\delta\mid (i^{\ell(\eta)})$, we have $\ell=\ell(\delta)=\ell(\eta)$ and  \[(\eta\delta)^{[(i^{\ell(\eta)})/\delta]}=((\eta_1\delta_1)^{i/\delta_1},\ldots,(\eta_\ell\delta_\ell)^{i/\delta_\ell})\] which is a composition of  $i|\eta|$.  Let $\mu$ be the M\"{o}bius function with the convention that $\mu(0)=1$. We let \[\mu(\delta)=\prod_{i=1}^\infty\mu(\delta_i).\index{$\mu(\delta)$}\] Let $\b{\delta}$ \index{$\b{\delta}$}be sequences of compositions with finite number of nonempty terms. Define
\begin{align*}
  \b{\delta}?&=\prod^\infty_{i=1}\b{\delta}_i?,& I_{\b{\delta}}&=((i^{\ell(\b{\delta}_i)}))^\infty_{i=1},& I_{\b{\delta}}^\times &=\prod^\infty_{i=1}i^{\ell(\b{\delta}_i)},
  &\mu(\b{\delta})&=\prod_{i=1}^\infty\mu(\b{\delta}_i).\index{$\b{\delta}?$}\index{$I_{\b{\delta}}$}\index{$I_{\b{\delta}}^\times$} \index{$\mu(\b{\delta})$}
\end{align*} Let $\b{\eta}$ be another sequence of compositions such that $\b{\delta}_i\mid I_{\b{\eta}_i}$ for all $i\in \NN$. We denote it as $\b{\delta}\mid I_{\b{\eta}}$ and   
define \[\b{\eta}\scont \b{\delta}=\#^\infty_{i=1}(\b{\eta}_i\b{\delta}_i)^{[(i^{\ell(\b{\eta}_i)})/\b{\delta}_i]}\index{$\b{\eta}\scont \b{\delta}$}\] and consider it as a composition by deleting zeroes. For an example of the notations introduced above, we refer the readers to a table in Example \ref{Ex: 221}.

For each $n\in\NN$ and $j\in [0,n]$, let $\varsigma_{j,n}$ \index{$\varsigma_{j,n}$}be the number of permutations $\pi\in\sym{n}$ such that $\pi$ has exactly $j$ cycles in the cycle type of $\pi$, i.e., $j=\ell(\rho(\pi))$ where $\rho(\pi)\in\P(n)$ denotes the cycle type of $\pi$. By convention, we define 
$\varsigma_{0,0}=1$. It is well-known that
$\prod_{i=0}^{m-1}(x+i)=\sum_{\pi\in\sym{m}}x^{\ell(\rho(\pi))}.
$ Therefore we obtain the following lemma.

\begin{lem}\label{L: gen length} For each $m\in\NN$, we have \[\prod_{i=0}^{m-1}(x+i)=\sum_{j=1}^m\varsigma_{j,m}x^j.\] Furthermore, $\varsigma_{j,m}$ is the sum of the product of the numbers in $S\subseteq [1,m-1]$ such that $|S|=m-j$.
\end{lem}


For sequences $\delta,\gamma$ in $\NN_0$ (with finite number of nonzero terms) such that $\delta_i\leq \gamma_i$ for all $i\in\NN$, we define \[\varsigma_{\delta,\gamma}=\prod_{i=1}^\infty i^{\gamma_i-\delta_i}\varsigma_{\delta_i,\gamma_i}.\index{$\varsigma_{\delta,\gamma}$}\] Notice that $\varsigma_{\delta,\gamma}\neq 0$ if and only if, for all $i\in\NN$, $\gamma_i>0$ implies $\delta_i>0$. 

Let $d_i=\dim_ F L^i(V)$ as in Theorem \ref{T: lie dim}. For each $i$, let $B_i=\{b_{i,1},\ldots,b_{i,d_i}\}$ be a basis for the $i$th Lie power $L^i(V)$ (for example, the Hall bases (see \cite{Reu})). The symmetric power $S^{m_i}(L^i(V))$ has a basis given by the product of ascending powers \[\{b_{i,j_1}\cdots b_{i,j_{m_i}}:1\leq j_1\leq \cdots\leq j_{m_i}\leq d_i\}.\] Therefore, a basis for $L^\lambda(V)$ can be obtained through the isomorphism $\psi$ given in Theorem \ref{T: high Lie isom} when $(\lambda,p)=1$. Furthermore, we have the dimension formula in the next corollary.

\begin{cor}\label{C: higher dim}  Suppose that $q\in\C(n)$, $(q,p)=1$, $\lambda(q)\in\P_p(n)$ and $\ell=\dim_ F V$. Let $R(q)$ be the set consisting of all sequences $\eta$ in $\NN_0$ such that $\eta_i\leq m_i(q)$ and, for all $i\in\NN$, $m_i(q)>0$ implies $\eta_i>0$. The dimension of $L^q(V)$ is \[\prod_{i=1}^\infty {m_i(q)+d_i-1\choose m_i(q)}=\frac{1}{q?}\sum_{\eta\in R(q)}\varsigma_{\eta,\varrho}\sum \mu(\delta)\ell^{\left |{\gamma}/{\delta}\right |}\] where $d_i=\dim_FL^i(V)$, $\varrho=(m_i(q))^\infty_{i=1}$ and the second sum on the right hand side is taken over all pairs of $(\delta,\gamma)$ such that $\delta\mid \gamma$ and $m(\gamma)=\eta$. 
\end{cor}
\begin{proof} The first formulation of the dimension for $L^q(V)\cong L^{\lambda(q)}(V)$ follows directly from Theorem \ref{T: high Lie isom}. Also,
\begin{align*}
\prod_{i=1}^\infty{d_i+m_i(q)-1\choose m_i(q)}&=\prod_{i=1}^\infty\frac{(d_i+(m_i(q)-1))\cdots (d_i+1)d_i}{m_i(q)!}\\
  &=\frac{1}{q?}\prod_{i=1}^\infty i^{m_i(q)}(d_i+(m_i(q)-1))\cdots (d_i+1)d_i\\
  &=\frac{1}{q?}\prod_{i=1}^\infty \sum_{j=1}^{m_i(q)}i^{m_i(q)-j}\varsigma_{j,m_i(q)}(id_i)^j\\
  &=\frac{1}{q?}\sum_{\eta\in R(q)}\prod_{i=1}^\infty i^{m_i(q)-\eta_i}\varsigma_{\eta_i,m_i(q)}(id_i)^{\eta_i}\\
  &=\frac{1}{q?}\sum_{\eta\in R(q)}\prod_{i=1}^\infty i^{m_i(q)-\eta_i}\varsigma_{\eta_i,m_i(q)}\left (\sum_{d\mid i}\mu(d)\ell^{i/d}\right )^{\eta_i}\\
  &=\frac{1}{q?}\sum_{\eta\in R(q)}\varsigma_{\eta,\varrho}\sum \mu(\delta)\ell^{\left |{\gamma}/{\delta}\right |}
\end{align*} where third and penultimate equalities are obtained using Lemma \ref{L: gen length} and Theorem \ref{T: lie dim} respectively, and, in the final equality, the second sum is taken over all pairs of $(\delta,\gamma)$ such that $\delta\mid \gamma$ and $m(\gamma)=\eta$ and the equality can be derived by expanding the product \[\underbrace{\mu(1)\ell^{1/1}\cdots \mu(1)\ell^{1/1}}_{\text{$\eta_1$ times}}\underbrace{(\mu(1)\ell^{2/1}+\mu(2)\ell^{2/2})\cdots (\mu(1)\ell^{2/1}+\mu(2)\ell^{2/2})}_{\text{$\eta_2$ times}}\cdots\] as the sum of `monomials' of the form $\mu(d)\ell^{i/d}$.
\end{proof}

To illustrate the dimension formula in terms of the summation as in Corollary \ref{C: higher dim}, we give an example.

\begin{eg} Assume that $p\geq 3$, $\ell=\dim_F V$ and $q=(2^2,1^2)$. We have $\varrho=(2,2)$, $R(q)=\{(2,2),(2,1),(1,2),(1,1)\}$ and
\begin{align*}
&\dim_ F L^{(2^2,1^2)}(V)\\
&=\frac{1}{16}((\underbrace{\ell^6-2\ell^5+\ell^4}_{\eta=(2,2)})+(\underbrace{\ell^5-2\ell^4+\ell^3}_{\eta=(1,2)})+ 2(\underbrace{\ell^4-\ell^3}_{\eta=(2,1)})+ 2(\underbrace{\ell^3-\ell^2}_{\eta=(1,1)}))\\
&=\frac{\ell^2(\ell+1)(\ell-1)(\ell^2-\ell+2)}{16}.
\end{align*}
\end{eg}

For the rest of this section, we view $V$ is the natural $F\GL(V)$-module.  Recall the formal character of the $i$th Lie power as in Theorem \ref{T: lie ch}, which we shall now denote as $\ell_i$, and the plethysm $\circ$ on the ring of symmetric functions in Subsection \ref{SS: symmetric functions}.

\begin{cor}\label{C: higher ch} Suppose that $q\in\C(n)$, $(q,p)=1$ and $\lambda(q)\in\P_p(n)$. The formal character for $L^q(V)$ is \[\prod_{i=1}^\infty h_{m_i(q)}\circ\ell_i=\sum \frac{1}{\b{\lambda}?I_{\b{\lambda}}^\times}\mu(\b{\delta})p_{\b{\lambda}\sharp\b{\delta}}\] where the sum is taken over all $\b{\lambda}\in\prod^\infty_{i=1}\P(m_i(q))$ and $\b{\delta}\mid I_{\b{\lambda}}$.
\end{cor}
\begin{proof} Let $\Fp$ be the prime subfield of $F$ and $(\Q_p,\Z_p,\Fp)$ be the $p$-modular system. By Theorem \ref{T: high Lie isom}, $\Lie_F(q)$ is a projective $F\sym{n}$-module and therefore, by \cite[Theorem 1.21]{BH} where $\Lie_F(q)\cong \Lie_\Fp(q)\otimes_\Fp F$, $\Lie_\Fp(q)$ is projective. The unique lift of $\Lie_\Fp(q)$ is the projective module $\Lie_{\Z_p}(q)$.  By \cite[Theorem 8.23]{Reu}, we have \[\ch(\Lie_{\Z_p}(q))=\prod^\infty_{i=1}h_{m_i(q)}\circ\ell_i.\] Since $\ch(\ch(fX))=\ch(X)$ for a tilting module $X$ (see \cite[\S2.3]{DE}),  by Theorem \ref{T: high Lie isom}, we have $\ch(L^q(V))=\ch(\Lie_{\Z_p}(q))$. It is well-known that \[h_m=\sum_{\lambda\in\P(m)}\frac{1}{\lambda?}p_\lambda=\sum_{\lambda\in\P(m)}\frac{1}{\lambda?}\prod_{j=1}^\infty p_j^{m_j(\lambda)}.\] By the definition of plethysm of symmetric functions, we have \[h_m\circ\ell_i=\sum_{\lambda\in\P(m)}\frac{1}{\lambda?}\prod_{j=1}^\infty \left (\frac{1}{i}\sum_{k\mid i}\mu(k)p_{jk}^{i/k}\right )^{m_j(\lambda)}=\sum_{\lambda\in\P(m)}\frac{1}{\lambda?i^{\ell(\lambda)}}\sum_{\delta\mid (i^{\ell(\lambda)})}\mu(\delta)p_{(\lambda d)^{[(i^{\ell(\lambda)})/\delta]}}.\] Therefore,
\begin{align*}
\prod_{i=1}^\infty h_{m_i(q)}\circ \ell_i&=\sum_{\b{\lambda}\in\prod^\infty_{i=1}\P(m_i(q))}\frac{1}{\b{\lambda}?I_{\b{\lambda}}^\times}\sum_{\b{\delta}\mid I_{\b{\lambda}}}\prod^\infty_{i=1} \mu(\b{\delta}_i)p_{(\b{\lambda}_i\b{\delta}_i)^{[(i^{\ell(\b{\lambda}_i)})/\b{\delta}_i]}}\\
&=\sum_{\b{\lambda}\in\prod^\infty_{i=1}\P(m_i(q))}\frac{1}{\b{\lambda}?I_{\b{\lambda}}^\times}\sum_{\b{\delta}\mid I_{\b{\lambda}}}\mu(\b{\delta})p_{\b{\lambda}\scont\b{\delta}}.
\end{align*}
\end{proof}


By Theorem \ref{T: high Lie isom}, when $(q,p)=1$ and $\lambda(q)$ is $p$-regular, the higher Lie power $L^q(V)$ is a direct sum of the tilting modules of the form $T(\lambda)$ for some $\lambda\in\P_p(m,n)$. In the next theorem, we give the multiplicity of each $T(\lambda)$ in $L^q(V)$ in terms of the irreducible Brauer characters $\beta^\gamma$'s. As a consequence, we have a complete decomposition of projective higher Lie modules into indecomposables.

\begin{thm}\label{T: higher Lie decomp} Suppose that $q\in\C(n)$, $(q,p)=1$ and $\lambda(q)\in\P_p(n)$. We have $L^q(V)\cong \bigoplus_{\gamma\in\P_p(m,n)}m_\gamma T(\gamma)$ where \[m_\gamma=\sum \frac{1}{\b{\lambda}?I_{\b{\lambda}}^\times}\mu(\b{\delta})\beta^\gamma({\b{\lambda}\scont\b{\delta}})\] where the sum is taken over all $\b{\lambda}\in\prod^\infty_{i=1}\P(m_i(q))$ and $\b{\delta}\mid I_{\b{\lambda}}$.
\end{thm}
\begin{proof} Since $\ch(\eta?\Char_{\eta})=p_\eta$, by Corollary \ref{C: higher ch}, we have $\ch(\chi_q)=\ch(L^q(V))$ where \[\chi_q=\sum_{\substack{\b{\lambda}\in\prod^\infty_{i=1}\P(m_i(q)),\\ \b{\delta}\mid I_{\b{\lambda}}}} \frac{1}{\b{\lambda}?I_{\b{\lambda}}^\times}\mu(\b{\delta})(\b{\lambda}\scont\b{\delta})?\Char_{\b{\lambda}\scont\b{\delta}}.\] For every such $\b{\lambda}$ and $\b{\delta}$, by definition, $\b{\lambda}\scont\b{\delta}$ is the concatenation of $(\b{\lambda}_i\b{\delta}_i)^{[(i^{\ell(\b{\lambda}_i)})/\b{\delta}_i]}$ which is not empty only if $p\nmid i$.  Furthermore, $(\b{\lambda}_i,p)=1$ as $|\b{\lambda_i}|=m_i(q)<p$. As such, $(\b{\lambda}\scont\b{\delta},p)=1$ and hence $\chi_q$ vanishes on $p$-singular elements of $\sym{n}$. By \cite[\S2.3(6)]{DE}, $m_\gamma=(\chi_q,\beta^\gamma)$. Since \[(\Char_{\b{\lambda}\scont\b{\delta}},\beta^\gamma)=\frac{1}{(\b{\lambda}\scont\b{\delta})?}\beta^\gamma(\b{\lambda}\scont\b{\delta}),\] we have the desired result.
\end{proof}

Assume that $m\geq n$ in Theorem \ref{T: higher Lie decomp}, apply the Schur functor, as $f(T(\gamma))=P^\gamma$, we get the corollary below.

\begin{cor}\label{C: proj higher Lie decomp} Suppose that $q\in\C(n)$, $(q,p)=1$ and $\lambda(q)\in\pReg(n)$. We have $\Lie_F(q)\cong\bigoplus_{\gamma\in\P_p(n)}m_\gamma P^\gamma$ where $m_\gamma$ is given as in Theorem \ref{T: higher Lie decomp}.
\end{cor}

We illustrate Theorem \ref{T: higher Lie decomp} and Corollary \ref{C: proj higher Lie decomp} with an example.

\begin{eg}\label{Ex: 221} Let $p=3$, $n=5$ and $q=(2,2,1)$. The Brauer character table for $\sym{5}$ is given below.
\[\begin{array}{l|rrrrrrr}
&(1^5)&(2,1^3)&(2^2,1)&(4,1)&(5)\\ \hline \hline
\beta^{(5)}&1&1&1&1&1\\
\beta^{(3,2)}&1&-1&1&-1&1\\
\beta^{(4,1)}&4&2&0&0&-1\\
\beta^{(2^2,1)}&4&-2&0&0&-1\\
\beta^{(3,1^2)}&6&0&-2&0&1
\end{array}\] We have $m_1(q)=1$ and $m_2(q)=2$. All the possible $\b{\lambda}\in\P(1)\times \P(2)$, $\b{\delta}$ such that $\b{\delta}\mid I_{\b{\lambda}}$, the compositions $\b{\lambda}\scont\b{\delta}$ and the values $\frac{1}{\b{\lambda}?I^\times_{\b{\lambda}}}\mu(\b{\delta})$ are summarised in the table below.
\[\def\arraystretch{1.3}\begin{array}{|c|c|c|c|c|} \hline \b{\lambda}&I_{\b{\lambda}}&\b{\delta}&\b{\lambda}\scont\b{\delta}&\frac{1}{\b{\lambda}?I^\times_{\b{\lambda}}}\mu(\b{\delta})\\ \hline
((1),(2))&((1),(2))&((1),(1))&(1,2,2)&\frac{1}{4}\\ \cline{3-5}
&&((1),(2))&(1,4)&-\frac{1}{4}\\ \hline
((1),(1,1))&((1),(2,2))&((1),(2,2))&(1,2,2)&\frac{1}{8}\\ \cline{3-5}
&&((1),(1,2))&(1,1,1,2)&-\frac{1}{8}\\ \cline{3-5}
&&((1),(2,1))&(1,2,1,1)&-\frac{1}{8}\\ \cline{3-5}
&&((1),(1,1))&(1^5)&\frac{1}{8}\\ \hline
\end{array}\] With respect to every $\gamma\in\P_3(5)$, we have the values $m_\gamma$ given below.
\[\begin{array}{|c|c|c|c|c|c|}
\hline \gamma&(5)&(4,1)&(3,2)&(3,1,1)&(2,2,1)\\ \hline
m_\gamma&0&0&1&0&1\\ \hline
\end{array}\] If $\dim_ F V=1$, we have $L^q(V)=0$. If $\dim_ F V=2$, we have $L^q(V)\cong T((3,2))$. If $\dim_ F V\geq 3$, we have $L^q(V)\cong T((3,2))\oplus T((2,2,1))$. Also, $\Lie_F(q)\cong P^{(3,2)}\oplus P^{(2,2,1)}$.
\end{eg}

\section{Periodic Higher Lie Modules}\label{S: period}

Throughout this section, we assume that $p>0$. After we have successfully established the decomposition of projective higher Lie modules (see Corollary \ref{C: proj higher Lie decomp}), in this section, we investigate the `first' non-projective case of the higher Lie modules, that is the non-projective periodic higher Lie modules $\Lie_F(q)$.  As such, in view of Corollary \ref{C: Complexity Lie}(ii), in this section, we assume that $q\in\C(n)$ with the following hypothesis:

\begin{enumerate}
\item [(H)] $(q,p)=1$ and there is a unique number $k$ such that $m_k(q)\in [p,2p-1]$ and $m_i(q)<p$ for all $i\neq k$.
\end{enumerate} In particular, we investigate their periods and Heller translates. This is in line with the study of the non-projective periodic Lie modules done earlier by Tan and the author \cite{LT16}.  We begin with some notation and collate some necessary background for this section.  

\smallskip
We adopt the notation $M\oplus \proj$ to denote the direct sum of an $FG$-module $M$ with some projective $FG$-module. 

\begin{lem}\label{L: Heller} Let $H\subseteq G$ and $K$ be finite groups. Suppose that $N,Q$ are $F H$- and $F K$-modules respectively such that $Q$ is projective. Let $n\in\Z$. We have
\begin{enumerate}[(i)]
\item $\Omega^n(N\boxtimes Q)\cong (\Omega^n N)\boxtimes Q$, and
\item {\cite[Proposition 4.4(viii)]{Carlson}} $\ind^G_H\Omega^n(N)\cong \Omega^n(\ind^G_HN)\oplus\proj$.
\end{enumerate}
\end{lem}
\begin{proof} For part (i), it suffices to prove the case when $n=1$ and both $N,Q$ are indecomposable. If $N$ is projective, $N\boxtimes Q$ is projective and therefore both modules in the statement are zero. Suppose that $N$ is not projective. There are short exact sequences
\begin{align*}
  0\to (\Omega N)\boxtimes Q\to P(N)\boxtimes Q\to N\boxtimes Q\to 0,\\
  0\to \Omega(N\boxtimes Q)\to P(N\boxtimes Q)\to N\boxtimes Q\to 0.
\end{align*} By Schanuel's lemma (see \cite[Lemma 1.5.3]{Ben1}) and \cite[Proposition 1.2]{Kuel}, we must have $\Omega N\boxtimes Q\cong \Omega(N\boxtimes Q)$.
\end{proof}

In \cite[\S4]{ES}, Erdmann-Schocker defined $S^p(\Lie_F(k))=\ind^{\sym{pk}}_{\sym{k}\wr\sym{p}}(\Lie_F(k)^{\wr p})$ as the $p$th symmetrisation of $\Lie_F(k)$ (or more generally, the $p$th symmetrisation of a projective $F\sym{k}$-module). In fact, it is isomorphic to $\Lie_F((k^p))$ by Theorem \ref{T: higher lie mod}.  Therefore, we may restate their result as follows.

\begin{thm}[{\cite[Theorem 10]{ES}}]\label{T: SES} Let $p\nmid k\in \NN$. There is a short exact sequence of $F\sym{pk}$-modules
\[0\to\Lie_F(pk)\to e_{(pk),F}F\sym{pk}\to \Lie_F((k^p))\to 0.\]
\end{thm}

As such, $\Omega(\Lie_F((k^p)))\cong \Omega^0(\Lie_F(pk))$ as the middle term is projective. The Heller translates and periods of $\Lie_F(pk)$ (and therefore for $\Lie_F((k^p))$) have been computed in \cite{LT16} and we shall now describe.


Recall the notation $S^\lambda$ for the Specht module labelled by a partition $\lambda$. We denote $S_\lambda$ for its dual. For any integer $j$, let $j_p$ be 0 if $p=2$ or the residue of $n$ upon division by $2p-2$ if $p\geq 3$. For each $j\in\Z$, define\[S(j)=\left \{\begin{array}{ll} S^{(p-j_p-1,1^{j_p+1})}&\text{if $j_p\in [0,p-2]$,}\\ S_{(j_p-p+3,1^{2p-j_p-3})}&\text{if $j_p\in [p-1,2p-3]$}.\end{array}\right .\index{$S(j)$}\] For example, when $p=2$, $S(j)=S^{(1,1)}$ is the trivial module for all $j$.

We denote $\Delta_{pk}=\Delta_k(\sym{p})\sym{k}^{[p]}\subseteq \sym{kp}$ \index{$\Delta_{pk}$}(recall Subsection \ref{SS: generality} for the notation). Notice that $\Delta_{pk}\cong \sym{p}\times\sym{k}$. For each $j\in\Z$ and $\nu\in\pReg(k)$, through the identification of $\Delta_{pk}$ with $\sym{p}\times \sym{k}$, define the $F\sym{pk}$-modules
\begin{align*}
Z_{j,\nu}&=\Omega^0(\ind^{\sym{kp}}_{\Delta_{pk}}(S(j)\boxtimes P^\nu)),\index{$Z_{j,\nu}$}
\end{align*} and $Z_j=\bigoplus_{\nu\in\P_p(k)} n_\nu Z_{j,\nu}$ \index{$Z_j$}with $n_\nu$ given as in Theorem \ref{T: DE proj decomp}. Using our notation, we may restate the main results in \cite{LT16} as follows.

\begin{thm}[{\cite[Theorems 1.1 and 3.5]{LT16}}]\label{T: LT main} Let $p\nmid k\in\NN$. The Lie module $\Lie_F(pk)$ is periodic with period 1 if $p=2$, or $2p-2$ if $p\geq 3$. Moreover, for any $j\in\Z$, \[\Omega^j(\Lie_F(pk))\cong Z_j=\bigoplus_{\nu\in\pReg(k)}n_\nu Z_{j,\nu}\] and each summand (namely, $n_\nu\neq 0$) $Z_{j,\nu}$ is non-projective indecomposable.
\end{thm}


It is not clear to the author if $Z_{j,\nu}$ remains indecomposable without the assumption $n_\nu\neq 0$. The problem boils down to the question if $\Omega^0(\ind^{\sym{kp}}_{\Delta_{pk}}(F\boxtimes P^\nu))$ is indecomposable or not.

Together with Theorem \ref{T: SES}, we have the following proposition.

\begin{prop}\label{P: LT period} For any $j\in\Z$ and $p\nmid k$, we have $\Omega^j(\Lie_F((k^p)))\cong Z_{j-1}$.
\end{prop}


In order to prove the description of the Heller translates of  higher Lie modules, we need the following two lemmas. The proofs require the notion of Green correspondence. Since we do not need it elsewhere in the paper, we refer the reader to standard textbooks, for example, \cite{Ben1}. 

\begin{lem}\label{L: Green corres} Let $j\in\Z$, $n,k\in\NN$ such that $n\geq kp$, $\nu\in\pReg(k)$ and $\gamma\in\pReg(n-kp)$. Suppose that $n_\nu\neq 0$. The module $X:=Z_{j,\nu}\boxtimes P^\gamma$ is indecomposable and its Green correspondent (also indecomposable) is \[Y:=\Omega^0(\ind^{\sym{n}}_{\sym{kp}\times\sym{n-kp}}(Z_{j,\nu}\boxtimes P^\gamma)).\] Moreover, $\ind^{\sym{n}}_{\sym{kp}\times \sym{n-kp}}X\cong Y\oplus\proj$ and $\res^{\sym{n}}_{\sym{kp}\times \sym{n-kp}}Y\cong X\oplus \proj$.
\end{lem}
\begin{proof} Let $C_p$ be the cyclic subgroup of $\sym{p}$ generated by any $p$-cycle. Since $Z_{j,\nu}$ is indecomposable (with vertex $\Delta_k(C_p)$), by \cite[Proposition 1.2]{Kuel}, the module $Z_{j,\nu}\boxtimes P^\gamma$ is indecomposable non-projective with a vertex $\Delta_k(C_p)\subseteq \Delta_k(\sym{p})$ and $\N_{\sym{n}}(\Delta_k(C_p))\subseteq \sym{kp}\times \sym{n-kp}$. By Green's correspondence (see, for example, \cite[Theorem 3.12.2(ii)]{Ben1}), since $|\Delta_k(C_p)|=p$, the module $\ind^{\sym{n}}_{\sym{kp}\times\sym{n-kp}}(Z_{j,\nu}\boxtimes P^\gamma)$ has a unique indecomposable non-projective summand, which is its Green's correspondent.
\end{proof}

\begin{lem}\label{L: M(p,r)} Suppose that $m=p+r$ for some $r\in [0,p-1]$. We have an isomorphism of $F\sym{m}$-modules \[\ind^{\sym{m}}_{\sym{p}\times \sym{r}}F\cong F\oplus \text{(proj)}.\]
\end{lem}
\begin{proof} We are in the cyclic-vertex case. Let $P$ be a Sylow $p$-subgroup of $\sym{p}$. Notice that $\N_{\sym{m}}(P)\subseteq \sym{p}\times \sym{r}$. Clearly, the trivial $F\sym{m}$-module restricts to the trivial $F(\sym{p}\times\sym{r})$-module. By Green's correspondence, we get the isomorphism as $P\cap gPg^{-1}=\{1\}$ whenever $g\not\in \sym{p}\times \sym{r}$.
\end{proof}


\begin{lem}\label{L: minus p} Suppose that $m=p+r$ for some $r\in [0,p-1]$ and $p\nmid k\in\NN$. We have an isomorphism of $F\sym{km}$-modules
\[\ind^{\sym{km}}_{\sym{kp}\times \sym{kr}}(\Lie_F((k^p))\boxtimes \Lie_F((k^r)))\cong \Lie_F((k^m))\oplus \proj.\]
\end{lem}
\begin{proof} Let $F_\delta=F_{\delta_k}^{\wr m}$ be the one-dimensional $F[C_k\wr\sym{m}]$-module where $F_{\delta_k}$ depends only on $k$ and $\Lie_F((k^m))\cong\ind^{\sym{km}}_{C_k\wr \sym{m}}F_\delta$ as in Theorem \ref{T: higher lie mod}. Take inflation and then tensor product with $F_\delta$ for the equation as in Lemma \ref{L: M(p,r)}, we have \[F_\delta\otimes \inf^{C_k\wr\sym{m}}_{\sym{m}}\ind^{\sym{m}}_{\sym{p}\times \sym{r}}F\cong \left (F_\delta\otimes \inf^{C_k\wr\sym{m}}_{\sym{m}}F\right )\oplus \left (F_\delta\otimes \inf^{C_k\wr\sym{m}}_{\sym{m}}\proj\right ).\] Since $p\nmid k$, the top group $\sym{m}$ contains a Sylow $p$-subgroup $P$ of $C_k\wr\sym{m}$ and therefore the restriction of the inflated module to $P$ is again projective. As such, the second summand is projective. Notice that the first summand is just $F_\delta$. So we have
\begin{align*}
F_\delta\oplus\text{(proj)}\cong F_\delta\otimes \inf^{C_k\wr\sym{m}}_{\sym{m}}\ind^{\sym{m}}_{\sym{p}\times \sym{r}}F&\cong F_\delta\otimes \ind^{C_k\wr\sym{m}}_{C_k\wr(\sym{p}\times \sym{r})}\inf^{C_k\wr(\sym{p}\times \sym{r})}_{\sym{p}\times\sym{r}}F\\
&\cong \ind^{C_k\wr\sym{m}}_{C_k\wr(\sym{p}\times \sym{r})}F_\delta.
\end{align*} Inducing the modules to $\sym{km}$, we have
\begin{align*}
  \ind^{\sym{km}}_{C_k\wr\sym{m}} (F_\delta\oplus \proj)&\cong \ind^{\sym{km}}_{C_k\wr(\sym{p}\times \sym{r})}F_\delta\\
  \ind^{\sym{km}}_{C_k\wr\sym{m}} (F_\delta)\oplus\proj&\cong \ind^{\sym{km}}_{\sym{kp}\times\sym{kr}}\ind^{\sym{kp}\times\sym{kr}}_{C_k\wr(\sym{p}\times \sym{r})}(F_\delta\boxtimes F_\delta)\\
  &\cong \ind^{\sym{km}}_{\sym{kp}\times \sym{kr}}(\Lie_F((k^p))\boxtimes \Lie_F((k^r))).
\end{align*} Notice that the left hand side is precisely $\Lie_F((k^m))\oplus\proj$.
\end{proof}

We are now in the position to state and prove our main theorem for this section.


\begin{thm}\label{T: Heller higher Lie} Let $q\in \C(n)$ satisfying Hypothesis (H), let $q'$ be a composition such that $q\approx (k^p)\cont q'$ and let $\Lie_F(q')\cong \bigoplus_{\gamma\in\pReg(n-kp)}s_\gamma P^\gamma$. Then \[\Omega^j(\Lie_F(q))\cong \bigoplus_{\substack{\nu\in\pReg(k),\\ \gamma\in\pReg(n-kp)}} s_\nu s_\gamma\Omega^0(\ind^{\sym{n}}_{\sym{kp}\times\sym{n-kp}}(Z_{j-1,\nu}\boxtimes P^\gamma)).\] Furthermore, each (non-zero, i.e., $s_\nu s_\gamma\neq 0$) summand $\Omega^0(\ind^{\sym{n}}_{\sym{kp}\times\sym{n-kp}}(Z_{j-1,\nu}\boxtimes P^\gamma))$ in the equation above is non-projective indecomposable.
\end{thm}
\begin{proof} Without loss of generality, assume that $q=(k^p)\cont q'=(k^{m_k})\cont q''$ where $q'=(k^{m_k-p})\cont q''$. Notice that $(q',p)=1=(q'',p)$ and both $m_i(q')<p$, $m_i(q'')<p$ for all $i$. Therefore, both $\Lie_F(q')$ and $\Lie_F(q'')$ are projective by Corollary \ref{C: Complexity Lie}(i). We decompose $\Lie_F(q')$ as in the statement. By Lemmas \ref{L: Heller}, \ref{L: minus p}, Proposition \ref{P: LT period} and Corollary \ref{C: higher lie mod split}
\begin{align*}
  \Omega^j(\Lie_F(q))&\cong\Omega^j(\ind^{\sym{n}}_{\sym{km_k}\times \sym{n-km_k}}(\Lie_F((k^{m_k}))\boxtimes \Lie_F(q'')))\\
  &\cong\Omega^j(\ind^{\sym{n}}_{\sym{kp}\times \sym{kr}\times \sym{n-km_k}}(\Lie_F((k^p))\boxtimes \Lie_F((k^r))\boxtimes \Lie_F(q'')))\\
  &\cong \bigoplus_{\substack{\nu\in\pReg(k),\\ \gamma\in\pReg(n-kp)}} s_\nu s_\gamma\Omega^0(\ind^{\sym{n}}_{\sym{kp}\times\sym{n-kp}}(Z_{j-1,\nu}\boxtimes P^\gamma)).
\end{align*} Each nonzero summand $\Omega^0(\ind^{\sym{n}}_{\sym{kp}\times\sym{n-kp}}(Z_{j-1,\nu}\boxtimes P^\gamma))$ is indecomposable by Lemma \ref{L: Green corres}. The statement now follows.
\end{proof}

We end the section with the following corollary describing the period of a non-projective periodic higher Lie module.

\begin{cor}\label{C: period} Let $q\in \C(n)$ satisfying Hypothesis (H). Then $\Lie_F(q)$ is non-projective periodic with the period $2p-2$ if $p\geq 3$ and 1 if $p=2$.
\end{cor}
\begin{proof} By Theorem \ref{T: Heller higher Lie}, the Heller translates of $\Lie_F(q)$ are clearly nonzero. Suppose that $\Omega^j(\Lie_F(q))\cong \Omega^{j'}(\Lie_F(q))$. By Theorem \ref{T: Heller higher Lie} and Lemma \ref{L: Green corres}, we have \[\bigoplus_{\substack{\nu\in\pReg(k),\\ \gamma\in\pReg(n-kp)}} s_\nu s_\gamma (Z_{j-1,\nu}\boxtimes P^\gamma)\cong \bigoplus_{\substack{\nu\in\pReg(k),\\ \gamma\in\pReg(n-kp)}} s_\nu s_\gamma (Z_{j'-1,\nu}\boxtimes P^\gamma)\] as $F[\sym{pk}\times\sym{n-pk}]$-modules and the summands are all indecomposable. Since $P^\gamma\cong P^{\gamma'}$ if and only if $\gamma=\gamma'$, by considering the restriction to $\sym{n-pk}$, we have, for any fixed $\gamma$ such that $s_\gamma\neq 0$, $\bigoplus_{\nu\in\pReg(k)}s_\nu s_\gamma Z_{j-1,\nu}\cong \bigoplus_{\nu\in\pReg(k)}s_\nu s_\gamma Z_{j'-1,\nu}$, i.e., \[Z_{j-1}=\bigoplus_{\nu\in\pReg(k)}s_\nu Z_{j-1,\nu}\cong \bigoplus_{\nu\in\pReg(k)}s_\nu Z_{j'-1,\nu}=Z_{j'-1}.\] By Theorem \ref{T: LT main}, we get the desired result.
\end{proof}

\section{Pivots and $q$-segmentations of Words}\label{S: pivot}

Our next focus is the module $\Xi^qR\sym{n}$. We aim to establish a free basis for the module and study its structure in the next section. This section sets up the necessary combinatorics. More precisely, for each $n\in\NN$ and $q\vDash n$, we define maps $\Phi:\sym{n}\to\sym{n}$ and  $\Upsilon_q:\sym{n}\to\sym{n}$ (see Definitions \ref{D: Phi} and \ref{D: Upsilon}), which will help us to partition $\sym{n}$ for the construction of a free basis for $\Xi^qR\sym{n}$.

Throughout this section, a word $w$ is a word in $\NN$ with distinct alphabets. We begin with a definition.

\begin{defn} Let $w=w_1\ldots w_n$ be a word in $\NN$ (with distinct alphabets) and $W=\{w_1,\ldots,w_n\}$. Define the numbers $\rho_0,\rho_1,\ldots,\rho_\ell$ inductively as follows. Let $\rho_0=0$. Suppose that $\rho_i\neq n$ for some $i\geq 0$. Let $\rho_{i+1}$ be the index such that $w_{\rho_{i+1}}$ is the smallest number in the set $W\backslash\{w_j:j\in [1,\rho_i]\}$. Continue in this manner until we get $\rho_\ell=n$ for some $\ell$.
\begin{enumerate}[(i)]
\item The pivots of $w$ are the alphabets $w_{\rho_1},\ldots,w_{\rho_\ell}$.
\item The pivot words and pivot cycles of $w$ are the subwords $w^{(i)}=w_{1+\rho_{i-1}}w_{2+\rho_{i-1}}\cdots w_{\rho_i}$ and the cycles $(w_{1+\rho_{i-1}},w_{2+\rho_{i-1}},\ldots,w_{\rho_i})$ respectively, one for each $i\in [1,\ell]$.
\item The pivot cycle type of $w$ is the sequence \[(\rho_1,\rho_2-\rho_1,\ldots,\rho_\ell-\rho_{\ell-1})=(|w^{(1)}|,\ldots,|w^{(\ell)}|)\vDash n.\]
\end{enumerate}
\end{defn}

\begin{eg}\label{Eg: S7a} Consider \[w=56\underline{1}34\underline{2}\hspace{.02cm}\underline{7}\in\sym{7}.\] The pivots of $w$ are $1,2,7$ as underlined. 
The pivot words of $w$ are $w^{(1)}=561$, $w^{(2)}=342$ and $w^{(3)}=7$, and its pivot cycle type is $(3,3,1)$.
\end{eg}

Recall the definition of reverse colexicographic order $\wleq$ on $\sym{n}$ in Subsection \ref{SS: generality}.

\begin{lem}\label{L: least permutation} Let $w(1),w(2),\ldots,w(k)$ be words such that $v=w(1)w(2)\cdots w(k)\in\sym{n}$, each $w(i)$ has a unique pivot, i.e., the least alphabet is at the last position of $w(i)$, and the pivots are in increasing order. Then $v$ is, with respect to $\wleq$, the least permutation containing $w(1),\ldots,w(k)$ as subwords.
\end{lem}
\begin{proof} There is nothing to prove if $k=1$. Assume that $k>1$. Let $c_i=|w(i)|$ so that $w(i)_{c_i}$ is the pivot of $w(i)$. Clearly $v$ is a permutation containing all of $w(1),w(2),\ldots,w(k)$ as subwords. Let $u$ be another such permutation where $u\wleq v$. By the definition of subword, for any $j\in [1,k]$, for any $i\in [1,c_j-1]$, the alphabet $w(j)_i$ in $u$ must be in an earlier position than the alphabet $w(j)_{i+1}$. Therefore, in particular, $u_n=w(j)_{c_j}$ for some $j\in [1,k]$. Since $u\wleq v$, we have $u_n\geq v_n=w(k)_{c_k}$ and thus $j=k$. Let $s$ be the largest index such that $u_s=w(j')_{c_{j'}}$ for some $j'\neq k$. In fact, $s$ is the largest position in $u$ such that $u_s$ is not an alphabet in $w(k)$. As such, $|n-s|\leq |w(k)|$ and $u_i=v_i$ for all $i\in [s+1,n]$, i.e., \[u_{s+1}\cdots u_n=w(k)_{c_k-n+s+1}\cdots w(k)_{c_k}=v_{s+1}\cdots v_n\] and  $u_s<w(k)_i$ for any $i\in [1,c_k]$. Since $u\wleq v$, we have $u_s\geq v_s$. If $c_k-n+s+1>1$ then \[w(k)_{c_k-n+s}>u_s\geq v_s=w(k)_{c_k-n+s},\] which is absurd. Therefore $u=u'w(k)$ for some word $u'$. Now we have $u'w(k)\wleq w(1)\ldots w(k-1)w(k)$. Repeating the argument by considering $u_{n-c_k}=u'_{n-c_k}=w(j)_{c_j}$ for some $j\in [1,k-1]$, we get $u=u''w(k-1)w(k)$ for some $u''$. Continuing in this manner, we get $u=v$ and the proof is now complete.
\end{proof}

\begin{defn}\label{D: Phi} Fix $n\in\NN$. We define the function $\Phi:\sym{n}\to\sym{n}$ \index{$\Phi$}as follows. Let $\sigma\in\sym{n}$. We write $\sigma$ uniquely as a product of disjoint cycles \[\sigma=(w_1,\ldots,w_{\rho_1})(w_{1+\rho_1},\ldots,w_{\rho_2})\cdots (w_{1+\rho_{\ell-1}},\ldots,w_{\rho_\ell})\] such that $w_{\rho_1}<\cdots<w_{\rho_\ell}$ and, for each $i\in [1,\ell]$, $w_{\rho_i}$ is the smallest number in the cycle containing it. Define \[\Phi(\sigma)=w_1w_2\cdots w_n\] by removing all parenthesis and commas.
\end{defn}

We give an example to illustrate the definition.

\begin{eg}\label{Eg: S7b} Let $\sigma=(5,6,1)(3,4,2)(7)$ so that $w=\Phi(\sigma)=5613427$. Notice that $w$ has cycle type $(4,2,1)$ which is different from the cycle type of $\sigma$ (or the pivot cycle type of $w$ as in Example \ref{Eg: S7a}).
\end{eg}

We have the following immediate lemma.

\begin{lem}\label{L: Phi bijective} The map $\Phi:\sym{n}\to\sym{n}$ is bijective.
\end{lem}
\begin{proof} The product of the pivot cycles of $w\in \sym{n}$ gives rise to a permutation $\sigma$ such that $\Phi(\sigma)=w$. As such, we get the inverse map.
\end{proof}

Following the definitions above, we are now ready to define the set which is central in our study for this and the next sections.

\begin{defn}\label{D: Base q} Let $q\vDash n$. We define
\begin{align*}
\Base_q&=\{\Phi(\sigma):\text{$\sigma\in\sym{n}$ has cycle type a weak refinement of $q$}\}\\
&=\{w:\text{$w\in\sym{n}$ has pivot cycle type a weak refinement of $q$}\}.\index{$\Base_q$}
\end{align*} It has cardinality the number permutations in $\sym{n}$ with cycle types weakly refining $q$ (cf. Lemma \ref{L: Phi bijective}), i.e., \[|\Base_q|=\sum_{\P(n)\ni\lambda\wref q}\frac{n!}{\lambda?}.\]
\end{defn}

Next, for each $q\vDash n$, we define a map $\Upsilon_q$. The maps are crucial in this section in the sense that they give rise to another interpretation of the sets $\Base_q$ and partition $\sym{n}$ into a disjoint union.

\begin{defn}\label{D: Upsilon} Let $q\vDash n$ and $w=w_1\ldots w_n\in \sym{n}$. The $q$-segments of $w$ are the words \[w_{q,i}=w_{1+q_{i-1}}w_{2+q_{i-1}}\cdots w_{q_i}\] one for each $i\in [1,\ell(q)]$. For each $q$-segment $w_{q,i}$, consider the pivot words $w_{q,i}^{(1)},\ldots,w_{q,i}^{(k_i)}$ of $w_{q,i}$ for some $k_i$. We define the map $\Upsilon_q:\sym{n}\to\sym{n}$ \index{$\Upsilon_q$}where $\Upsilon_q(w)$ is, with respect to $\wleq$, the least permutation containing all the subwords $w_{q,i}^{(j)}$'s; equivalently, let $k=\sum_{i=1}^{\ell(q)}k_i$ and $w(1),w(2),\ldots,w(k)$ be the rearrangements of the subwords $w_{q,1}^{(1)},\ldots,w_{q,\ell(q)}^{(k_{\ell(q)})}$ in the increasing order with respect to their respective unique pivots, by Lemma \ref{L: least permutation}, we have \[\Upsilon_q(w)=w(1)w(2)\cdots w(k).\]
\end{defn}


By the construction, notice that $\Upsilon_q(w)$ has pivot cycle type $(|w(1)|,\ldots,|w(k)|)$ a weak refinement of $q$. We should help the reader with a few examples below.

\begin{eg} Consider $q=(4,3)$ and $w=5613427$ as in Example \ref{Eg: S7a}. The $q$-segments of $w$ are $w_{q,1}=5613$ and $w_{q,2}=427$. The corresponding pivots words are $w_{q,1}^{(1)}=561$, $w_{q,1}^{(2)}=3$, $w_{q,2}^{(1)}=42$ and $w_{q,2}^{(2)}=7$. Therefore, $\Upsilon_q(w)=w_{q,1}^{(1)}w_{q,2}^{(1)}w_{q,1}^{(2)}w_{q,2}^{(2)}=5614237$ and it has pivot cycle type $(3,2,1,1)$ which is a weak refinement of $(4,3)$. Notice that $u=5436127$ is also a word with the subwords $w_{q,i}^{(j)}$'s but $5614237\swleq 5436127$.
\end{eg}

\begin{eg} We get $\Upsilon_{(1^n)}(w)=12\cdots n$ for any $w\in\sym{n}$. Also, we get $\Upsilon_q(12\cdots n)=12\cdots n$ for any $q\vDash n$.
\end{eg}

\begin{eg}\label{Eg: (2,2)} Let $q=(2,2)$. In the figure below, we draw an arrow $u\rightarrow v$ if $\Upsilon_q(u)=v$ for $u,v\in\sym{4}$ but ignore the arrow if $\Upsilon_q(u)=u$ and embolden all images of $\Upsilon_q$ . For instance, we ignore the arrow from 1234 to 1234 as the previous example shows.
\bigskip
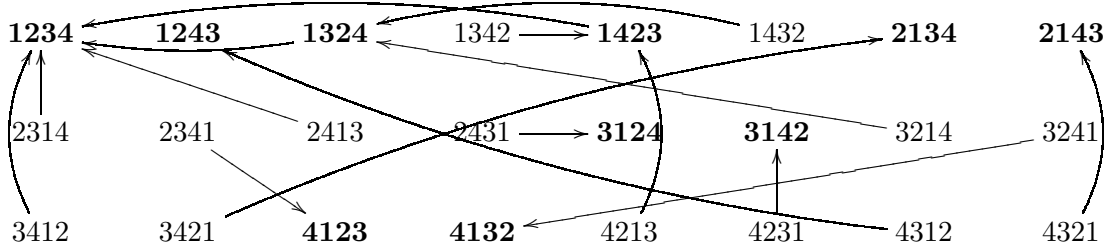
\[{\xymatrix{\mathbf{1234}& \mathbf{1243}&\mathbf{1324}\ar@/^/[ll]&1342\ar[r]&\mathbf{1423}\ar@/_1pc/[llll]&1432\ar@/_1pc/[lll]& \mathbf{2134}&\mathbf{2143}\\ 2314\ar[u]&2341\ar[rd]&2413\ar[llu]&2431\ar[r]&\mathbf{3124}&\mathbf{3142}&3214\ar[llllu]&3241\ar[lllld]\\
3412\ar@/^1pc/[uu]&3421\ar@/^1pc/[rrrrruu]&\mathbf{4123}&\mathbf{4132}&4213\ar@/_1pc/[uu]& 4231\ar[u]&4312\ar@/^1pc/[llllluu]&4321\ar@/_1pc/[uu]}}\]
 \smallskip Observe that $\Base_{(2,2)}=\im\Upsilon_{(2,2)}$.
\end{eg}

Indeed, the final equality we obtained in Example \ref{Eg: (2,2)} is not a coincidence. We have the following lemma.

\begin{lem}\label{L: B = U} Let $q\vDash n$. Then $\Base_q=\im\Upsilon_q$.
\end{lem}
\begin{proof} Let $w\in\sym{n}$. By definition, $\Upsilon_q(w)=w(1)w(2)\cdots w(k)$ with pivot cycle type $c=(|w(1)|,|w(2)|,\ldots,|w(k)|)$. Since $w(1),\ldots,w(k)$ are the rearrangements of the pivot words of the $q$-segments of $w$, we get that $c$ is a weak refinement of $q$. Therefore $\im\Upsilon_q\subseteq \Base_q$.

Suppose now that $w$ has pivot cycle type a weak refinement of $q$. Let $w^{(1)},\ldots,w^{(\ell)}$ be the pivot words of $w$. By assumption, there is a rearrangement of the pivot words of $w$, say $w(1),\ldots,w(\ell)$, and some $j_0=0<j_1<\cdots<j_{\ell(q)}$ such that, for each $i\in [1,\ell(q)]$, \[q_i=|w(1+j_{i-1})|+|w(2+j_{i-1})|+\cdots+|w(j_i)|\] and the pivots of $w(1+j_{i-1}),w(2+j_{i-1}),\ldots, w(j_i)$ are in increasing order. Let $u=w(1)\cdots w(\ell)$. By the construction of $u$, the $q$-segments of $u$ are the subwords $w(1+j_{i-1}),w(2+j_{i-1}),\ldots, w(j_i)$ one for each $i\in[1,\ell(q)]$ and the pivot words of the $q$-segments of $u$ are precisely $w^{(1)},\ldots,w^{(\ell)}$.  Since the pivots of $w^{(1)},\ldots,w^{(\ell)}$ are increasing, we have $\Upsilon_q(u)=w$. The proof is now complete.
\end{proof}

In view of the proof of the previous lemma, we define the following set.

\begin{defn}\label{D: compatible} Let $q\vDash n$, $w\in\Base_q$ and $w^{(1)},\ldots,w^{(\ell)}$ be the pivot words of $w$. A rearrangement $(t_1,\ldots,t_\ell)$ of the ordered set $[1,\ell]$ is called $(q,w)$-compatible if there exist $j_0=0<j_1<\cdots<j_{\ell(q)}=\ell$ (a subsequence of $[1,\ell]$) such that, for each $i\in [1,\ell(q)]$, $w^{(t_{1+j_{i-1}})},w^{(t_{2+j_{i-1}})},\ldots, w^{(t_{j_i})}$ are precisely the pivot words of the word \[w(i)=w^{(t_{1+j_{i-1}})}w^{(t_{2+j_{i-1}})}\cdots w^{(t_{j_i})}\] (or equivalently, $t_{1+j_{i-1}}<t_{2+j_{i-1}}<\cdots<t_{j_i}$) and $q_i=|w(i)|$. We define the subset \[\F_q(w)=\{w^{(t_1)}\cdots w^{(t_\ell)}:\text{$(t_1,\ldots,t_\ell)$ is $(q,w)$-compatible}\}.\index{$\F_q$}\]
\end{defn}

\begin{eg} Let $w=416235$ and $q=(3,2,1)$. The pivot words are $w^{(1)}=41$, $w^{(2)}=62$, $w^{(3)}=3$ and $w^{(4)}=5$ and the $(q,w)$-compatible rearrangements are \[ (1,3,2,4), (2,3,1,4), (1,4,2,3), (2,4,1,3). \] Therefore $\F_q(w)=\{413625,623415,415623,625413\}$.
\end{eg}

\begin{eg} In Example \ref{Eg: (2,2)}, we see that $\F_{(2,2)}(1324)=\{1432,3214\}=\Upsilon_{(2,2)}^{-1}(1324)$. Notice that $1324\not\in \F_{(2,2)}(1324)$.
\end{eg}

As mentioned earlier, we can now partition $\sym{n}$ into disjoint unions.

\begin{cor}\label{C: fiber disjoint union} For each $q\vDash n$ and $v\in\Base_q$, we have $\F_q(v)=\Upsilon_q^{-1}(v)$. Furthermore, we have a disjoint union \[\sym{n}=\bigsqcup_{v\in \Base_q}\F_q(v).\]
\end{cor}
\begin{proof} Fix a composition $q\vDash n$. Suppose that $\Upsilon_q(w)=v$. Let
\[w_{q,1}^{(1)},\ldots,w_{q,1}^{(\rho_1)},\ldots,w_{q,\ell(q)}^{(1)},\ldots,w_{q,\ell(q)}^{(\rho_{\ell(q)})}\] be the sequence of the pivot words obtained as in Definition \ref{D: Upsilon}. By the definition, it is a rearrangement of the sequence of pivot words $v^{(1)},\ldots, v^{(\rho)}$ of $v$ where $\rho=\sum \rho_i$ and let $t:=(t_1,\ldots,t_\rho)$ denote such rearrangement. The rearrangement $t$ of $[1,\rho]$ is $(q,v)$-compatible where, in Definition \ref{D: compatible}, we have, for each $s\in [1,\ell(q)]$, $j_s=\sum_{i=1}^s\rho_i$ and the pivot words of $w_{q,i}$ are precisely $w_{q,i}^{(1)},\ldots,w_{q,i}^{(\rho_i)}$ by definition. Therefore $w\in \F_q(v)$. Conversely, suppose that $w\in \F_q(v)$, i.e., $w=v^{(t_1)}\cdots v^{(t_\ell)}$ for some $(q,v)$-compatible rearrangement of $[1,\ell]$ with the existence of the $j_0=0<j_1<\cdots<j_{\ell(q)}=\ell$. By Definition \ref{D: Upsilon}, for each $i\in [1,\ell(q)]$, we have $\rho_i=j_i-j_{i-1}$, \[w_{q,i}=v^{(1+t_{j_{i-1}})}v^{(2+t_{j_{i-1}})}\cdots v^{(t_{j_i})}\] and the pivot words of $w_{q,i}$ are precisely $v^{(1+t_{j_{i-1}})},v^{(2+t_{j_{i-1}})},\ldots, v^{(t_{j_i})}$. It is clear that the least word containing the subwords $v^{(1)},\ldots,v^{(\ell)}$ is $v$. So $\Upsilon_q(w)=v$.

We now prove the second assertion. For any $w\in\sym{n}$, we have $v=\Upsilon_q(w)\in \Base_q$ by Lemma \ref{L: B = U}. So $w\in \Upsilon_q^{-1}(v)$. For different $v,v'\in \Base_q$, the fibers $\Upsilon_q^{-1}(v),\Upsilon_q^{-1}(v')$ are clearly disjoint.
\end{proof}

To conclude the section, we give a proposition concerning the size of $\F_q(v)$ as an independent interest.

\begin{prop} Let $q\vDash n$ and suppose that $v\in\Base_q$ with pivot cycle type $r$. Then $|\F_q(v)|=|\Upsilon^{-1}_q(v)|\geq \facmulti{q}$ and with equality if $r\approx q$.
\end{prop}
\begin{proof} Let $w\in\Upsilon^{-1}_q(v)$ and let $w_{q,i}$'s be the $q$-segments of $w$ as in Definition \ref{D: Upsilon}. Notice that $|w_{q,i}|=q_i$. Permute the $q$-segments of $w$ of the same sizes yields $\facmulti{q}$ different words $u$ with the same $q$-segments as $w$ and therefore $\Upsilon_q(u)=v$ by definition. So $\facmulti{q}\leq |\Upsilon^{-1}_q(v)|$. Now assume that $r\approx q$. If $\Upsilon^{-1}_q(v)$ contains $w$ such that $w_{q,i}$ is not its own pivot word, then, by definition, $v=\Upsilon_q(w)$ has pivot cycle type $\swref{r}{q}$, contradicting $r\approx q$. Therefore, $\Upsilon^{-1}_q(v)$ contains nothing else except the permutation of the $q$-segments we have mentioned earlier.
\end{proof}

\section{The module $\Xi^qR\sym{n}$ and its structure}\label{S: Xi module}

Throughout this section, a word $w$ is again a word in $\NN$ with distinct alphabets and, for each composition $q\in\C(n)$, we denote \[X^q_R=\Xi^q R\sym{n}\index{$X^q_R$}\] where $\Xi^q\in R\sym{n}$ is the Solomon's descent element labelled by $q$ as in Subsection \ref{SS: descent}.

In this section, we study the modules $X^q_R$'s. We shall give a free basis for $X^q_R$ and study the structure of $X^q_F$ both in the ordinary and modular cases. It turns out that $X^q_F$'s are closely related to the projective modules generated by the modular idempotents of the descent algebra in Section \ref{S: mod idem} and the higher Lie modules in Section \ref{S: higher lie}.  As always expected, the modular case is significantly more difficult than the ordinary case and we could only take a glance at it. We split this section into three subsections. In the first subsection, we show that the right $R\sym{n}$-module $X^q_R$ has a free $R$-basis as in Theorem \ref{T: basis Xiq}. In the second part, we study the irreducible characters of $X^q_F$ when $p=0$. For the third part, we study the case when $p>0$.

\medskip
We begin with the following basic observation.

\begin{lem}\label{L: XiMod surj Lie} Let $q,r$ be compositions of $n$ such that $r\approx q$. Then $X^r_R\cong X^q_R$ and there is a surjection $X^r_R\twoheadrightarrow \Lie_R(q)$ given by the left multiplication by $\omega^q\sigma$ where $\sigma\Xi^r=\Xi^q$.
\end{lem}
\begin{proof} By Equation \ref{Eq: tau_r}, we get $\tau_r\Xi^r=\Xi^q$ and the isomorphism. Now the lemma follows by the definition $\omega_q=\omega^q\Xi^q$.
\end{proof}

\subsection{A free basis}\label{SS: Xiq basis}

We now introduce a few notations we shall be using throughout this subsection. For an element $\alpha\in R\sym{n}$, we write $\supp(\alpha)$ \index{$\supp$}for the subset of $\sym{n}$ consisting of the words or permutations appearing in $\alpha$ with nonzero coefficient. For each $q\vDash n$, corresponding to the set $\Base_q$ in Definition \ref{D: Base q}, we let \[\B_q=\{\Xi^qw:w\in\Base_q\}\subseteq X^q_R.\index{$\B_q$}\] The main result is the following free basis for $X^q_R$.

\begin{thm}\label{T: basis Xiq} As $R$-module, $X^q_R$ has an $R$-basis $\B_q$, and \[|\B_q|=\sum_{\P(n)\ni\lambda\wref q}\frac{n!}{\lambda?},\] i.e., the number of permutations of $\sym{n}$ with cycle types weakly refining $q$. 
\end{thm}

To show the linear independence part for Theorem \ref{T: basis Xiq}, we need the following lemma. 

\begin{lem}\label{L: supp} Let $v\in\sym{n}$ and suppose that $w\in\supp(\Xi^qv)$. Then $\Upsilon_q(w)\wleq v$.
\end{lem}
\begin{proof} Let $u=\Upsilon_q(w)$. By definition, we have $w=u^{(i_1)}\cdots u^{(i_\ell)}$ where $u^{(1)},\ldots,u^{(\ell)}$ are the pivot words of $u$ and $\{i_1,\ldots,i_\ell\}=[1,\ell]$. On the other hand, since $\Xi^q$ is a sum of words associated to the row standard
$q$-tableaux (with respect to the numbers $1,2,\ldots,n$), $\Xi^qv$ is a sum of words associated to the row standard $q$-tableaux with respect to the order $v_1<\cdots<v_n$ where $v=v_1v_2\cdots v_n$. As $w\in\supp(\Xi^qv)$, we get that $v$ is a permutation containing the subwords $u^{(i_1)},\ldots,u^{(i_\ell)}$. By definition (see Definition \ref{D: Upsilon} and Lemma \ref{L: least permutation}), $u$ is the least of such permutations and therefore $u\wleq v$.
\end{proof}

Recall the notations $Q_S$ and $s_\mathbf{j}$ introduced in Subsection \ref{SS: descent}. To motivate the next lemma and the proof of the spanning part of Theorem \ref{T: basis Xiq}, we give an example. Let $w=231$. We reverse the alphabets in $w$ and consider \[Q_{132}=Q_{\{1,3,2\}}=132-312-213+231.\] Notice that $u\wleq 231$ for any word $u$ appearing in $Q_{132}$. Recall that, using Theorem \ref{T: GR 2.1}, $\Xi^{(2,1)}Q_{132}=0$ since $(3)\not\wref (2,1)$. Therefore, \[\Xi^{(2,1)}231= \Xi^{(2,1)}312+\Xi^{(2,1)}213-\Xi^{(2,1)}132.\] 

\begin{lem}\label{L: max element} Let $w^{(1)},\ldots, w^{(\ell)}$ be the pivot words of $w\in\sym{n}$. For each $i\in[1,\ell]$, let $S_i=\{w^{(i)}_{c_i},\ldots,w^{(i)}_2,w^{(i)}_1\}$ be the ordered set with respect to the pivot word $w^{(i)}=w^{(i)}_1\ldots w^{(i)}_{c_i}$, i.e., $S_i$ reverses the alphabets in $w^{(i)}$. Then the largest word, with respect to the reverse colexicographic order $\wleq$, appearing in $Q_{S_1}\cdots Q_{S_\ell}\in R\sym{n}$ is $w$ with coefficient $(-1)^{\ell+\sum_{i=1}^\ell c_i}$.
\end{lem}
\begin{proof} Let us use $\overline{w^{(i)}}=w^{(i)}_{c_i}\ldots w^{(i)}_1$ and let $u$ be a word appearing in $Q_{S_1}\cdots Q_{S_\ell}$. Then it is of the form $u=u(1)\cdots u(\ell)$ where, for each $i\in[1,\ell]$, $u(i)$ is a word appearing in \[Q_{S_i}=\sum (-1)^{|\mathbf{j}|}s_{\mathbf{j}}\cdot \overline{w^{(i)}}=\sum (-1)^{|\mathbf{j}|}w^{(i)}_{c_i-j_s+1}\ldots w^{(i)}_{c_i-j_1+1}w^{(i)}_{c_i-k_1+1}\ldots w^{(i)}_{c_i-k_{n-s}+1}\] where both sums are taken over all $\mathbf{j}=\{j_1<\cdots<j_s\}\sqsubseteq [2,c_i]$ where $\{1=k_1<\cdots<k_{n-s}\}\sqcup \mathbf{j}=[1,c_i]$. Suppose that $u\wgeq w$.  In particular, we have \[S_\ell\ni u(\ell)_{c_\ell}=u_n\leq w_n=w^{(\ell)}_{c_\ell}.\] Since $w^{(\ell)}$ is a pivot word, we have $w^{(\ell)}_{c_\ell}=\min S_\ell$ and therefore $u(\ell)_{c_\ell}=w^{(\ell)}_{c_\ell}$. Consider the particular case $i=\ell$. The only $\mathbf{j}\sqsubseteq [2,c_\ell]$ such that $s_{\mathbf{j}}\cdot \overline{w^{(\ell)}}$ ends with $w^{(\ell)}_{c_\ell}$ is $\mathbf{j}=[2,c_\ell]$. Therefore, $u(\ell)=s_{[2,c_\ell]}\cdot \overline{w^{(\ell)}}=w^{(\ell)}$. We now work by reverse induction to conclude that $u(i)=w^{(i)}$ for all $i\in[1,\ell]$. Notice that, for each $i\in [1,\ell]$, the coefficient of $u(i)=s_{[2,c_i]}\cdot \overline{w^{(i)}}=w^{(i)}$ in $Q_{S_i}$ is $(-1)^{c_i-1}$.
\end{proof}

We are now ready to prove the theorem.

\begin{proof}[Proof of Theorem \ref{T: basis Xiq}] Let $T$ be the $R$-linear span of the set $\B_q$. We argue by induction on the reverse colexicographic order to show that $\Xi^qw$ lies in $T$ for any word $w\in\sym{n}$. If $w=12\ldots n$ then $w\in \Base_q$ because the pivot cycle type of $w$ is $(1^n)$ and hence $\Xi^qw\in T$. Let $w$ be a word in $\sym{n}$. If $w\in \Base_q$ then $\Xi^qw\in T$. Suppose that $w\not\in \Base_q$. Let $w=w^{(1)}\cdots w^{(\ell)}$ be written in the pivot words. For each $i\in [1,\ell]$, let $S_i$ be the ordered set with respect to $w^{(i)}$ as in Lemma \ref{L: max element}. Since $w\not\in \Base_q$, i.e., the composition $(|S_1|,\ldots,|S_\ell|)$ is not a weak refinement of $q$, by Theorem \ref{T: GR 2.1}, we have \[\Xi^qQ_{S_1}\cdots Q_{S_\ell}=0.\] By Lemma \ref{L: max element}, we can rewrite \[\Xi^qw=\Xi^q(w\pm Q_{S_1}\cdots Q_{S_\ell})\] where the right hand side is a sum of elements of the form $\Xi^qu$ such that $u\swleq w$. By induction, $\Xi^qw\in T$. This shows that $T=X^q_R$.

We order $\B_q$ according to the reverse colexicographic on $\Base_q$. By Lemma \ref{L: supp}, for any $v\in \Base_q$, any word $w\in \supp(\Xi^qv)$ satisfies $\Upsilon_q(w)\wleq v$, i.e., $\supp(\Xi^qv)$ contains words in $\F_q(u)$ with $\Base_q\ni u\wleq v$. Clearly, $v\in \supp(\Xi^qv)$. Therefore $\B_q$ is linearly independent over $R$ using Corollary \ref{C: fiber disjoint union}.
\end{proof}

\subsection{The ordinary case}

For this subsection, we assume throughout that $p=0$. We study the decomposition problem of $X^q_F$ in the ordinary case.  More precisely, we study the module $X^q_F$ and write it as the direct sum of higher Lie modules (see Theorem \ref{T: Xi as higher Lie} and cf. Theorem \ref{T: basis Xiq}).  As the irreducible constituents of the higher Lie modules have been described by Schocker in Theorem \ref{T: Schoc Lie(q)}, we get the complete knowledge about $X^q_F$ in this case.

We remind the readers about the elements $I_q$'s and $E_\lambda$'s  in Subsection \ref{SS: ord idem}.

\begin{lem}\label{L: E isom} Let $\lambda\in\P(n)$ and $\alpha=\sum_{\lambda(q)=\lambda} c_qI_q$ for some $c_q\in F$ such that $\sum_{\lambda(q)=\lambda}c_q\neq 0$. Then $E_\lambda F\sym{n}\cong \alpha F\sym{n}$ as $F\sym{n}$-modules. In particular, for $q\in\C(n)$ such that $\lambda(q)=\lambda$, we have $E_\lambda F\sym{n}\cong I_qF\sym{n}$.
\end{lem}
\begin{proof} Define $\phi:E_\lambda F\sym{n}\to \alpha F\sym{n}$ as $\phi(E_\lambda\gamma)=\alpha E_\lambda\gamma$ for any $\gamma\in F\sym{n}$. Since the map $\phi$ is defined by multiplying $\alpha$ on the left, $\phi$ is a right $F\sym{n}$-module homomorphism. By Theorem \ref{T: EI=0}(ii), \[\alpha E_\lambda\gamma=\left (\sum_{\lambda(q)=\lambda} c_qI_q\right )E_\lambda\gamma=\left (\sum_{\lambda(q)=\lambda} c_qI_q\right )\gamma=\alpha\gamma.\] Therefore $\phi$ is surjective. Suppose that $\phi(E_\lambda\gamma)=0$, i.e., $\alpha\gamma=0$. Multiplying on the left by $E_\lambda$, by Theorem \ref{T: EI=0}(iii), we have
\begin{align*}
  0&=E_\lambda\alpha\gamma=E_\lambda\left (\sum_{\lambda(q)=\lambda} c_qI_q\right )\gamma=\left (\facmulti{\lambda}\sum_{\lambda(q)=\lambda} c_q\right )E_\lambda\gamma.
\end{align*} So $E_\lambda\gamma=0$. This shows that $\phi$ is injective. The last assertion now follows.
\end{proof}

\begin{lem}\label{L: I direct sum} For each partition $\lambda$ of $n$, let $\alpha_\lambda=\sum_{\lambda(q)=\lambda} c_{\lambda,q}I_q$ for some $c_{\lambda,q}\in F$ such that $\sum_{\lambda(q)=\lambda}c_{\lambda,q}\neq 0$ whenever $\alpha_\lambda\neq 0$. The sum $\sum_{\lambda\vdash n}\alpha_\lambda F\sym{n}$ in $F\sym{n}$ is direct.
\end{lem}
\begin{proof} Suppose that
\begin{equation}\label{Eq: 1} \sum_{\lambda\vdash n}\alpha_\lambda\gamma_\lambda=0
\end{equation} for some $\gamma_\lambda\in F\sym{n}$, one for each $\lambda\vdash n$. We want to prove that $\alpha_\lambda\gamma_\lambda=0$ for all partitions $\lambda$. We argue by reverse induction with respect to the weak refinement on $\P(n)$. Let $\mu=(1^n)$. Multiply $E_\mu$ on the left of Equation \ref{Eq: 1}, by Theorems \ref{T: GR 4.1}(iii) and \ref{T: EI=0}(iii), since $I_\mu=n!E_\mu=\Xi^\mu$, we have \[0=E_\mu c_{\mu,\mu} I_\mu\gamma_\mu=c_{\mu,\mu} I_\mu\gamma_\mu=\alpha_\mu\gamma_\mu.\] Now fix an arbitrary $\mu\in\P(n)$ and suppose that, for any partition $\xi\neq \mu$ such that $\xi$ is a weak refinement of $\mu$, we have $\alpha_\xi\gamma_\xi=0$. Multiply $E_\mu$ on the left of Equation \ref{Eq: 1}, by Theorems \ref{T: GR 4.1}(iii) and \ref{T: EI=0}(iii), we have
\begin{align*}
0=E_\mu\sum_{\lambda\vdash n}\alpha_\lambda\gamma_\lambda=\sum_{\text{$\lambda\wref\mu$}} E_\mu\alpha_\lambda\gamma_\lambda=E_\mu\alpha_\mu\gamma_\mu&=\sum_{\lambda(q)=\mu}c_{\mu,q}E_\mu I_q\gamma_\mu=\left (\sum_{\lambda(q)=\mu}\facmulti{\mu}c_{\mu,q}\right )E_\mu\gamma_\mu.
\end{align*} Therefore, since $\sum_{\lambda(q)=\mu}c_{\mu,q}\neq 0$, we have $E_\mu\gamma_\mu=0$ and hence, by Theorem \ref{T: EI=0}(ii), \[\alpha_\mu\gamma_\mu=\sum_{\lambda(q)=\mu} c_{\mu,q}I_q\gamma_\mu=\sum_{\lambda(q)=\mu} c_{\mu,q}I_qE_\mu\gamma_\mu=0.\] The proof is now complete.
\end{proof}

We are now ready to state and prove the main result in this subsection.

\begin{thm}\label{T: Xi as higher Lie} Let $q$ be a composition of $n$. We have $X^q_F\cong \bigoplus_{\P(n)\ni\lambda\wref q}\Lie_F(\lambda)$.
\end{thm}
\begin{proof} For each $n\dashv\lambda\wref q$, let
\begin{align*}
  \alpha_\lambda&=\sum_{r\sref q,\ \lambda(r)=\lambda}\frac{1}{\k!(r,q)}I_r
\end{align*} and $W=\sum_{n\dashv\lambda\wref q} \alpha_\lambda F\sym{n}$. Clearly, $\sum_{r\sref q, \lambda(r)=\lambda}\frac{1}{\k!(r,q)}\neq 0$. We first prove that $X^q_F=W$. By Theorem \ref{T: GR 3.4}, since $\Xi^q=\sum_{n\dashv\lambda\wref q}\alpha_\lambda$, we have $X^q_F\subseteq W$. We now check that $W\subseteq X^q_F$. It suffices to show that $\alpha_\lambda\in X^q_F$. Fix now $\lambda\wref q$. By Theorem \ref{T: GR 4.1}, when $s\wref q$ and $\lambda(s)=\lambda$, we have $\Xi^qI_s=\sum_{r\sref q,\ \lambda(r)=\lambda} d_{r,s,q}I_r$ for some suitable integers $d_{r,s,q}$. Therefore \[\Xi^qE_\lambda=\frac{1}{\ell(\lambda)!}\sum_{\lambda(s)=\lambda}\Xi^qI_s=\sum_{\lambda(r)=\lambda} d_{\lambda,q}I_r\] for some suitable integers $d_{\lambda,q}$. Furthermore, using Theorem \ref{T: GR 4.1}(i) and the definition of $E_\lambda$, we have
\begin{align}\label{Eq: 6.9}
\Xi^q=\Xi^q\cdot \sum_{\lambda\vdash n}E_\lambda=\sum_{n\dashv\lambda\wref q}\Xi^qE_\lambda=\sum_{n\dashv\lambda\wref q}\sum_{\lambda(r)=\lambda} d_{\lambda,q}I_r.
\end{align} By Theorem \ref{T: GR 3.4}, $\{I_r:r\vDash n\}$ forms a basis for $\Des{n}{F}$ and compare Equation \ref{Eq: 6.9} with $\Xi^q=\sum_{n\dashv \lambda\wref q}\alpha_\lambda$, we conclude that \[\alpha_\lambda=\sum_{\lambda(r)=\lambda} d_{\lambda,q}I_r=\Xi^qE_\lambda\in X^q_F.\] By Lemmas \ref{L: E isom} and \ref{L: I direct sum}, the sum in $W$ is direct and $W\cong \bigoplus_{n\dashv \lambda\wref q}E_\lambda F\sym{n}$. The isomorphism in the statement now follows from Proposition \ref{P: E iso nu}.
\end{proof}

In the example below, we demonstrate how the various right ideals $\Xi^\lambda F\sym{4}$, where $\lambda\in\P(4)$, stack up one on top of the others.

\begin{eg} Let $n=4$. The numbers beside the modules $X^q_F$ are their dimensions and $\lambda_q$ denotes $\Lie_F(q)$.
\[\xymatrix{&\ \ \ \ \ X^{(4)}_F\ \ \text{\textcolor{black}{24}} &\\ \ \ \ \ \  X^{(3,1)}_F \ \ \text{\textcolor{black}{15}} \ar@{-}[ru]^{\lambda_{(4)}+\lambda_{(2,2)}}&&\ \ \ \ \ X^{(2,2)}_F \ \ \text{\textcolor{black}{10}} \ar@{-}[lu]_{\lambda_{(4)}+\lambda_{(3,1)}}\\ &\ \ \ \ \  X^{(2,1,1)}_F\ \ \text{\textcolor{black}{7}} \ar@{-}[ru]^{\lambda_{(2,2)}}\ar@{-}[lu]_{\lambda_{(3,1)}}&\\ &\ \ \ \ \  X^{(1^4)}_F \ \ \text{\textcolor{black}{1}}\ar@{-}[u]^{\lambda_{(2,1,1)}}&\\ &0\ar@{-}[u]^{\lambda_{(1^4)}}& }\]
\end{eg}

As a consequence of Theorems \ref{T: Xi as higher Lie} and  \ref{T: Schoc Lie(q)}, we have the following corollary.

\begin{cor} The multiplicity of the irreducible character $\zeta^\mu$ in $X^q_F$ is \[\sum_{\P(n)\ni\lambda\wref q}C^\mu_\lambda\] where $C^\mu_\lambda$ is the number given as in Theorem \ref{T: Schoc Lie(q)}.
\end{cor}

\subsection{The modular case} In this subsection, we assume throughout that $p>0$. As expected, since, for example, $X^{(n)}_F$ is the regular module, the decomposition into indecomposable summands and composition factors of such modules $X^q_F$'s are difficult to describe in general. In the main result in this subsection (see Theorem \ref{T: proj summand Xi}), we identify a projective summand of $X^q_F$ under the assumption that $q$ is coprime to $p$ and $\lambda(q)$ is $p$-regular.

We refer the reader to the proof of \cite[Proposition 7]{ES} for the simple modules $\{M_{\lambda,F}:\lambda\in\pReg(n)\}$ \index{${M}_{\lambda,F}$}of the descent algebra $\Des{n}{F}$ in the argument below.


\begin{lem}\label{L: dim Xie} Let $\mu\in\pReg(n)$ and suppose that $\varphi^{q,F}(\mu)\neq 0$. Then $\lambda(q)\in\pReg(n)$, $\mu\wref q$,  $\dim_F\Xi^qe_{\mu,F}F\sym{n}=|\ccl{\mu,p}|$ and $\Xi^qe_{\mu,F}F\sym{n}\cong e_{\mu,F}F\sym{n}$.
\end{lem}
\begin{proof} Since $\varphi^{q,F}(\mu)\neq 0$, we have $\mu\wref q$ by Lemma \ref{L: phi nonzero}. By Theorem \ref{T: Sol epi}, we have \[(c_{n,F}(\Xi^qe_{\mu,F}))(\lambda)=(\varphi^{q,F}\Char_{\mu,F})(\lambda)=\delta_{\lambda,\mu}\varphi^{q,F}(\mu).,\] As such, we have $\Xi^qe_{\mu,F}\not\in\ker c_{n,F}$ and therefore $\lambda(q)\in\pReg(n)$. For any $\lambda\in\pReg(n)$, $\Xi^q e_{\mu,F}$ acts on the simple left $\Des{n}{F}$-module $M_{\lambda,F}$ as multiplication by $\varphi^{q,F}(\mu)$ if $\lambda=\mu$ and annihilates it if $\lambda\neq \mu$.  Hence $\Xi^qe_{\mu,F}M_{\lambda,F}=\delta_{\lambda,\mu}M_{\lambda,F}$. Therefore, $\Xi^qe_{\mu,F}F\sym{n}$ has dimension the number of composition factors isomorphic to $M_{\mu,F}$ in $F\sym{n}$ which is the dimension of $e_{\mu,F}F\sym{n}$ as pointed out in the proof of \cite[Proposition 7]{ES}. The final isomorphism is now clear since the map $\psi:e_{\mu,F}F\sym{n}\to \Xi^qe_{\mu,F}F\sym{n}$ given by the left multiplication by $\Xi^q$ is surjective and both right ideals have the same dimension.
\end{proof}

We remark that, if $r$ is $p$-equivalent to $\mu\in\pReg(n)$, then $r\wref \mu$ (see \cite[\S2]{ES}). Therefore, under the hypothesis given in Lemma \ref{L: dim Xie}, we have that $\dim_F\Xi^qe_{\mu,F}F\sym{n}$ is also equal to the number of permutations with cycle types both $p$-equivalent to $\mu$ and a weak refinement of $q$ (with the second condition being superfluous in this case). We conjecture that the equality holds in general (see Conjecture \ref{Con: dim}) and the importance is illustrated in Conjecture \ref{Con: direct sum}.

We are now ready to state and prove the main result for this subsection.

\begin{thm}\label{T: proj summand Xi} Suppose that $(q,p)=1$ and $\lambda=\lambda(q)\in\pReg(n)$. Then we have isomorphism of modules \[\Xi^qe_{\lambda,F}F\sym{n}\cong \Lie_F(q)\cong e_{\lambda,F}F\sym{n}.\] In particular, $\Xi^qe_{\lambda,F}F\sym{n}$ is a projective summand of $X^q_F$ of dimension $|\ccl{\lambda,p}|=|\ccl{\lambda}|$.
\end{thm}
\begin{proof} The second isomorphism is given in Corollary \ref{C: Lie iso proj}. Let $f:\Xi^q e_{\lambda,F}F\sym{n}\to \omega_\lambda F\sym{n}$ be defined by the left multiplication with $\sigma\omega^q$ where $\sigma=\tau_q$ so that $\sigma\omega_q=\omega_\lambda$ (see Equation \ref{Eq: tau_r}). It is well-defined because \[f(\Xi^q e_{\lambda,F})=\sigma\omega^q\Xi^q e_{\lambda,F}=\sigma\omega_qe_{\lambda,F}=\omega_\lambda e_{\lambda,F}.\] Using Theorem \ref{T: BL results}(iii) and Corollary \ref{C: e * omega}, we have \[\lambda?\omega_\lambda=\omega_\lambda\omega_\lambda=\omega_\lambda e_{\lambda,F}\omega_\lambda.\] Since $\lambda$ is both $p$-regular and coprime to $p$, we have $\lambda?\neq 0$ in $F$ and therefore $f(\frac{1}{\lambda?}\Xi^qe_{\lambda,F}\omega_\lambda)=\omega_\lambda$. As such, $f$ is surjective and hence an isomorphism due to Lemma \ref{L: dim Xie} and Theorem \ref{T: basis Xiq}. The last assertion follows since $\Xi^qe_{\lambda,F}F\sym{n}$ is a projective (equivalently, injective) submodule of $X^q_F$.
\end{proof}

We now demonstrate the decomposition of $X^q_F$ with a few examples below.

\begin{eg} The module $X^{(2)}_F$ is the regular $F\sym{2}$-module. In the case when $p=2$, $\omega_{(2)}=\omega_{(1,1)}=12+21$ and therefore $\Lie_F((2))\cong F\cong \Lie_F((1,1))$. In the case when $p\neq 2$, we have $\Lie_F((2))\cong \sgn(2)$ and $\Lie_F((1,1))\cong F$. Therefore, \[X^{(2)}_F\cong \left \{\begin{array}{ll} \Lie_F((2))\oplus \Lie_F((1))&p\neq 2,\\ {\begin{bmatrix}\Lie_F((2))\\ \Lie_F((1,1))\end{bmatrix}}&p=2.\end{array}\right .\] In the $p=2$ case above, it means $X^{(2)}_F$ is indecomposable with its head and socle both isomorphic to the trivial module and the choice of $\Lie_F((2))$ on its top is due to Lemma \ref{L: XiMod surj Lie}.
\end{eg}

\begin{eg} Let $p=2$. The regular module $X^{(3)}_F$ has dimension strictly bigger than $\sum_{\lambda\in\P(3)}\Lie_F(\lambda)$ (see Appendix \ref{Appen C}). Therefore Theorem \ref{T: Xi as higher Lie} would not hold without the assumption $p=0$, even with the direct-sum condition relaxed and the extra assumption that $q$ is coprime to $p$.
\end{eg}

\begin{eg} 
Consider $V=X^{(2,1)}_F$. Notice that $\dim_FV=4$ and, following Theorem \ref{T: basis Xiq}, $V$ has a basis $\B_{(2,1)}=\{\Xi^{(2,1)},\Xi^{(2,1)}213,\Xi^{(2,1)}312,\Xi^{(2,1)}132\}$. By Lemmas \ref{L: basic Xi} and \ref{L: XiMod surj Lie}, there are surjections of $V$ onto $X^{(1,1,1)}_F\cong F$ and $\Lie_F((2,1))$ respectively.

Suppose first that $p=2$. Under this assumption, $(2,1)$ is not coprime to $2$ and $\omega_{(2,1)}=2\Xi^{(2,1)}-\Xi^{(1,1,1)}=\Xi^{(1,1,1)}$. Therefore we have $X^{(1,1,1)}_F\cong F\cong \Lie_F((2,1))$. On the other hand, using the idea as in the proof of Theorem \ref{T: basis Xiq}, we can compute the matrix representation of the element $(1,2)$ on $V$ with respect to the basis $\B_{(2,1)}$, that is \[[(1,2)]_{\B_{(2,1)}}=\begin{pmatrix}
  0&1&0&0\\ 1&0&0&0\\ -1&1&1&0\\ 0&1&1&-1
\end{pmatrix}.\] Since the matrix $[(1,2)]_{\B_{(2,1)}}-I_4$ has rank 2 and $\langle(1,2)\rangle$ is a Sylow $2$-subgroup of $\sym{3}$, the module $V$ is projective. Since $V$ is a submodule of the regular module \[F\sym{3}\cong Y^{(2,1)}\oplus Y^{(2,1)}\oplus Y^{(1,1,1)},\] it must be a direct sum of some of the Young modules in the decomposition where $Y^{(2,1)}\cong D^{(2,1)}$ and $Y^{(1,1,1)}\cong\begin{bmatrix}
  F\\ F
\end{bmatrix}$. Comparing the dimensions and using the fact $V$ surjects onto $F$, we must have \[V\cong Y^{(2,1)}\oplus Y^{(1,1,1)}\cong D^{(2,1)}\oplus \begin{bmatrix}
  F\\ F
\end{bmatrix}.\] We have seen that $\Lie_F((2,1))\cong F$. In fact, $\Lie_F((1,1,1))\cong F$ too. It is well-known that $\Lie_F((3))$ is projective (see also Corollary \ref{C: Complexity Lie}) and has dimension 2. An easy calculation shows that $\Lie_F((3))$ does not admit a trivial submodule and therefore it is isomorphic to $Y^{(2,1)}\cong D^{(2,1)}$. As such, the composition factors of $V$ must involve the simple module $\Lie_F((3))$ and not just $\Lie_F((2,1))$ and $\Lie_F((1,1,1))$. This example shows that, in some modular cases, $X^q_F$ may involve part, if not all, of $\Lie_F(\lambda)$ of some partition $\lambda\not\wref q$ (cf. Theorem \ref{T: Xi as higher Lie}).

Suppose now that $p\neq 2$. We first examine $\Lie_F((2,1))$. Since $(2,1)$ is coprime to $p$ and $p$-regular, by Theorem \ref{T: proj summand Xi}, we have \[\Xi^{(2,1)}e_{(2,1),F}F\sym{3}\cong \Lie_F((2,1))\] is a projective summand of $\Xi^{(2,1)}F\sym{3}$ of dimension 3. One could use Corollary \ref{C: proj higher Lie decomp} to get \[\Lie_F((2,1))\cong \left \{\begin{array}{ll} P((2,1))&p=3,\\ P((2,1))\oplus P((3))&p\geq 5.\end{array}\right .\] Alternatively, by Theorem \ref{T: higher lie mod}, \[\Lie_F((2,1))\cong \Ind_{\sym{1}\times\sym{2}}^{\sym{3}}(\Lie_F(1)\boxtimes \Lie_F(2))\cong \Ind_{\sym{1}\times\sym{2}}^{\sym{3}}(F\boxtimes \sgn(2))\] with Specht factors $S^{(2,1)}$ and $S^{(1,1,1)}$. Using the result of Brauer and Robinson \cite{Brau,Rob}; namely, the Nakayama Conjecture, these two factors lie in the same $p$-block if and only if $p=3$ (as $p\neq 2$ by our initial assumption). Therefore, \[\Lie_F((2,1))\cong \left \{\begin{array}{ll} \begin{bmatrix}
  D^{(2,1)}\\ F\\ D^{(2,1)}
\end{bmatrix}& p=3,\\ S^{(2,1)}\oplus S^{(1,1,1)}& p\neq 3,\end{array}\right .\cong \left \{\begin{array}{ll} Y^{(1,1,1)}&p=3,\\ Y^{(2,1)}\oplus Y^{(1,1,1)}&p\neq 3,\end{array}\right .\] where both $S^{(2,1)}$ and $S^{(1,1,1)}$ are simple in the case of $p\neq 3$ and all Young modules appearing above are projective. Since $V$ also  
surjects onto $X^{(1,1,1)}_F\cong F$ and $\mathrm{Hom}_{F\sym{3}}(\Lie_F((2,1)),F)=0$, we must have \[X^{(2,1)}_F\cong \Lie_F((2,1))\oplus \Lie_F((1,1,1)).\] 

To conclude our example, in terms of decomposition into Young modules and composition factors, we have
\begin{align*}
X^{(2,1)}_F&\cong \left \{\begin{array}{ll} Y^{(2,1)}\oplus Y^{(1^3)}&p=2,\\ Y^{(1^3)}\oplus Y^{(3)}&p=3,\\ Y^{(2,1)}\oplus Y^{(1^3)}\oplus Y^{(3)}&p\geq 5,\end{array}\right .\sim\left \{\begin{array}{ll} D^{(2,1)}+2D^{(3)}&p=2,\\ 2D^{(2,1)}+2D^{(3)}&p=3,\\ D^{(2,1)}+D^{(1^3)}+D^{(3)}&p\geq 5.\end{array}\right .\end{align*}
\end{eg}

As we have remarked right after Lemma \ref{L: dim Xie}, we give some computational data to further support the conjectures below.

\begin{eg}\label{Eg: dimXie} Let $p=2$. Using Magma \cite{Magma}, we obtain the following table.
\[\begin{array}{|c|c|c|c|c|} \hline \text{Partition $\lambda$}& \dim X^\lambda_F& \dim \Xi^\lambda e_{(5),F}F\sym{5}& \dim \Xi^\lambda e_{(4,1),F}F\sym{5}& \dim \Xi^\lambda e_{(3,2),F}F\sym{5}\\ \hline
( 5 )& 120& 24& 56& 40\\
\hline( 4, 1 )& 76& 0& 56& 20\\
\hline( 3, 2 )& 66& 0& 26& 40\\
\hline( 3, 1, 1 )& 31& 0& 11& 20\\
\hline( 2, 2, 1 )& 26& 0& 26& 0\\
\hline( 2, 1, 1, 1 )& 11& 0& 11& 0\\
\hline( 1, 1, 1, 1, 1 )& 1& 0& 1& 0\\ \hline\end{array}\] Observe that $\dim \Xi^\lambda e_{\mu,F}F\sym{5}$ is the number of permutations with cycle types weak refinements of $\lambda$ and $2$-equivalent to $\mu$.
\end{eg}

\begin{conj}\label{Con: dim} Let $\mu\in \pReg(n)$ and $q\in\C(n)$. The dimension of $\Xi^qe_{\mu,F}F\sym{n}$ is the number of permutations with cycle types both weak refinements of $q$ and $p$-equivalent to $\mu$.
\end{conj}

Since \[X^q_F=\sum_{\mu\in\pReg(n)}\Xi^qe_{\mu,F}F\sym{n},\] Conjecture \ref{Con: dim} implies the following conjecture.

\begin{conj}\label{Con: direct sum} Let $q\in\C(n)$. Then we have a direct sum decomposition of right $F\sym{n}$-modules \[X^q_F=\bigoplus_{\mu\in\pReg(n)}\Xi^qe_{\mu,F}F\sym{n}.\]
\end{conj}

There is another motivation for Conjecture \ref{Con: dim}. It implies the following conjecture which generalizes the short exact sequence obtained by Erdmann-Schocker (see Theorem \ref{T: SES}) in the case of $k=1$.

\begin{conj} Let $\ell\in\NN$ and $\xi=(p^{\ell-1},\ldots,p^{\ell-1})\vDash p^\ell$. We have a short exact sequence of $F\sym{p^\ell}$-modules \[\xymatrix{0\ar[r]&\mathrm{Lie}(p^\ell)\ar[r]^-\alpha&e_{(p^\ell),F}F\sym{p^\ell}\ar[r]^-\beta& \Xi^\xi e_{(p^\ell),F}F\sym{p^\ell}\ar[r]&0}\] where $\alpha(\omega_{p^\ell}\gamma)=e_{(p^\ell),F}\omega_{p^\ell}\gamma=\omega_{p^\ell}\gamma$ and $\beta(e_{(p^\ell),F}\gamma)=\Xi^\xi e_{(p^\ell),F}\gamma$ for any $\gamma\in F\sym{p^\ell}$.
\end{conj}

The fact $\beta\alpha=0$ follows from $\Xi^\xi\omega_{p^\ell}=0$ using Theorem \ref{T: BL results}(ii).

\newpage
\appendix

\section{Modular Idempotents when $p=2$ and $n=2,3,4,5,6$}\label{Appen A}

{\[\def\arraystretch{1.3}\begin{array}{|l|r|l|} \hline \text{$n$}&\text{$2$-regular}&\text{Modular Idempotents $e_{\lambda,F}$ in terms of $\Xi^q$'s}\\ &\text{partitions $\lambda$}& \\ \hline
2&(2)&\Xi^{(2)}\\ \hline
3&(3)&\Xi^{(3)}+\Xi^{(2,1)}+\Xi^{(1,1,1)}\\ \cline{2-3}
&(2,1)&\Xi^{(2,1)}+\Xi^{(1,1,1)}\\ \hline
4&(4)&\Xi^{(4)}+\Xi^{(3,1)}+\Xi^{(2,1,1)}+\Xi^{(1,1,1,1)}\\ \cline{2-3}
&(3,1)&\Xi^{(3,1)}+\Xi^{(2,1,1)}+\Xi^{(1,1,1,1)}\\ \hline
5&(5)&\Xi^{(5)}+\Xi^{(4,1)}+\Xi^{(3,2)}+\Xi^{(3,1,1)}+\Xi^{(2,2,1)}+\\
&&\Xi^{(2,1,2)}+\Xi^{(1,2,2)}+\Xi^{(1,2,1,1)}+\Xi^{(1,1,1,1,1)}\\ \cline{2-3}
&(4,1)&\Xi^{(4,1)}+\Xi^{(3,1,1)}+\Xi^{(2,2,1)}+\Xi^{(2,1,1,1)}+\Xi^{(2,1,2)}+\Xi^{(1,1,1,2)}+\\
&&\Xi^{(1,1,1,1,1)}\\ \cline{2-3}
&(3,2)&\Xi^{(3,2)}+\Xi^{(1,2,2)}+\Xi^{(2,1,1,1)}+\Xi^{(1,2,1,1)}+\Xi^{(1,1,1,2)}\\ \hline
6&(6)&\Xi^{( 6 )}+\Xi^{( 1, 1, 2, 1, 1 )}
        +\Xi^{( 4, 1, 1 )}
        +\Xi^{( 4, 2 )}
        +\Xi^{( 2, 2, 1, 1 )}
        +\Xi^{( 5, 1 )}+\\
&&
        \Xi^{( 3, 1, 2 )}+\Xi^{( 1, 1, 1, 2, 1 )}
        +\Xi^{( 2, 1, 2, 1 )}
        +\Xi^{( 1, 1, 1, 1, 1, 1 )}+\\
&&        \Xi^{( 2, 2, 2 )}
        +\Xi^{( 1, 2, 1, 2 )}\\ \cline{2-3}
&(5,1)& \Xi^{( 5, 1 )}+\Xi^{( 4, 1, 1 )}
        +\Xi^{( 2, 2, 1, 1 )}
        +\Xi^{( 1, 2, 1, 1, 1 )}
        +\Xi^{( 2, 1, 2, 1 )}+\\
&&       \Xi^{( 3, 1, 1, 1 )}+ \Xi^{( 1, 1, 1, 1, 1, 1 )}
        +\Xi^{( 1, 2, 2, 1 )}
        +\Xi^{( 3, 2, 1 )}\\ \cline{2-3}
&(4,2)&\Xi^{( 4, 2 )}+\Xi^{( 1, 1, 2, 1, 1 )}
        +\Xi^{( 3, 1, 2 )}
        +\Xi^{( 1, 1, 1, 1, 1, 1 )}
        +\Xi^{( 2, 2, 2 )}+\\
&&        \Xi^{( 1, 2, 1, 2 )}\\ \cline{2-3}
&(3,2,1)&\Xi^{( 3, 2, 1 )}+\Xi^{( 1, 2, 1, 1, 1 )}
        +\Xi^{( 1, 1, 1, 2, 1 )}
        +\Xi^{( 3, 1, 1, 1 )}+\\
&&       \Xi^{( 1, 1, 1, 1, 1, 1 )}
        +\Xi^{( 1, 2, 2, 1 )}\\ \hline
\end{array}\]}
\newpage

\section{Modular Idempotents when $p=3$ and $n=2,3,4,5,6$}\label{Appen B}

{\[\def\arraystretch{1.3}\begin{array}{|l|r|l|} \hline \text{$n$}&\text{$3$-regular}&\text{Modular Idempotents $e_{\lambda,F}$ in terms of $\Xi^q$'s}\\ &\text{partitions $\lambda$}& \\ \hline
2&(2)&\Xi^{(2)}+\Xi^{(1,1)}\\ \cline{2-3}
&(1,1)&2\Xi^{(1,1)}\\ \hline
3&(3)&\Xi^{(3)}+2\Xi^{(2,1)}+2\Xi^{(1,1,1)}\\ \cline{2-3}
&(2,1)&\Xi^{(2,1)}+\Xi^{(1,1,1)}\\ \hline
4&(4)&\Xi^{( 4 )}+ \Xi^{( 2, 2 )}+2 \Xi^{( 3, 1 )}+2 \Xi^{( 1, 1, 1, 1 )}+ \Xi^{( 1, 1, 2 )}\\ \cline{2-3}
&(3,1)&\Xi^{(3,1)}+2\Xi^{(2,1,1)}\\  \cline{2-3}
&(2,2)&2\Xi^{( 2, 2)}+
        2\Xi^{( 1, 1, 1, 1)}+
        2\Xi^{( 2, 1, 1)}+
        2\Xi^{( 1, 1, 2)}\\ \cline{2-3}
&(2,1,1)&2\Xi^{(2,1,1)}+2\Xi^{(1,1,1,1)}\\ \hline
5&(5)&\Xi^{( 5 )}+
        2\Xi^{( 4, 1 )}+
        \Xi^{( 2, 1, 2 )}+
        \Xi^{( 2, 1, 1, 1 )}+
        \Xi^{( 3, 1, 1 )}+
        2\Xi^{( 3, 2 )}+\\
        &&
        2\Xi^{( 1, 1, 1, 1, 1 )}+
        \Xi^{( 1, 1, 1, 2 )}+
        2\Xi^{( 1, 1, 2, 1 )}+
        \Xi^{( 1, 2, 1, 1 )}\\ \cline{2-3}
&(4,1)& \Xi^{( 4, 1 )}+
        2\Xi^{( 2, 1, 1, 1 )}+
        2\Xi^{( 3, 1, 1 )}+
        2\Xi^{( 1, 1, 1, 1, 1 )}+
        \Xi^{( 2, 2, 1 )}+
        2\Xi^{( 1, 2, 1, 1 )}\\ \cline{2-3}
&(3,2)&\Xi^{( 3, 2 )}+ 2\Xi^{( 2, 1, 2 )}+
        2\Xi^{( 2, 1, 1, 1 )}+
        \Xi^{( 3, 1, 1 )}+
        2\Xi^{( 1, 1, 1, 1, 1 )}+ 2\Xi^{( 1, 1, 1, 2 )}+ \\
        &&
        2\Xi^{( 1, 1, 2, 1 )}+
        \Xi^{( 1, 2, 1, 1 )}\\ \cline{2-3}
&(3,1,1)& 2\Xi^{( 3, 1, 1 )}+2\Xi^{( 2, 1, 1, 1 )}+
        \Xi^{( 1, 1, 1, 1, 1 )}+
        2\Xi^{( 1, 2, 1, 1 )}\\ \cline{2-3}
&(2,2,1)&2\Xi^{( 2, 2, 1 )}+2\Xi^{( 2, 1, 1, 1 )}+
        2\Xi^{( 1, 1, 1, 1, 1 )}+
        2\Xi^{( 1, 1, 2, 1 )}\\ \hline
6&(6)&\Xi^{( 6 )}+
        \Xi^{( 3, 3 )}+
        \Xi^{( 1, 1, 2, 1, 1 )}+
        \Xi^{( 4, 1, 1 )}+
        2\Xi^{( 1, 1, 1, 1, 2 )}+
        2\Xi^{( 1, 1, 2, 2 )}+ \\ &&
        2\Xi^{( 4, 2 )}+
        2\Xi^{( 2, 2, 1, 1 )}+
        2\Xi^{( 5, 1 )}+
        \Xi^{( 3, 1, 2 )}+
        2\Xi^{( 1, 1, 1, 2, 1 )}+
        \Xi^{( 2, 1, 2, 1 )}+ \\ &&
        2\Xi^{( 3, 1, 1, 1 )}+
        2\Xi^{( 2, 1, 3 )}+
        2\Xi^{( 2, 2, 2 )}+
        \Xi^{( 1, 1, 3, 1 )}+
        2\Xi^{( 1, 3, 1, 1 )}+ \\ &&
        2\Xi^{( 3, 2, 1 )}+
        \Xi^{( 2, 1, 1, 1, 1 )}\\ \cline{2-3}
&(5,1)&\Xi^{( 5, 1 )}+2\Xi^{( 4, 1, 1 )}+
        2\Xi^{( 1, 2, 1, 1, 1 )}+
        2\Xi^{( 1, 1, 1, 2, 1 )}+
        \Xi^{( 2, 1, 2, 1 )}+ \\&&
        2\Xi^{( 1, 1, 1, 1, 1, 1 )}+
        2\Xi^{( 1, 1, 3, 1 )}+
        2\Xi^{( 1, 3, 1, 1 )}+
        2\Xi^{( 3, 2, 1 )}+
        \Xi^{( 2, 1, 1, 1, 1 )}\\ \cline{2-3}
&(4,2)& \Xi^{( 4, 2 )}+ \Xi^{( 1, 1, 2, 1, 1 )}+
        \Xi^{( 4, 1, 1 )}+
        2\Xi^{( 1, 1, 1, 1, 2 )}+
        \Xi^{( 1, 1, 2, 2 )}+ \\&&
        \Xi^{( 2, 2, 1, 1 )}+
        2\Xi^{( 3, 1, 2 )}+
        2\Xi^{( 3, 1, 1, 1 )}+
        2\Xi^{( 1, 1, 1, 1, 1, 1 )}+
        \Xi^{( 2, 2, 2 )}\\ \cline{2-3}
&(3,3)&2\Xi^{( 3, 3 )}+
        2\Xi^{( 1, 1, 1, 1, 2 )}+
        \Xi^{( 1, 1, 1, 2, 1 )}+
        2\Xi^{( 2, 1, 2, 1 )}+
        \Xi^{( 2, 1, 3 )}+ \\ &&
        2\Xi^{( 1, 3, 1, 1 )}+
        \Xi^{( 3, 2, 1 )}+
        \Xi^{( 2, 1, 1, 1, 1 )}\\ \cline{2-3}
&(4,1,1)&2\Xi^{( 4, 1, 1 )}+
        2\Xi^{( 2, 2, 1, 1 )}+
        \Xi^{( 1, 2, 1, 1, 1 )}+
        \Xi^{( 3, 1, 1, 1 )}+
        \Xi^{( 1, 1, 1, 1, 1, 1 )}+ \\ &&
        \Xi^{( 2, 1, 1, 1, 1 )}\\  \cline{2-3}
&(3,2,1)& \Xi^{( 3, 2, 1 )}+\Xi^{( 1, 1, 1, 2, 1 )}+
        2\Xi^{( 2, 1, 2, 1 )}+
        \Xi^{( 3, 1, 1, 1 )}+
        \Xi^{( 2, 1, 1, 1, 1 )}\\  \cline{2-3}
&(2,2,1,1)&\Xi^{( 2, 2, 1, 1 )}+\Xi^{( 1, 1, 2, 1, 1 )}+
        \Xi^{( 1, 1, 1, 1, 1, 1 )}+
        \Xi^{( 2, 1, 1, 1, 1 )}\\ \hline
\end{array}\]}

\section{$\dim_F\Lie_F(q)$ when $p=0,2,3$ and $n=2,3,4,5,6$}\label{Appen C}

\[\begin{array}{|c|c|c|c|c|}
  \hline n&\text{Partition $q$}&p=0&p=2&p=3\\ \hline
  2&(2)&1&1&1\\ \cline{2-5}
  &(1,1)&1&1&1\\ \hline
  3&(3)&2&2&2\\ \cline{2-5}
  &\textcolor{black}{(2,1)}&3&1&3\\ \cline{2-5}
  &(1,1,1)&1&1&1\\ \hline
  4&(4)&6&6&6\\ \cline{2-5}
  &\textcolor{black}{(3,1)}&8&8&7\\ \cline{2-5}
  &\textcolor{black}{(2,2)}&3&1&3\\ \cline{2-5}
  &\textcolor{black}{(2,1,1)}&6&1&6\\ \cline{2-5}
  &(1,1,1,1)&1&1&1\\ \hline
  5&(5)&24&24&24\\  \cline{2-5}
  &\textcolor{black}{(4,1)}&30&21&30\\ \cline{2-5}
  &(3,2)&20&20&20\\ \cline{2-5}
  &\textcolor{black}{(3,1,1)}&20&20&11\\ \cline{2-5}
  &\textcolor{black}{(2,2,1)}&15&1&15\\ \cline{2-5}
  &\textcolor{black}{(2,1,1,1)}&10&1&10\\ \cline{2-5}
  &(1,1,1,1,1)&1&1&1\\ \hline
  6&(6)&120&120&120\\ \cline{2-5}
  &(5,1)&144&144&144\\ \cline{2-5}
  &\textcolor{black}{(4,2)}&90&41&90\\ \cline{2-5}
  &\textcolor{black}{(4,1,1)}&90&41&90\\ \cline{2-5}
  &(3,3)&40&40&40\\ \cline{2-5}
  &\textcolor{black}{(3,2,1)}&120&40&60\\ \cline{2-5}
  &\textcolor{black}{(3,1,1,1)}&40&40&16\\ \cline{2-5}
  &\textcolor{black}{(2,2,2)}&15&1&15\\ \cline{2-5}
  &\textcolor{black}{(2,2,1,1)}&45&1&45\\ \cline{2-5}
  &\textcolor{black}{(2,1,1,1,1)}&15&1&15\\ \cline{2-5}
  &(1,1,1,1,1,1)&1&1&1\\ \hline
\end{array}\]

\printindex

\end{document}